\documentclass[a4paper,11pt,final]{amsart}
\usepackage[hmargin=2.0cm,vmargin=2.0cm]{geometry}
\usepackage{amscd}
\usepackage{ascmac}
\usepackage{amsfonts}
\usepackage{amsmath}
\usepackage{amssymb}
\usepackage{amsthm}
\usepackage{bm}     % Newly added
\usepackage{caption}
\usepackage{colonequals}
\usepackage{delarray}
\usepackage{enumerate}
\usepackage{epsfig}
\usepackage{fancyhdr}
\usepackage{float}
\usepackage{graphicx}
\usepackage{latexsym}
\usepackage{lscape}
\usepackage{mathrsfs}
\usepackage{mathtools}
\usepackage{multicol}
\usepackage{pst-grad} % For gradients
\usepackage{pst-plot} % For axes
\usepackage{pdfpages}
\usepackage[all]{xy}

\newtheorem{theorem}{Theorem}[section] 
\newtheorem{proposition}[theorem]{Proposition}
\newtheorem{corollary}[theorem]{Corollary}
\newtheorem{definition}[theorem]{Definition}
\newtheorem{lemma}[theorem]{Lemma}
\newtheorem{remark}{Remark}

\numberwithin{equation}{section}
\newcommand{\z}{\mathfrak z}

\DeclareMathOperator{\GL}{\rm{GL}}

\begin{document}
\title[Bargmann type transformation]%on the punctured cotangent bundle of Cayley projective plane]
{Calabi-Yau structure and Bargmann type transformation on the Cayley projective plane}

\author[Kurando Baba]{Kurando Baba}
\author[Kenro Furutani]{Kenro Furutani$^\diamondsuit$}
\address{K.~Baba: Department of Mathematics, Faculty of Science and Technology, 
Tokyo University of Science, 2641 Yamazaki, Noda, Chiba, 278-8510, Japan.}
\email{kurando.baba@rs.tus.ac.jp}
\address{K.~Furutani: Osaka City University Advanced Mathematical Institute (=OCAMI)
3-3-138 Sugimoto, Sumiyoshi-ku, Osaka, 558-8585, Japan.}
%\email{fkenro@sci.osaka-cu.ac.jp}
\email{furutani$\_$kenro@ma.noda.tus.ac.jp}

\date{\today}
\thanks{$^\diamondsuit$The second author was partially supported by
  JSPS fund 17K05284.
Also this work was partially supported by Osaka City University
Advanced Mathematical Institute (MEXT Joint Usage/Research Center on
Mathematics and Theoretical Physics JPMXP0619217849).}

%{32Q25}{17C36}


\begin{abstract}
Our purpose is to show the existence
of a Calabi-Yau structure on the punctured cotangent bundle $T^{*}_{0}(P^2\mathbb{O})$
of the Cayley projective plane $P^{2}\mathbb{O}$ and to construct a Bargmann
type transformation from 
a space of holomorphic functions on $T^{*}_{0}(P^2\mathbb{O})$ to
$L_{2}$-space on $P^{2}\mathbb{O}$. 
The space of holomorphic functions
corresponds to the Fock space in the case of the original Bargmann transformation. 
A K\"ahler structure on $T^{*}_{0}(P^{2}\mathbb{O})$ was shown 
by identifying it with a quadrics in the complex space 
$\mathbb{C}^{27}\backslash\{0\}$
and the natural symplectic form of the cotangent bundle
$T^{*}_{0}(P^2\mathbb{O})$ is expressed as a K\"ahler form.
Our method to construct the transformation is the pairing
of polarizations, one is the natural Lagrangian foliation given by the
projection map 
${\bf q}:T^{*}_{0}(P^2\mathbb{O})\longrightarrow P^{2}\mathbb{O}$ 
and the polarization given
by the K\"ahler structure. 

The transformation gives a quantization of the geodesic flow in terms
of one parameter group of elliptic Fourier integral operators whose
canonical relations are defined by the graph of the geodesic flow action at
each time. It turn out that for the Cayley projective
plane the results are not same with other cases of
the original Bargmann transformation for Euclidean space, spheres 
and other projective spaces. 
\end{abstract}

\maketitle

\section{Introduction}

The fundamental and historical problem in the quantization theory
will be how to assign a function on a phase space to
an operator acting on the space of quantum states and 
the assignment satisfies some algebraic condition, like
a Lie algebra homomorphism. The phase space appearing in the theory has a
structure, a symplectic structure. There are many theory relating with this
problem. One method is the
theory of deformation quantization. Also there is the opposite theory,
an assignment of operators to functions,
from an operator to a function. In the (pseudo)differential 
operator theory and Fourier integral operator 
theory, the basic assignment of
operators to their principal symbol (and sub-principal symbol) 
is a fundamental isomorphism between the spaces of operators and
functions on the phase space modulo lower order classes.
 
The famous transformation, called Bargmann transformation was
introduced in \cite{Ba} and gives one aspect of the quantization of the unitary  
representation. The method to 
construct such a transformation 
is given by the pairing of two
polarizations, real polarization and complex polarization,
on $\mathbb{C}^{n}$ interpreted as
$\mathbb{C}^{n}\cong T^{*}(\mathbb{R}^{n})\cong\mathbb{R}^{n}\times\mathbb{R}^{n}$, complex space and fiber
space by Lagrangian fibers $\pi:\mathbb{C}^{n}$ $\to$ $\mathbb{R}^{n}$. 
Under precise treatments of this method it was given a similar operator for
the case of the sphere in \cite{Ra2}, in \cite{FY} for the complex
projective space and for the quaternion projective spaces in \cite{Fu1}.

Among the projective spaces the Cayley projective plane
$P^2\mathbb{O}$ is the exceptional one
and our purpose in this paper is to show that we can also construct such an operator 
for this manifold in the same method. This case  will be
one of the non-trivial examples to which we can apply
this method, "{\it pairing of two polarizations}" (\cite{Ra2},
\cite{Ii1}, \cite{Ii2}, \cite{Fu1}, \cite{FY}). 

In the paper \cite{Fu2} a K\"ahler structure on its
punctured cotangent bundle $T^{*}_{0}(P^2\mathbb{O})$ was constructed by embedding 
it into the complex space $\mathbb{C}^{27}\backslash\{0\}$ as an
intersection of null sets of several quadric polynomials, which gives the realization of 
the natural symplectic form as a K\"ahler form. Here we show the
holomorphic triviality of the canonical line bundle of this complex manifold by
giving a nowhere vanishing global holomorphic $16$-form explicitly.

There are several study of the existence of K\"ahler structure on the
(punctured) cotangent bundle of a certain class of manifolds, like
\cite{Ra1},
\cite{Sz1}, \cite{Sz2}, \cite{Koi}, \cite{Li}, \cite{FT}, also see
\cite{Be}, \cite{So} 
in relation with a special property of the geodesic flow, $SC_{\ell}$-manifolds. 

The classical Bargmann transformation gives a correspondence between
monomials on $\mathbb{C}^{n}$ and Hermite functions on
$\mathbb{R}^{n}$, which are the eigenfunctions of the harmonic
oscillator and this facts were applied to various problems, especially to
T\"oplitz operator theory (there are so many, but here I just cite one book
\cite{BS}). Also there are many precise treatments and modifications of this
transformation (for examples,  \cite{Ii1}, and recent works in \cite{Ch1}, \cite{Ch2} ). 

For our case we show the restrictions of monomials defined on $\mathbb{C}^{27}\backslash\{0\}$
to the embedded punctured cotangent bundle $T^{*}_{0}(P^{2}\mathbb{O})$
are mapped to eigenfunctions of the Laplacian on $P^{2}\mathbb{O}$.

This paper is organized as follows.
In $\S 2$, we explain a realization of quaternion  and octanion number
fields, $\mathbb{H}$ and $\mathbb{O}$, 
in a complex matrix algebra. Multiplication law in the octanion
is interpreted in the two $2\times 2$-complex matrix algebra
$\mathbb{C}(2)\times \mathbb{C}(2)$.

In $\S 3$, we introduce the Jordan algebra $\mathcal{J}(3)$ of
$3\times 3$ octanion matrices. Cayley projective plane
$P^{2}\mathbb{O}$ is realized in this
Jordan algebra. Following an earlier result in \cite{Fu2} we explain
the embedding of the punctured cotangent bundle $T^{*}_{0}(P^2\mathbb{O})$
of the Cayley projective plane
into the complexified Jordan algebra 
$\mathbb{C}\otimes_{\mathbb{R}}\hspace{-0.1cm}\mathcal{J}(3)=:\mathcal{J}(3)^{\mathbb{C}}$
of $3\times 3$ complexified octanion matrices:
\[
\tau_{\mathbb{O}}:T^{*}(P^2\mathbb{O})\longrightarrow \mathcal{J}(3)^{\mathbb{C}}.
\]
We denote the image $\tau_{\mathbb{O}}(T^{*}(P^{2}\mathbb{O}))=\mathbb{X}_{\mathbb{O}}$.
Also we state that the natural symplectic form
$\omega^{P^{2}\mathbb{O}}$ is a K\"ahler form.

In $\S 4$, using the defining equations of the punctured cotangent
bundle of the Cayley projective plane embedded in the complex Jordan algebra $\mathcal{J}(3)^{\mathbb{C}}$,
we give an open covering by complex coordinates neighborhoods and show
by an elementary way that 
the canonical line bundle of the complex structure is holomorphically
trivial by explicitly constructing a nowhere vanishing holomorphic
global section (we put it $\Omega_{\mathbb{O}}$), that is, a $16$-degree holomorphic differential form
which coincides with the restriction of a smooth $16$-degree
differential form on the whole complexified Jordan algebra $\mathcal{J}(3)^{\mathbb{C}}$.

In $\S 5$, we resume a basic fact on symplectic manifolds with
integral symplectic form  and
a method of the geometric quantization. Here we
consider two types of typical polarizations (real and positive
complex). Then we apply the method to our case (= $T^{*}_{0}(P^2\mathbb{O})$)
and give a Bargmann type transformation in the form of a fiber
integration on the punctured cotangent bundle
$T^{*}_{0}(P^2\mathbb{O})$ to the base space $P^{2}\mathbb{O}$.

In $\S 6$, first we show the nowhere vanishing holomorphic global
section $\Omega_{\mathbb{O}}$ constructed in $\S 4$ is
$F_{4}$-invariant. Incidentally, we
determine the product $\Omega_{\mathbb{O}}\wedge \overline{\Omega_{\mathbb{O}}}$
in terms of the Liouville volume form 
$\displaystyle{dV_{T^{*}(P^2\mathbb{O})}:=\frac{1}{16\,!}\left(\omega^{P^{2}\mathbb{O}}\right)^{16}}$ of the
cotangent bundle $T^{*}(P^{2}\mathbb{O})$.

Also we introduce a class of subspaces consisting of holomorphic functions
on $\mathbb{X}_{\mathbb{O}}$ satisfying some $L_{2}$ conditions. These
will correspond to the Fock space in the Euclidean case.

In $\S 7$, we determine the exterior product of the Riemann volume form
pull-backed to the cotangent bundle $T^{*}_{0}(P^2\mathbb{O})$
and the nowhere vanishing global holomorphic section $\Omega_{\mathbb{O}}$ in terms
of the Liouville volume form.
% $dV_{P^{2}\mathbb{O}}$. 
For this purpose we fix a local coordinates
at a point in $P^2\mathbb{O}$ which is also used in the section $\S 9$.

In $\S 8$, we discuss invariant polynomials and a similar feature to harmonic polynomials
with respect to the natural representation of the group $F_{4}$ to the Jordan algebra $\mathcal{J}(3)$
and its extension to the
polynomial algebra. 
Then,  based on a general theorem in \cite{He} (also \cite{HL} and
\cite{Ko}) we state the eigenfunction decomposition  of $L_{2}$ space of $P^{2}\mathbb{O}$.

\bigskip
In $\S 9$, based on the data obtained until $\S 8$ 
we discuss our Bargmann type transformation is a bounded
operator, or isomorphism or unbounded 
according to the Hilbert space structures in the Fock-like space.
Some cases says 
there are quantum states in $L_{2}(P^{2}\mathbb{O})$ which  are
approximated by classical phenomena, but can not be
observed directly by classical mechanical way.  

Finally in $\S 10$ we mention that our Fock-like spaces have the
reproducing kernel and a relation with the geodesic flow action.

\section{Representations of quaternion and octanion algebras by  complex matrix algebras}

First, we fix a representation of quaternion numbers 
$h=h_{0}{\bf 1}+h_{1}{\bf i}+h_{2}{\bf j}+h_{3}{\bf k}$
~({\bf 1}, {\bf i}, {\bf j}, {\bf k} are standard basis of the quaternion number field $\mathbb{H}$~
and $h_{i}\in\mathbb{R}$) 
as a $2\times 2$
complex matrix in the following way:
\begin{equation}\label{quaternion to complex matrix}
\rho_{\mathbb{H}}:\mathbb{H}\ni h\longmapsto
\begin{pmatrix}
h_{0}+h_{1}{\sqrt{-1}}&\,h_{2}+h_{3}\sqrt{-1}\\-h_{2}+h_{3}\sqrt{-1}&\,h_{0}-h_{1}{\sqrt{-1}}\end{pmatrix}
=\begin{pmatrix}\lambda&\,\mu\\-\overline{\mu}&\,\overline{\lambda}\end{pmatrix}\in\mathbb{C}(2),
\end{equation}
where we understand that quaternions 
$h_{0}{\bf 1}+h_{1}{\bf i}~\text{and}~ h_{2}+h_{3}{\bf i}\in\mathbb{H}$
are complex numbers $\lambda=h_{0}+\sqrt{-1}h_{1}~\text{and}~\mu
=h_{2}+h_{3}\sqrt{-1}\in\mathbb{C}$ respectively.
Hence by this representation the complexification
$\mathbb{C}\otimes_{\mathbb{R}}\mathbb{H}$
is isomorphic to the ``algebra'' of the whole  $2\times 2$ complex matrix
algebra $\mathbb{C}(2)$
(we put $z_{i}=x_{i}+\sqrt{-1}y_{i}\in\mathbb{C}$):
\begin{align}
&\mathbb{C}\otimes_{\mathbb{R}}\hspace{-0.1cm}\mathbb{H}\ni h=z_{0}{\bf 1}+z_{1}{\bf i}+z_{2}{\bf
j}+z_{3}{\bf k}
\longmapsto 
\begin{pmatrix}z_{0}+\sqrt{-1}z_{1}&z_{2}+\sqrt{-1}z_{3}\\
\quad &\quad \\
-z_{2}+\sqrt{-1}z_{3}&z_{0}-\sqrt{-1}z_{1}\end{pmatrix}\in\mathbb{C}(2).\label{matrix representation}
\end{align}
We denote this map also by $\rho_{\mathbb{H}}$
and the inverse map is 
\begin{equation}\label{matrix to vector}
{\rho_{\mathbb{H}}}^{-1}:\mathbb{C}(2)\ni 
A=\begin{pmatrix}z_{1}&z_{2}\\ z_{3}&z_{4}\end{pmatrix}
\longmapsto
{\rho_{\mathbb{H}}}^{-1}(A)=\frac{z_{1}+z_{4}}{2}{\bf 1}+\frac{z_{1}-z_{4}}{2\sqrt{-1}}{\bf i}+\frac{z_{2}-z_{3}}{2}{\bf j}+\frac{z_{2}+z_{3}}{2\sqrt{-1}}{\bf k}.
\end{equation}
For $h=h_{0}{\bf 1}+h_{1}{\bf i}+h_{2}{\bf j}+h_{3}{\bf k}\in\mathbb{H}$
(or $\in\mathbb{C}\otimes_{\mathbb{R}}\hspace{-0.08cm}\mathbb{H}$), we denote its conjugation by
$\theta(h)=h_{0}{\bf 1}-h_{1}{\bf i}-h_{2}{\bf j}-h_{3}{\bf k}$, then
for $\rho_{\mathbb{H}}(h)=\begin{pmatrix}w_{1}&w_{2}\\w_{3}&w_{4}\end{pmatrix}$, 
$\rho_{\mathbb{H}}(\theta(h))
=\begin{pmatrix}w_{4}&-w_{2}\\-w_{3}&w_{1}\end{pmatrix}$ and 
the product
$\rho_{\mathbb{H}}(\theta(h))\rho_{\mathbb{H}}(h)
=\rho_{\mathbb{H}}(h)\rho_{\mathbb{H}}(\theta(h))=
(w_{1}w_{4}-w_{2}w_{3})\cdot\text{Id}=\det\, \rho_{\mathbb{H}}(h)\cdot
\text{Id}$, where $\text{Id}$
is $2\times 2$ identity matrix.

Let $\{{\bf e}_{i}\}_{i=0}^{7}$ be the standard basis of the octanion
number field $\mathbb{O}$ such that 
${\bf e}_{0}$ is the basis of the center.
We identify ${\bf e}_{0}={\bf 1}$, ${\bf e}_{1}={\bf i}$, ${\bf e}_{2}={\bf j}$ and
${\bf e}_{3}={\bf k}$ with the basis ${\{\bf 1,i,j,k\}}$ of the quaternion number field.
By the multiplication law
${\bf e}_{i}{\bf e}_{4}={\bf e}_{i+4}$ ($i=0,1,2,3$) we 
express  an (complexified) octanion number $x=\sum x_{i}{\bf e}_{i}$ as
the sum of two quaternion numbers:
\[
x=\sum_{i=0}^3 x_{i}{\bf e}_{i}+\sum_{i=0}^{3}~x_{i+4}{\bf e}_{i}\cdot
{\bf e}_{4}=a+b\cdot{\bf e}_{4}\in\mathbb{H}\oplus \mathbb{H}{\bf e}_{4}~\text{or} ~\in 
\mathbb{C}\otimes_{\mathbb{R}}\hspace{-0.08cm}\mathbb{H}\oplus
\mathbb{C}\otimes_{\mathbb{R}}\hspace{-0.08cm}\mathbb{H}{\bf e}_{4}.
\]
The complexification $\mathbb{C}\otimes_{\mathbb{R}}\hspace{-0.08cm}\mathbb{O}$ is identified as
\[
\mathbb{C}\otimes_{\mathbb{R}}\hspace{-0.08cm}\mathbb{O}\cong \mathbb{C}(2)\oplus\mathbb{C}(2){\bf e}_{4}
\]
through the map $\rho_{\mathbb{H}}\oplus\rho_{\mathbb{H}}=:\rho_{\mathbb{O}}$. 
%So, for 
%$h=a+b\cdot{\bf e}_{4}\in\mathbb{C}\otimes_{\mathbb{R}}\hspace{-0.08cm}\mathbb{O}$ we denote the matrix
%\[
%\rho_{\mathbb{H}}(a)+\rho_{\mathbb{H}}(b){\bf e}_{4}:=\rho_{\mathbb{O}}(h).
%\]

We define the conjugation operation in $\mathbb{O}$ 
(and also in $\mathbb{C}\otimes_{\mathbb{R}}\mathbb{O}$)
with the same notation $\theta$ for the quaternion case as
%The conjugation in $\mathbb{O}$ is defined by
\[
\theta:h=\sum h_{i}{\bf e}_{i}\longmapsto \theta(h)=h_{0}{\bf 1}-\sum_{i=1}^{7}
h_{i}{\bf e}_{i}.
\]
\allowdisplaybreaks{
The conjugation $\theta$ 
is interpreted in the matrix representation through the representation
$\rho_{\mathbb{O}}$ as
\begin{align}
&\theta:\mathbb{C}(2)\oplus\mathbb{C}(2){\bf e}_{4}\ni Z+W{\bf e}_{4}
=\begin{pmatrix}z_{1}&z_{2}\\z_{3}&z_{4}\end{pmatrix}+\begin{pmatrix}w_{1}&w_{2}\\w_{3}&w_{4}\end{pmatrix}{\bf
e}_{4}\label{conjugation in matrix form}\\
&\longmapsto
\theta(Z+W{\bf e}_{4})=\theta(Z)-W{\bf e}_{4}=\begin{pmatrix}z_{4}&-z_{2}\\-z_{3}&z_{1}\end{pmatrix}
-\begin{pmatrix}w_{1}&w_{2}\\w_{3}&w_{4}\end{pmatrix}{\bf e}_{4}.\notag
\end{align}
}
\begin{remark}\label{on multiplication}
The multiplication law of the octanions in the matrix form is 
given in \eqref{condition for off diagonal}.
\end{remark}

%\begin{remark}\label{coefficients rep of octanions}
%Hereafter, we will indicate the coefficients of the basis ${\bf e}_{i}$
%for an octanion as $z=\sum\limits_{i=0}^{7}\,\{z\}_{i}{\bf e}_{i}\in
%\mathbb{O}$ {\em(}or $z\in \mathbb{C}\otimes_{\mathbb{R}}\hspace{-0.08cm}\mathbb{O}${\em)},
 %to avoid the duplication from the matrix expression by $\rho_{\mathbb{O}}$ 
 %as above \eqref{conjugation in matrix form} untill $\S 8$.  
%\end{remark}

\begin{remark}\label{remark on conjugation operation}
We use the conjugation $\overline{z}=x-\sqrt{-1}y$ only for the complex number
$z=x+\sqrt{-1}y$ and do not use the operation $\theta$
for the conjugate of complex numbers to avoid confusion. So, for a complex
octanion number
$z=\sum\,\{z\}_{i}{\bf e}_{i}, ~\{z\}_{i}\in\mathbb{C}$, we mean
$\overline{z}=\sum \,\overline{\{z\}_{i}}{\bf e}_{i}$ and it holds  
$\theta(\overline{z})=\overline{\theta(z)}$. Also for an octanion
matrix $A=\Big(\,\,z_{ij}\,\,\Big)$ we mean
$\overline{A}:=\Big(\,\,\overline{z_{ij}}\,\,\Big)$
and $\theta(A):=\Big(\,\,\theta(z_{ij})\,\,\Big)$.
\end{remark}

\section{Cayley projective plane and its punctured cotangent bundle }

In this section, we refer \cite{SV}, \cite{M} and \cite{Yo} 
for all the necessary facts on the exceptional group $F_{4}$ and 
the Cayley projective plane. 

Let $\mathcal{J}(3)$ be a subspace of the $3\times 3$ octanion matrices:
\[
\mathcal{J}(3)=\left\{
\begin{pmatrix}t_{1}&z&\theta(y)\\
\theta(z)&t_{2}&x\\
y&\theta(x)&t_{3}
\end{pmatrix}~\Bigg|~x,y,z\in \mathbb{O}, t_{i}\in\mathbb{R}
\right\}.
\]
We introduce a product in $\mathcal{J}(3)$, called a ``Jordan product'',
by
\[
\mathcal{J}(3)\times \mathcal{J}(3) \ni (A,B)\longmapsto A\circ B:=\frac{AB+BA}{2}\in\mathcal{J}(3).
\]
It is called an exceptional Jordan algebra and of $27$-dimensional
over $\mathbb{R}$.
Then the group of $\mathbb{R}$-linear algebra automorphisms 
is the exceptional Lie group $F_{4}$:
\begin{equation}\label{F4}
F_{4}:=\left\{~g\in GL(\mathcal{J}(3))\cong GL(27,\mathbb{R})~|~g(A\circ
B)=g(A)\circ g(B),~g(Id)=Id,\,~A,B\in\mathcal{J}(3)\right\}.
\end{equation}
There are various characterizations for the group $F_{4}$ (see for
examples, \cite{Yo}, \cite{SV}).

The complexification $\mathbb{C}\otimes_{\mathbb{R}}\hspace{-0.03cm}\mathcal{J}(3)=:\mathcal{J}(3)^{\mathbb{C}}$ consists of
$3\times 3$ matrices with components of the complexified 
octanions of the form:
\[
\mathcal{J}(3)^{\mathbb{C}}
=\left\{
\begin{pmatrix}\xi_{1}&z&\theta(y)\\
\theta(z)&\xi_{2}&x\\
y&\theta(x)&\xi_{3}
\end{pmatrix}~\Bigg|~x,y,z~\in ~\mathbb{C}\otimes_{\mathbb{R}}\hspace{-0.08cm}\mathbb{O}, ~\xi_{i}\in\mathbb{C}
\right\}
\]
and is an exceptional Jordan algebra over $\mathbb{C}$ of
the complex dimension $27$. The complex linear automorphisms 
$\alpha:\mathcal{J}(3)^{\mathbb{C}}\longrightarrow 
\mathcal{J}(3)^{\mathbb{C}}$ 
satisfying the conditions
\[
\alpha(A\circ B)=\alpha(A)\circ \alpha(B), ~\alpha(Id)=Id
\]
is the complex Lie group ${F_{4}}^{\mathbb{C}}$. 
We may regard $F_{4}\subset {F_{4}}^{\mathbb{C}}$ in a natural way.

%By definition, 
%the algebra automorphisms of the octanion number field $\mathbb{O}$ is 
%the exceptional Lie group $G_{2}$, that is, $g\in G_{2}$, if it
%satisfyies $g(xy)=g(x)g(y), x,y\in\mathbb{O}$ and $g({\bf 1})={\bf 1}$.
%For $g\in G_{2}$, the
%transformation $\tilde{g}$ on $\mathcal{J}(3)$ defined by
%\begin{equation}\label{G2 to F4}
%\tilde{g}:\mathcal{J}(3)\ni
%A=\begin{pmatrix}\xi_{1}&z&\theta(y)\\\theta(z)&\xi_{2}&x\\y&\theta(x)&\xi_{3}\end{pmatrix}
%\longmapsto
%\tilde{g}(A)=\begin{pmatrix}\xi_{1}&g(z)&\theta(g(y))\\\theta(g(z))&\xi_{2}&g(x)\\g(y)
%&\theta(g(x))&\xi_{3}\end{pmatrix}
%\end{equation}
%is an algebra automorphism of $\mathcal{J}(3)$, $\tilde{g}\in F_{4}$.
 %In this sense we regard $G_{2}\subset F_{4}$.
\begin{definition}\label{Cayley projective plane}
The Cayley projective plane $P^2\mathbb{O}$ is defined as
\[
P^2\mathbb{O}=\left\{~X\in\mathcal{J}(3)~|~X^2=X, ~{\text{\em tr}\,(X)}=\xi_{1}+\xi_{2}+\xi_{3}=1~\right\}.
\]
\end{definition}
It is known that the group 
$F_{4}$ acts on $P^2\mathbb{O}$ in two point homogeneous way. 

Let $X_{1}=\begin{pmatrix}1&~0&~0\\0&~0&~0\\0&~0&~0\end{pmatrix}\in P^{2}\mathbb{O}$, 
then it is known that the stationary subgroup of the
point $X_{1}$ in $F_{4}$ is isomorphic to $Spin(9)$ and 
$F_{4}\ni g\longmapsto g\cdot X_{1}$ gives an isomorphism:
\begin{equation}\label{isomorphism between P2O and F4/Spin(9)}
F_{4}/Spin(9)\cong P^2\mathbb{O}.
\end{equation}
For $X=\begin{pmatrix}
\xi_{1}&x_{3}&\theta(x_{2})\\
\theta(x_{3})&\xi_{2}&x_{1}\\
x_{2}&\theta(x_{1})&\xi_{3}
\end{pmatrix}$,
$Y=\begin{pmatrix}
\eta_{1}&y_{3}&\theta(y_{2})\\
\theta(y_{3})&\eta_{2}&y_{1}\\
y_{2}&\theta(y_{1})&\eta_{3}
\end{pmatrix}\in\mathcal{J}(3),
$
we define their inner product by
\begin{equation}\label{inner product on J(3)}
<X,Y>^{\mathcal{J}(3)}:=\text{tr}(X\circ Y)=\sum\limits_{i=1}^{3}\,\xi_{i}\eta_{i}+2<x_{i},y_{i}>^{\mathbb{R}^{8}},
\end{equation}
where $<\cdot,\cdot>^{\mathbb{R}^{8}}$ denotes the standard Euclidean inner
product of $x_{i}~\text{and}~y_{i}\in \mathbb{O}\cong\mathbb{R}^{8}$.

This inner product has a property
\begin{equation}\label{selfadjointness of Jordan product in the inner product}
<X\circ Y,\,Z>^{\mathcal{J}(3)}=<X,Y\circ Z>^{\mathcal{J}(3)}, ~X,Y,Z\in\mathcal{J}(3).
\end{equation}

In particular, since the trace function 
$\mathcal{J}(3)\ni A\longmapsto \text{tr}\,(A)$ is invariant under the
$F_{4}$ action, that is
\begin{equation}\label{invariance of trace} 
\text{tr}\,(g\cdot A)=\text{tr}\,(A), ~g\in F_{4},~A\in\mathcal{J}(3),
\end{equation}
this inner product is invariant under the action by
$F_{4}$ (hence $F_{4}$ can be seen as $F_{4}\subset SO(27)$):
\begin{equation}\label{invariance of inner product}
<\,g\cdot A\,,\,g\cdot B>^{\mathcal{J}(3)}=\text{tr}(\,g\cdot A\circ g\cdot B\,)
=\text{tr}\,\big(g\cdot(A\circ B)\big)=\text{tr}\,(A\circ B)=<A,B>^{\mathcal{J}(3)}.%\,,~g\in F_{4},~A,B\in\mathcal{J}(3).
\end{equation}

The tangent bundle $T(P^{2}\mathbb{O})$ is identified with a subspace in
$\mathcal{J}(3)\times\mathcal{J}(3)$ such that
\[
T(P^{2}\mathbb{O})=\left\{(X,Y)\in\mathcal{J}(3)\times\mathcal{J}(3)~\Big|~X\in
  P^2\mathbb{O},~X\circ Y=\frac{1}{2}Y\right\}.
\]

{\it We consider the Riemannian metric
$g^{P^2\mathbb{O}}$ on the manifold $P^2\mathbb{O}$
being induced from the
inner product in $\mathcal{J}(3)$ :
$g_{X}^{P^2\mathbb{O}}(Y_{1},Y_{2}):=<Y_{1},Y_{2}>^{\mathcal{J}(3)}$,
$Y_{1},~Y_{2}\in T_{X}(P^{2}\mathbb{O})$}.

{\it Using this metric, 
hereafter we identify
the tangent bundle $T(P^{2}\mathbb{O})$ 
and the cotangent bundle $T^{*}(P^{2}\mathbb{O})$.}
\smallskip

Let
$Y_{1}=\begin{pmatrix}0&\frac{1}{\sqrt{2}}&0\\\frac{1}{\sqrt{2}}&0&0\\0&0&0\end{pmatrix}
\in T_{X_{1}}(P^2\mathbb{O})$. 
The stationary
subgroup at the point $(X_{1},Y_{1})\in T(P^2\mathbb{O})$ is known as being
isomorphic to $Spin(7)$ and the two point homogeneity of the action by
$F_{4}$ gives us the isomorphism 
$F_{4}/Spin(7)\cong S(P^2\mathbb{O})$, 
the unit (co)tangent  sphere bundle of $P^2\mathbb{O}$.

The inner product on $\mathcal{J}(3)$, $<\,\cdot\,,\, \cdot>^{\mathcal{J}(3)}$,
is extended to the complexification $\mathcal{J}(3)^{\mathbb{C}}$
as a {\it complex bi-linear form} in a natural way, which we denote 
by $<\cdot\, ,\,\cdot>^{\mathcal{J}(3)^{\mathbb{C}}}$. 
Then the extension as the {\it Hermitian inner product}
on the complexification $\mathcal{J}(3)^{\mathbb{C}}$ 
is given by $<A,\overline{B}>^{\mathcal{J}(3)^{\mathbb{C}}}$, $A,~B\in
\mathcal{J}(3)^{\mathbb{C}}$ (see Remark \ref{remark on conjugation operation}  
for the matrix $\overline{B}$).

{\it We will denote the norm of $a\in\mathbb{O}$ 
by $|a|=\sqrt{<a,a>^{\mathbb{R}^8}}$ and
by $||X||=\sqrt{<X,X>^{\mathcal{J}(3)}}$ the norm
of $X\in\mathcal{J}(3)$, respectively. Also with the same way 
for elements $a\in\mathbb{C}\otimes_{\mathbb{R}}\hspace{-0.08cm}{\mathbb{O}}$
and $A\in\mathcal{J}(3)^{\mathbb{C}}$, we denote their norms.}

The punctured cotangent bundle
$T^{*}(P^{2}\mathbb{O})\backslash\{0\}=:T^{*}_{0}(P^2\mathbb{O})$
is realized as a subspace in $\mathcal{J}(3)^{\mathbb{C}}$ with the following form:
\begin{theorem}{\em(\cite{Fu2})}
Let $\mathbb{X}_{\mathbb{O}}$ be a subspace in $\mathcal{J}(3)^{\mathbb{C}}$:
\begin{align}
\mathbb{X}_{\mathbb{O}}=\left\{~
A=\begin{pmatrix}
\xi_{1}&z&\theta(y)\\
\theta(z)&\xi_{2}&x\\
y&\theta(x)&\xi_{3}
\end{pmatrix}~\Bigg|~x,y,z\in\mathbb{C}{\otimes_{\mathbb{R}}}\hspace{-0.01cm}{\mathbb{O}},~\xi_{i}\in\mathbb{C},~A^2=0, ~A\not=0
\right\}.\label{condition Cayley cotangent}
\end{align}
Then the  correspondence between $T^{*}_{0}(P^{2}\mathbb{O})$ and $\mathbb{X}_{\mathbb{O}}$
is given by
\begin{equation}\label{complex structure map}
\tau_{\mathbb{O}}:
T^{*}_{0}(P^{2}\mathbb{O})\ni (X,Y)\longmapsto
\tau_{\mathbb{O}}(X,Y)=1\otimes\left(||Y||^2X-Y^{2}\right)+\sqrt{-1}\otimes\frac{||Y||Y}{\sqrt{2}}.
\end{equation}
\end{theorem}

Then
\begin{theorem}{\em(\cite{Fu2})}\label{Kaehler form}
\begin{equation}\label{Kaehler form and symplectic form}
{\tau_{\mathbb{O}}}^{*}\left(\sqrt{-2}\,\,\overline{\partial}\partial \,||A||^{1/2}\right)={\omega^{P^2\mathbb{O}}},
\end{equation}
where we denote by $\omega^{P^2\mathbb{O}}$ the natural
symplectic form on the cotangent bundle $T^{*}(P^{2}\mathbb{O})$.
\end{theorem}
The inverse ${\tau_{\mathbb{O}}}^{-1}$ is given by
\[
{\tau_{\mathbb{O}}}^{-1}:\mathbb{X}_{\mathbb{O}}\ni A\longmapsto (X,Y)=(X(A),Y(A))\in\mathcal{J}(3)\times \mathcal{J}(3),
\]
\begin{equation}\label{inverse of tau-O}
\left\{
\begin{array}{l}
X(A)=\frac{1}{2||A||}\cdot \left(A+\overline{A}\right)+\frac{A\circ
  \overline{A}}{||A||^2},\\
\\
Y(A)=-\frac{\sqrt{-1}}{\sqrt{2}}\cdot{||A||^{-1/2}\left(A-\overline{A}\right)}.
\end{array}\right.
\end{equation}

\section{Complex coordinate neighborhoods and Calabi-Yau structure}
{\it We denote the holomorphic part of the complexified cotangent bundle 
$T^{*}(\mathbb{X}_{\mathbb{O}})\otimes\mathbb{C}$ by
$T^{*'}(\mathbb{X}_{\mathbb{O}})^{\mathbb{C}}$ (and likewise
$T^{*''}(\mathbb{X}_{\mathbb{O}})^{\mathbb{C}}$ is the anti-holomorphic subbundle).}
 
In this section we show that the canonical line bundle 
$\stackrel{16}{\bigwedge} T^{*\,'}(\mathbb{X}_{\mathbb{O}})^{\mathbb{C}}$ 
is holomorphically
trivial by explicitly constructing  a nowhere vanishing 
global holomorphic section (Theorem \ref{holomorphic global section}).

For this purpose we consider an open covering by explicit coordinates neighborhoods
and show that the Jacobians of the coordinates transformations is 
a coboundary form of $\mathbb{C}^{*}$-valued zero form.
 
The condition in \eqref{condition Cayley cotangent} 
is expressed in the following six equations in terms of octanions:
\begin{align}
&(\xi_{3}+\xi_{2})x+\theta(yz)=0,%\label{2,3},
~(\xi_{1}+\xi_{3})y+\theta(zx)=0,%\label{3,1},
~(\xi_{2}+\xi_{1})z+\theta(xy)=0,\label{2,3}\\
&{\xi_{1}}^2+z\theta(z)+\theta(y)y=0,%\label{1,1}
~{\xi_{2}}^2+\theta(z)z+x\theta(x)=0,%\label{2,2}
~{\xi_{3}}^2+\theta(x)x+y\theta(y)=0.\label{3,3}
\end{align}

The condition $0\not=A\in\mathbb{X}_{\mathbb{O}}$ is equivalent to 
one of the components $x,y$, or $z$ being non zero.
Then this implies
\begin{proposition}\label{trace is zero}
\[
\mathbb{X}_{\mathbb{O}}\ni A, ~\text{then}~\text{{\em tr}}\,(A)=\xi_{1}+\xi_{2}+\xi_{3}=0.
\]
\end{proposition}
This property does not appear in an explicit form in 
{\eqref{2,3} and \eqref{3,3}} %\eqref{condition Cayley cotangent} 
but plays an important role in $\S 8$. 
Although it is proved in \cite{Fu2}, we give an elementary proof based on the permitted regulations in the octanion. 
\begin{proof}
Since the associativity 
\[
a\cdot\theta(a)b=a\theta(a)\cdot b
\]
holds,
by multiplying $z$ from the left to the equality $(\xi_{3}+\xi_{2})x+\theta(yz)=0$ 
it holds the equality:
\begin{align*}
&z\cdot (\xi_{3}+\xi_{2})x+z\cdot \theta(z)\theta(y)=
(\xi_{3}+\xi_{2})zx+ z\theta(z)\cdot \theta(y)=-(\xi_{3}+\xi_{2})(\xi_{1}+\xi_{3})\theta(y)+ z\theta(z)\cdot \theta(y)=0.
\end{align*}
Hence if we assume $y\not=0$
\begin{align*}
&(\xi_{3}+\xi_{2})(\xi_{1}+\xi_{3})=z\theta(z)\\
\intertext{and by the same way}
&(\xi_{2}+\xi_{1})(\xi_{3}+\xi_1)=\theta(x)x.
\end{align*}
These imply that
\[
(\xi_{3}+\xi_{2})(\xi_{1}+\xi_{3})+(\xi_{2}+\xi_{1})(\xi_{3}+\xi_1)+{\xi_{2}}^2
=(\xi_{1}+\xi_{2}+\xi_{3})^2=0.
\]
and we have
\[
\xi_{1}+\xi_{2}+\xi_{3}=0.
\]
From the arguments above the same holds for other cases of $x\not=0$ or $z\not=0$.
\end{proof}
\begin{remark}\label{remark on trace zero proof}
The property above can be seen easily, if we use the
transitivity 
of the action of the group $F_{4}$ on the {\em (co)}tangent sphere bundle.

Also from the definition of the map $\tau_{\mathbb{O}}$,
$\text{\em tr}(A)=\text{\em tr}(\tau_{\mathbb{O}}(X,Y))=0$ is
equivalent to $\text{\em tr}\,(Y)=0$.
\end{remark}

Here we mention the following fact, which is a special case
described in Proposition \ref{vanishing of singular harmonic functions}.
\begin{lemma}\label{uniqueness of the zero linear form}
Assume that a linear function $f:\mathcal{J}(3)^{\mathbb{C}}\to\mathbb{C}$ 
\[
f(A)=2\sum\limits_{i=0}^{7}\,(a_{i}\{w\}_{i}+b_{i}\{v\}_{i}+ c_{i}\{u\}_{i})+\sum_{i=1}^{3} \alpha_{i}\xi_{i}
\]
vanishes on $\mathbb{X}_{\mathbb{O}}$. 
Then $f$ is a constant multiple of the trace function $A\mapsto
\text{\em tr}\,(A),~ A\in \mathcal{J}(3)^{\mathbb{C}}$. 
\end{lemma}
\allowdisplaybreaks{
\begin{proof}
Put $a=\sum \,a_{i}{\bf e}_{i}$, $b=\sum \,b_{i}{\bf e}_{i}$, $c=\sum
\,c_{i}{\bf e}_{i}\in\mathbb{C}\otimes_{\mathbb{R}}\mathbb{O}$ and
$
B=\begin{pmatrix}\alpha_{1}&a&\theta(b)\\\theta(a)&\alpha_{2}&c\\b&\theta(c)&\alpha_{3}\end{pmatrix}
\in \mathcal{J}(3)^{\mathbb{C}}.%,~\text{where}~\alpha_{j}\in\mathbb{C}.
$

Then 
\[
f(A)=\text{tr}\,(A\circ B):=f_{B}(A).
\]
Let
$Y=\begin{pmatrix}0&z&\,\theta(y)\\\theta(z)&0&0\\y&0&0\end{pmatrix}\in T_{X_{1}}(P^{2}\mathbb{O})$,
where $z=\sum z_{i}{\bf e}_{i},y=\sum y_{i}{\bf e}_{i}\in\mathbb{O}$,
then
\begin{align*}
&\tau_{\mathbb{O}}(X_{1},Y)=
\begin{pmatrix}|z|^2+|y|^2&0&0\\0&-|z|^2&-\theta(yz)\\0&-yz&-|y|^2\end{pmatrix}
+\sqrt{|z|^2+|y|^2}
\begin{pmatrix}0&\sqrt{-1}z&\sqrt{-1}\theta(y)\\\sqrt{-1}\theta(z)&0&0\\\sqrt{-1}y&0&0\end{pmatrix}\in \mathbb{X}_{\mathbb{O}}.
\end{align*}
Here we put $y=0$, and we may assume for such $A=\tau_{\mathbb{O}}(X_{1},Y)$ 
\[
f_{B}(A)=\text{tr}\,(B\circ A)=(|z|^2)(\alpha_{1}-\alpha_{2})+
2\sum\sqrt{-1}|z|z_{i}a_{i}=0,
\]
for any $\pm z_{i}\in\mathbb{R}$.
Then $\alpha_{1}=\alpha_{2}$ and also $a_{i}=0$ for
$i=0,\cdots,7$. Likewise we have $\alpha_{1}=\alpha_{3}$ and
$b_{i}=0$ for $i=0,\cdots,7$.

Then we may assume
\[
f_{B}(\tau_{\mathbb{O}}(X_{1},Y))= <\sum\,c_{i}{\bf e}_{i},\theta(yz)>^{\mathbb{R}^{8}}=0
~\text{for any}~y,z\in\mathbb{O}.
\]
Hence $c_{i}=0$ for $i=0,\cdots,7$, which shows our assertion, that is
$a=b=c=0$, $\alpha:=\alpha_{1}=\alpha_{2}=\alpha_{3}$ and
\[
f_{B}(A)=\alpha_{1}\xi_{1}+\alpha_{2}\xi_{2}+\alpha_{3}\xi_{3}=\alpha\cdot\text{tr}\,(A).
\]
\end{proof}
}
\begin{corollary}\label{26 dim}
The space spanned by $\mathbb{X}_{\mathbb{O}}$
{\em(}$:=[\,\mathbb{X}_{\mathbb{O}}\,]${\em)} is a $26$-dimensional subspace in $\mathcal{J}(3)^{\mathbb{C}}$. 
\end{corollary}
Let $z,y,x\in\mathbb{C}\otimes_{\mathbb{R}}\hspace{-0.08cm}\mathbb{O}$
and put
\begin{equation}\label{rep of complex octanion by  2 by 2 complex matrix 1}
\left\{\begin{array}{l}
\rho_{\mathbb{\mathbb{O}}}(z)=Z+W{\bf e}_{4}=\begin{pmatrix}z_{1}&z_{2}\\z_{3}&z_{4}\end{pmatrix}+\begin{pmatrix}w_{1}&w_{2}\\w_{3}&w_{4}\end{pmatrix}{\bf e}_{4},\\
\rho_{\mathbb{O}}(y)=Y+V{\bf e}_{4}
=\begin{pmatrix}y_{1}&~y_{2}\\y_{3}&~y_{4}\end{pmatrix}+\begin{pmatrix}v_{1}&~v_{2}\\v_{3}&~v_{4}\end{pmatrix}{\bf e_{4}},\\
\rho_{\mathbb{O}}(x)=X+U{\bf e}_{4}
=\begin{pmatrix}x_{1}&x_{2}\\x_{3}&x_{4}\end{pmatrix}
+\begin{pmatrix}u_{1}&u_{2}\\u_{3}&u_{4}\end{pmatrix}{\bf e}_{4},
~\,\,\text{where}~z_{i},w_{i},y_{i},v_{i},x_{i},u_{i}\in\mathbb{C}.\\
\end{array}\right.
\end{equation}
Then the conditions \eqref{2,3} and \eqref{3,3}
are rewritten in terms of the matrices $Z$, $W$, $Y$, $V$, $X$, $U$ as
\begin{equation}\label{condition for off diagonal}
\left\{\begin{array}{l}
\xi_{1}(\theta(X)-U{\bf e}_{4})
=(Y+V{\bf e}_{4})(Z+W{\bf e}_{4})=YZ-\theta(W)V+(WY+V\theta(Z)){\bf e}_{4},\\
\xi_{2}(\theta(Y)-V{\bf e}_{4})
=(Z+W{\bf e}_{4})(X+U{\bf e}_{4})=ZX-\theta(U)W+(UZ+W\theta(X)){\bf e}_{4},\\
\xi_{3}(\theta(Z)-W{\bf e}_{4})
=(X+U{\bf e}_{4})(Y+V{\bf e}_{4})=XY-\theta(V)U+(VX+U\theta(Y)){\bf e}_{4},\\
\end{array}\right.
\hspace{-8mm}
\end{equation}
\vspace{-0.5cm}
\begin{equation}\label{condition for diagonal}
\left\{
\begin{array}{l}
{\xi_{1}}^2+\det Z+\det W+\det Y+\det V=0,\\
{\xi_{2}}^2+\det Z+\det W+\det X+\det U=0,\\
{\xi_{3}}^3+\det Y+\det V+\det X+\det U=0.
\end{array}\right.
\hspace{47mm}
\end{equation}

{\it Hereafter (until $\S 7$), we denote the matrix $A
=\begin{pmatrix}\xi_{1}&z&\theta(y)\\\theta(z)&\xi_{2}&x\\y&\theta(x)&\xi_{3}\end{pmatrix}
\in\mathcal{J}(3)^{\mathbb{C}}
$
in the form of a vector $\in\mathbb{C}^{27}$:
\begin{align}
A\longleftrightarrow&
(\xi_{1},\xi_{2},\xi_{3},z_{1},\ldots,z_{4},w_{1},\ldots,w_{4},y_{1},\ldots,y_{4},v_{1},\ldots,v_{4},x_{1},\ldots,x_{4},u_{1},\ldots,u_{4})\label{vector expression of A}\\
&=(a_{1},a_{2},a_{3},a_{4},\ldots,\ldots,a_{27})
\in \mathbb{C}^{27},\notag
\end{align}
using the components given in \eqref{rep of complex octanion by  2 by 2
complex matrix 1}  by the map $\rho_{\mathbb{O}}$.}

The conditions for matrices in $\mathbb{X}_{\mathbb{O}}$ require that at least one
of the off-diagonal components in the matrix 
$A$ is non-zero. Hence, for example, 
we assume that there is at least one component in the matrix 
$\rho_{\mathbb{O}}(z)=Z+W{\bf e}_{4}=\begin{pmatrix}z_{1},&z_{2}\\z_{3}&z_{4}\end{pmatrix}
+\begin{pmatrix}w_{1}&w_{2}\\w_{3}&w_{4}\end{pmatrix}{\bf e}_{4}$, say 
$z_{1}\not=0$ and put
$O_{z_{1}}=\{A\in\mathbb{X}_{\mathbb{O}}~|~z_{1}\not=0\}$. Also we
define other open subsets $\{O_{z_{i}}, O_{w_{i}},
O_{y_{i}},O_{v_{i}},O_{x_{i}},O_{u_{i}}\}_{i=1}^{4}$ in a same way like $O_{z_{1}}$.
Then we have
\begin{proposition}\label{open covering}
The $24$ subsets 
\begin{equation}\label{notation of open covering}
\{O_{z_{i}}, O_{w_{i}},
O_{y_{i}},O_{v_{i}},O_{x_{i}},O_{u_{i}}\}_{i=1}^{4}=:\mathcal{U}_{0}
\end{equation} 
are all open coordinate neighborhoods and totally is an open covering of $\mathbb{X}_{\mathbb{O}}$.
\end{proposition}
\begin{proof}
We show a coordinates for the case $O_{z_1}$. Other cases will be shown by the same way.

From the equations in \eqref{condition for off diagonal} we 
select 5 equations expressed in $2\times 2$ complex matrices including the complex variable $z_{1}$
and from the equation \eqref{condition for diagonal}
we select one equation also including the complex variable $z_{1}$: 
\allowdisplaybreaks{
\begin{equation}\label{matrix equation including $z_{1}$}
\left\{
\begin{array}{l}
\xi_{1}\begin{pmatrix}x_{4}&-x_{2}\\-x_{3}&x_{1}\end{pmatrix}
=\begin{pmatrix}y_{1}&y_{2}\\y_{3}&y_{4}\end{pmatrix}
\begin{pmatrix}z_{1}&z_{2}\\z_{3}&z_{4}\end{pmatrix}
-\begin{pmatrix}w_{4}&-w_{2}\\-w_{3}&w_{1}\end{pmatrix}
\begin{pmatrix}v_{1}&v_{2}\\v_{3}&v_{4}\end{pmatrix},\\
-\xi_{1}\begin{pmatrix}u_{1}&~u_{2}\\u_{3}&~u_{4}\end{pmatrix}
=\begin{pmatrix}w_{1}&w_{2}\\w_{3}&w_{4}\end{pmatrix}
\begin{pmatrix}y_{1}&y_{2}\\y_{3}&y_{4}\end{pmatrix}
+\begin{pmatrix}v_{1}&~v_{2}\\v_{3}&~v_{4}\end{pmatrix}
\begin{pmatrix}z_{4}&-z_{2}\\-z_{3}&z_{1}\end{pmatrix},\\
\xi_{2}\begin{pmatrix}y_{4}&-y_{2}\\-y_{3}&y_{1}\end{pmatrix}
=\begin{pmatrix}z_{1}&z_{2}\\z_{3}&z_{4}\end{pmatrix}
\begin{pmatrix}x_{1}&x_{2}\\x_{3}&x_{4}\end{pmatrix}
-\begin{pmatrix}u_{4}&-u_{2}\\-u_{3}&u_{1}\end{pmatrix}
\begin{pmatrix}w_{1}&w_{2}\\w_{3}&w_{4}\end{pmatrix},\\
-\xi_{2}\begin{pmatrix}v_{1}&~v_{2}\\v_{3}&~v_{4}\end{pmatrix}
=\begin{pmatrix}u_{1}&u_{2}\\u_{3}&u_{4}\end{pmatrix}
\begin{pmatrix}z_{1}&z_{2}\\z_{3}&z_{4}\end{pmatrix}
+\begin{pmatrix}w_{1}&w_{2}\\w_{3}&w_{4}\end{pmatrix}
\begin{pmatrix}x_{4}&-x_{2}\\-x_{3}&x_{1}\end{pmatrix},\\
\xi_{3}\begin{pmatrix}z_{4}&-z_{2}\\-z_{3}&z_{1}\end{pmatrix}
=\begin{pmatrix}x_{1}&x_{2}\\x_{3}&x_{4}\end{pmatrix}
\begin{pmatrix}y_{1}&y_{2}\\y_{3}&y_{4}\end{pmatrix}
-\begin{pmatrix}v_{4}&-v_{2}\\-v_{3}&v_{1}\end{pmatrix}
\begin{pmatrix}u_{1}&u_{2}\\u_{3}&u_{4}\end{pmatrix},\\
{\xi_{2}}^2+z_{1}z_{4}-z_{2}z_{3}+w_{1}w_{4}-w_{2}w_{3}+x_{1}x_{4}-x_{2}x_{3}+u_{1}u_{4}-u_{2}u_{3}=0.
%\xi_{1}+\xi_{2}+\xi_{3}=0~\text{: (\it this is common for all cases in the covering $\mathcal{U}_{0}${\em )}}.
\end{array}\right.
\end{equation}
%\end{align*}
}
\allowdisplaybreaks{
From these we can select $10$ equations including the variable $z_{1}$:
\begin{equation}\label{components equations of A times A}
\left\{
\begin{array}{l}
f_{1}=-\xi_{2}y_{4}+z_{1}x_{1}+z_{2}x_{3}-(u_{4}w_{1}-u_{2}w_{3})=0,\\
f_{2}=\xi_{2}y_{2}+z_{1}x_{2}+z_{2}x_{4}-(u_{4}w_{2}-u_{2}w_{4})=0,\\
f_{3}=\xi_{2}v_{1}+u_{1}z_{1}+u_{2}z_{3}+(w_{1}x_{4}-w_{2}x_{3})=0,\\
f_{4}=\xi_{2}v_{3}+u_{3}z_{1}+u_{4}z_{3}+(w_{3}x_{4}-w_{4}x_{3})=0,\\
f_{5}=-\xi_{1}x_{4}+y_{1}z_{1}+y_{2}z_{3}-(w_{4}v_{1}-w_{2}v_{3})=0,\\
f_{6}=\xi_{1}x_{3}+y_{3}z_{1}+y_{4}z_{3}-(-w_{3}v_{1}+w_{1}v_{3})=0,\\
f_{7}=\xi_{1}u_{2}-v_{1}z_{2}+v_{2}z_{1}+w_{1}y_{2}+w_{2}y_{4}=0,\\
f_{8}=\xi_{1}u_{4}-v_{3}z_{2}+v_{4}z_{1}+w_{3}y_{2}+w_{4}y_{4}=0,\\
f_{9}=-\xi_{3}z_{1}+x_{3}y_{2}+x_{4}y_{4}-(-v_{3}u_{2}+v_{1}u_{4})=0,\\
{f}_{10}={\xi_{2}}^2+z_{1}z_{4}-z_{2}z_{3}+w_{1}w_{4}-w_{2}w_{3}+x_{1}x_{4}-x_{2}x_{3}+u_{1}u_{4}-u_{2}u_{3}=0.
%f_{11}=\xi_{1}+\xi_{2}+\xi_{3}=0.
\end{array}\right.
\end{equation}}
The $10$ variables 
\begin{equation}\label{dependent variables}
x_{1},x_{2},u_{1},u_{3},y_{1},y_{3},v_{2},v_{4},\xi_{3},z_{4}
\end{equation}
are coefficients of the variable $z_{1}$, and
can be solved easily. 

In fact, with one more additional equation
\begin{equation}\label{common linear equation}
f_{11}=\xi_{1}+\xi_{2}+\xi_{3}=0,
\end{equation}
we can solve the $11$ variables  
\begin{equation}\label{dependent variables 2}
\{x_{1},x_{2},u_{1},u_{3},y_{1},y_{3},v_{2},v_{4},\,z_{4},\,\xi_{3},\,\xi_{1}\}
\end{equation}
in terms of the remaining $16$ variables
\begin{equation}\label{local coordinates}
\{x_{3},x_{4},u_{2},u_{4},y_{2},y_{4},v_{1},v_{3},\,z_{1},z_{2},z_{3},w_{1},w_{2},w_{3},w_{4}\,,\,\xi_{2}\},
\end{equation}
in which, except $z_{1}\not=0$ other variables 
%\[
%\{x_{3},x_{4},u_{2},u_{4},y_{2},y_{4},v_{1},v_{3},\,,z_{2},z_{3},w_{1},w_{2},w_{3},w_{4}\,,\,\xi_{2}\}
%\] 
can take any values in $\mathbb{C}$.

Here, if we choose the equation
\[
{f}_{10}={\xi_{1}}^2+z_{1}z_{4}-z_{2}z_{3}+w_{1}w_{4}-w_{2}w_{3}+y_{1}y_{4}-y_{2}y_{3}+v_{1}v_{4}-v_{2}v_{3}=0,
\]
instead of the tenth equation $f_{10}$ in 
\eqref{components equations of A times A}(= first equation in \eqref{condition for diagonal}),
then the variable $\xi_{1}$ should be chosen as an independent variable.

In any choice(in $O_{z_{1}}$ case, $\xi_{1}$ or $\xi_{2}$) 
once we fix
them (here we choose as above), and 
denote  by $P_{z_{1}}$ the projection map 
\begin{align*}%\label{coordinates map}
P_{z_{1}}:O_{z_{1}}\ni
&(\xi_{1},\xi_{2},\xi_{3},z_{1},\ldots,z_{4},w_{1},\ldots,w_{4},y_{1},\ldots,y_{4},v_{1},\ldots,v_{4},x_{1},\ldots,x_{4},u_{1},\ldots,u_{4})\\
&=(a_{1},\ldots,a_{27})
\longmapsto
  (\xi_{2},z_{1},z_{2},z_{3},w_{1},w_{2},w_{3},w_{4},y_{2},y_{4},v_{1},v_{3},x_{3},x_{4},u_{2},u_{4})\\
&=(a_{2},a_{4},a_{5},a_{6},a_{8},a_{9},a_{10},a_{11},a_{13},a_{15},a_{16},a_{18},a_{22},a_{23},a_{25},a_{27})
\in\mathbb{C}\times\mathbb{C}^{*}\times\mathbb{C}^{14}.
\end{align*}
Then, the pair $(O_{z_{1}},P_{z_{1}})$ is a local coordinates neighborhood 
(note that $\dim_{\mathbb{C}} \mathbb{X}_{\mathbb{O}}=16$).
\end{proof}
{\it In any case in $\mathcal{U}_{0}$, once we fix
independent variables, then
we denote the dependent variables as $x_{1}(*), x_{2}(*)\cdots,$ {\em(}or
$a_{1}(*), a_{2}(*),\ldots,${\em)} etc., where $*$ means the independent variables.}

\begin{corollary}\label{codimension of the coordinate neighborhood}
Each coordinate neighborhood ${O}_{\ell}$ in $\mathcal{U}_{0}$
is dense in $\mathbb{X}_{\mathbb{O}}$. 
Hence any number of intersections of open sets in $\mathcal{U}_{0}$
is also open dense.
\end{corollary}
\begin{proof}
It will be enough to show the case ${O}_{z_{1}}$. So, let $A\in \mathbb{X}_{\mathbb{O}}\backslash {O}_{z_{1}}$.  
Assume, say $A\in {O}_{x_{1}}$, then the subset $z_{1}=0$ is defined by
an rational equation:
$z_{1}=\frac{\xi_{2}y_{4}-z_{2}x_{3}+u_{4}w_{1}-u_{2}w_{3}}{x_{1}}=0$.
Hence the subset $z_{1}=0$ must be at most codimension $1$ in
$\mathbb{X}_{\mathbb{O}}$. 
\end{proof}

\begin{proposition}\label{coordinates transformation and Jacobian}
Let ${O}_{a_{i}}$ and ${O}_{a_{j}}$ be any of two open coordinate neighborhoods in
$\mathcal{U}_{0}=\{O_{a_{i}}\}_{i=1}^{24}$. Then the Jacobian $J_{a_{j},a_{i}}=\det \,d\big(P_{a_{j}}\circ {P_{a_{i}}}^{\hspace{-0.2cm}-1}\big)$ 
of the coordinate transformation $P_{a_{j}}\circ {P_{a_{i}}}^{\hspace{-0.2cm}-1}$ is given by
\begin{equation}\label{formula of general Jacobian}
J_{a_{j},a_{i}}=\left(\frac{a_{j}}{a_{i}}\right)^{5}~\text{on $P_{a_{i}}({O}_{a_{j}}\cap {O}_{a_{i}})$}.
\end{equation}
\end{proposition}
\begin{proof}
Let $\sigma=\begin{pmatrix}0&~1\\1&~0\end{pmatrix}$ and define
a map
\begin{equation}
\tilde{\sigma}:\mathbb{C}(2)\ni S\longmapsto \tilde{\sigma}(S):=\sigma\cdot S\cdot\sigma\in\mathbb{C}(2),
\end{equation}
then $\theta(\tilde{\sigma}(S))=\tilde{\sigma}(\theta(S))$. This
property of $\tilde{\sigma }$ 
naturally induces an automorphism of $\mathcal{J}(3)^{\mathbb{C}}$,
which we denote by the same notation $\tilde{\sigma}:\mathcal{J}(3)^{\mathbb{C}}\to\mathcal{J}(3)^{\mathbb{C}}$.

By the Lemma \ref{codimension of the coordinate neighborhood}, it will
be enough to determine the Jacobian $J_{z_{1}, a_{i}}$ for the cases of
\[
O_{z_{1}}\cap O_{a_{i}}=O_{a_{4}}\cap O_{a_{i}},~\text{for}~i\geq 5.
\] 
Also by the symmetry of the components $y$ and $x$ it is enough for the
cases of
\[
O_{z_{1}}\cap O_{a_{i}}=O_{a_{4}}\cap O_{a_{i}},~\text{for}~i=17 \sim 24.
\]
Finally, by the automorphism $\tilde{\sigma}$ explained above and the symmetry between $z$ and $x$ 
we see that it is enough to determine them for the $5$ cases
\begin{align*}
O_{z_{1}}\cap O_{z_{2}},
O_{z_{1}}\cap O_{z_{4}},
O_{z_{1}}\cap O_{w_{1}},
O_{z_{1}}\cap O_{w_{2}},
O_{z_{1}}\cap O_{x_{1}}.
\end{align*}

All the determinations can be done by the basic way of the calculation of the
determinants. So we show two cases $O_{z_{1}}\cap O_{z_{2}}$ and 
$O_{z_{1}}\cap O_{w_{1}}$ how they look like. 

%We mention that for the case $O_{z_{1}}\cap O_{x_{1}}$, the
%calculation is a little bit lengthy.

{\bf [\,I\,]} $O_{z_{1}}\cap O_{z_{2}}= O_{a_{4}}\cap O_{a_{5}}$ case: For this case we consider the coordinate
transformation $P_{z_{2}}\circ {P_{z_{1}}}^{\hspace{-0.2cm}-1}$, which is given by the
correspondence:
\begin{align*}
&(\xi_{2},z_{1},z_{2},z_{3},w_{1},w_{2},w_{3},w_{4},y_{2},y_{4},v_{1},v_{3},x_{3},x_{4},u_{2},u_{4})\\
&\longmapsto
(\xi_{2},z_{1},z_{2},z_{4},w_{1},w_{2},w_{3},w_{4},y_{2},y_{4},v_{2},v_{4},x_{1},x_{2},u_{2},u_{4})
\end{align*}
where the coordinates 
$(\xi_{2},z_{1},z_{2},z_{4},w_{1},w_{2},w_{3},w_{4},y_{2},y_{4},v_{2},v_{4},x_{1},x_{2},u_{2},u_{4})$
are given by the rational functions:
\allowdisplaybreaks{
\begin{equation}\label{explicit expression of dependent variables}
\left\{\begin{array}{l}
x_{1}=\frac{\xi_{2}y_{4}-z_{2}x_{3}+(u_{4}w_{1}-u_{2}w_{3})}{z_{1}},\,\, 
x_{2}=\frac{-\xi_{2}y_{2}-z_{2}x_{4}+(u_{4}w_{2}-u_{2}w_{4})}{z_{1}},\quad
u_{2}=u_{2},\,\, u_{4}=u_{4},\\
y_{2}=y_{2},\,\,,y_{4}=y_{4},\quad 
v_{2}=\frac{-\xi_{1}u_{2}+v_{1}z_{2}-w_{1}y_{2}-w_{2}y_{4}}{z_{1}},\,\,\,
v_{4}=\frac{-\xi_{1}u_{4}+v_{3}z_{2}-w_{3}y_{2}-w_{4}y_{4}}{z_{1}},\\
z_{1}=z_{1},\,\,\,z_{2}=z_{2},\,\,\,
z_{4}=\frac{-{\xi_{2}}^2+z_{2}z_{3}-w_{1}w_{4}+w_{2}w_{3}-x_{1}x_{4}+x_{2}x_{3}-u_{1}u_{4}+u_{2}u_{3}}{z_{1}},\\
w_{1}=w_{1},\quad w_{2}=w_{2},\quad w_{3}=w_{3},\quad
  w_{4}=w_{4},\quad \xi_{2}=\xi_{2}.
\end{array}
\right.
\end{equation}
}
We change the orderings of the coordinates in $P_{z_{1}}(O_{z_{1}})$
with ``{\it even}'' permutations as
\begin{align*}
&(\xi_{2},z_{1},z_{2},w_{1},w_{2},w_{3},w_{4},y_{2},y_{4},u_{2},u_{4},z_{3},v_{1},v_{3},x_{3},x_{4})
\intertext{and $P_{z_{2}}(O_{z_{2}})$ as}
&(\xi_{2},z_{1},z_{2},w_{1},w_{2},w_{3},w_{4},y_{2},y_{4},u_{2},u_{4},z_{4},v_{2},v_{4},x_{1},x_{2}).
\end{align*}
Then the Jacobi matrix is of the form that
\begin{equation}
\begin{pmatrix}
Id_{11}&~C\\
\vspace{-0.3cm}&\\
0_{5,11}&~D
\end{pmatrix},
\end{equation}
where
$Id_{11}$ is $11\times 11$ identity matrix, $0_{5,11}$ is $5\times 11$ zero matrix
and $D$ is given by
\begin{equation}D=
\begin{pmatrix}
\frac{z_{2}}{z_{1}}&0&0&0&0\\
0&\frac{z_{2}}{z_{1}}&0&0&0\\
0&*&\frac{z_{2}}{z_{1}}&0&0\\
x_{2}&*&*&-\frac{z_{2}}{z_{1}}&0\\
-x_{1}&*&*&*&-\frac{z_{2}}{z_{1}}
\end{pmatrix}
\end{equation}
(the $11\times 5$ matrix $C$ and components $\vspace{0.15cm}*$ are given by some functions).
Hence the Jacobian $J_{z_{2},z_{1}}$ is
\[
J_{z_{2},z_{1}}=\det D=\left(\frac{z_{2}}{z_{1}}\right)^{5}.
\]
\smallskip

{\bf [II]} $O_{w_{1}}\cap O_{z_{1}}= O_{a_{8}}\cap O_{a_{4}}$ case: 
For this case we consider the coordinate
transformation $P_{w_{1}}\circ {P_{z_{1}}}^{\hspace{-0.2cm}-1}$, which is given by the
correspondence:
\begin{align*}
&(\xi_{2},z_{1},z_{2},z_{3},w_{1},w_{2},w_{3},w_{4},y_{2},y_{4},v_{1},v_{3},x_{3},x_{4},u_{2},u_{4})\\
&\longmapsto (\xi_{2},z_{1},z_{2},z_{3},z_{4},w_{1},w_{2},w_{3},y_{3},y_{4},v_{1},v_{2},x_{1},x_{3},u_{1},u_{2}),
\end{align*}
where the coordinates 
$(x_{1},x_{3},u_{1},u_{2},y_{3},y_{4},v_{1},v_{2},z_{1},z_{2},z_{3},z_{4},w_{1},w_{2},w_{3},\xi_{2})$
are given by the rational functions:
\allowdisplaybreaks{
\begin{equation}
\left\{
\begin{array}{l}
x_{1}=\frac{\xi_{2}y_{4}-z_{2}x_{3}+(u_{4}w_{1}-u_{2}w_{3}}{z_{1}},\quad
  x_{3}=x_{3},\quad u_{1}=\frac{-\xi_{2}v_{1}-u_{2}z_{3}-(w_{1}x_{4}-w_{2}x_{3}}{z_{1}},\quad
  u_{2}=u_{2},\\
y_{3}=\frac{-\xi_{1}x_{3}-y_{4}z_{3}+(-w_{3}v_{1}+w_{1}v_{3}}{z_{1}},\quad
  y_{4}=y_{4},\quad v_{1}=v_{1},\quad v_{2}=\frac{-\xi_{1}u_{2}+v_{1}z_{2}-w_{1}y_{2}-w_{2}y_{4}}{z_{1}},\\
z_{1}=z_{1},\quad z_{2}=z_{2},\quad z_{3}=z_{3},\quad z_{4}=\frac{-{\xi_{2}}^2+z_{2}z_{3}-w_{1}w_{4}+w_{2}w_{3}-x_{1}x_{4}+x_{2}x_{3}-u_{1}u_{4}+u_{2}u_{3}}{z_{1}},\\
w_{1}=w_{1},\quad w_{2}=w_{2},\quad w_{3}=w_{3},\quad \xi_{2}=\xi_{2}.
\end{array}
\right.
\end{equation}
}
We change the orderings of the coordinates in $P_{z_{1}}(O_{z_{1}})$
by the ``{\it odd}'' permutation as
\begin{align*}
&(\xi_{2},z_{1},z_{2},z_{3},w_{1},w_{2},w_{3},w_{4},y_{2},y_{4},v_{1},v_{3},x_{3},x_{4},u_{2},u_{4})\\
&\longmapsto (\xi_{2},z_{1},z_{2},z_{3},w_{1},w_{2},w_{3},y_{4},v_{1},x_{3},u_{2},w_{4},y_{2},v_{3},x_{4},u_{4})\\
\intertext{and $P_{w_{1}}(O_{w_{1}})$ by the {\it even} permutation as}
&(\xi_{2},z_{1},z_{2},z_{3},z_{4},w_{1},w_{2},w_{3},y_{3},y_{4},v_{1},v_{2},x_{1},x_{3},u_{1},u_{2})\\
&\longmapsto
(\xi_{2},z_{1},z_{2},z_{3},w_{1},w_{2},w_{3},y_{4},v_{1},x_{3},u_{2}, z_{4},y_{3},v_{2},x_{1},u_{1})\\
\end{align*}
Then the Jacobi matrix is of the form that
\begin{equation}
\begin{pmatrix}
Id_{11}&C'\\
0_{5,11}&D'
\end{pmatrix},
\end{equation}
where the matrix $D'$ is given by
\begin{equation}D'=
\begin{pmatrix}
-\frac{w_{1}}{z_{1}}&0&0&0&0\\
*&0&-\frac{w_{1}}{z_{1}}&0&0\\
*&\frac{w_{1}}{z_{1}}&0&0&0\\
*&0&0&0&-\frac{w_{1}}{z_{1}}\\
*&0&0&\frac{w_{w}}{z_{1}}&0
\end{pmatrix}.
\end{equation}
(the matrix $C'$ and components {$\vspace{0.15cm}*$} are given by some functions)
Hence the Jacobian $J_{w_{1},z_{1}}$ is
\[
J_{w_{1},z_{1}}=\det D'=-\left(\frac{w_{1}}{z_{1}}\right)^{5}.
\]
\end{proof}

From the above Proposition \ref{coordinates transformation and Jacobian}
we can see that the $\mathbb{C}^{*}$-valued $1$-cocycle defined by 
$\{J_{a_{j},a_{i}}\}_{a_{i},a_{j}\in \{z_{1},\cdots,\cdots,u_{4}\}}$
is the coboundary of $\mathbb{C}^{*}$-valued $0$-cochain $\{h_{i}=\frac{1}{a_{i}^{5}}\}$, we have
\begin{theorem}\label{holomorphic global section}
The set of holomorphic sections 
\[
\left\{h_{z_{i}}=\frac{1}{{z_{i}}^5},
h_{w_{i}}=\frac{1}{{w_{i}}^5},h_{y_{i}}=\frac{1}{{y_{i}}^5},h_{v_{i}}=\frac{1}{{v_{i}}^5},
h_{x_{i}}=\frac{1}{{x_{i}}^5},h_{u_{i}}=\frac{1}{{u_{i}}^5}\right\},
\]
each function is defined on the open coordinate neighborhood $O_{z_{i}}$, $O_{w_{i}}$ and so on,
together define a nowhere vanishing global holomorphic section $\Omega_{\mathbb{O}}$
of the canonical line bundle $\stackrel{16}\bigwedge T^{*'}(\mathbb{X}_{\mathbb{O}})^{\mathbb{C}}$.

Here the above local sections, for example $h_{z_{1}}$ defined on $O_{z_{1}}$, should be understood
as the coefficient of a $16$-degree {\em(}= highest degree{\em)} holomorphic differential form:
\begin{align*}
&h_{z_{1}}(\xi_{2},z_{1},z_{2},z_{3},w_{1},w_{2},w_{3},w_{4},y_{2},y_{4},v_{1},v_{3},x_{3},x_{4},u_{2},u_{4})\\
&=\frac{1}{{z_{1}}^5}\,
d\xi_{2}{\wedge dz_{1}}{\wedge dz_{2}}{\wedge dz_{3}}
{\wedge dw_{1}}{\wedge dw_{2}}{\wedge dw_{3}}{\wedge dw_{4}}{\wedge dy_{2}}{\wedge dy_{4}}{\wedge dv_{1}}
{\wedge dv_{3}}{\wedge dx_{3}}{\wedge dx_{4}}{\wedge du_{2}}{\wedge du_{4}}\\
&=\frac{1}{{a_{4}}^{5}}{da_{2}}{\wedge da_{4}}{\wedge
  da_{5}}{\wedge da_{6}}{\wedge da_{8}}
{\wedge da_{9}}{\wedge da_{10}}{\wedge da_{13}}{\wedge da_{15}}{\wedge
  da_{16}}{\wedge da_{18}}
{\wedge da_{22}}{\wedge da_{23}}{\wedge da_{25}}{\wedge da_{27}}.
\end{align*}
\end{theorem}
\begin{remark}
As in the case for the sphere, the nowhere vanishing
global holomorphic $16$-form $\Omega_{\mathbb{O}}$ coincides with the
restriction of a smooth $16$-form $\tilde{\Omega}_{\mathbb{O}}$ 
defined on the whole space $\mathcal{J}(3)^{\mathbb{C}}$ and there is a
smooth $11$-form $\mathfrak{n}$ on $\mathcal{J}(3)^{\mathbb{C}}$ with
the property that
\[
\tilde{\Omega}_{\mathbb{O}}\wedge {\mathfrak{n}}=
{da_{1}}{\wedge a_{2}}{\wedge a_{3}}\cdots\cdot{\wedge a_{27}}.
\]
For the description of these smooth forms we need a troublesome
preparation for the coordinates choices and we do not use the forms
later so that we omit the construction. 
\end{remark}

%We always line up in this order of the coordinates 
%and denote a smooth $16$-form by
%\begin{equation}\label{smooth 16 form}
%\sigma_{z_{i}}:=\overline{z_{i}}^{5}dz_{\alpha_1}\wedge\cdots d\xi_{o_{1}}\wedge\xi_{o_{2}}.
%\end{equation}
%Put
%\begin{equation}\label{global smooth form}
%\widetilde{\Omega}_{\mathbb{O}}:=\frac{1}{w(A)}\cdot\sum\,(-1)^{i_{*}}\sigma_{z_{i}}+etc,
%\end{equation}
%where $w(A)$ is defined by
%\[
%w(A)=\sum |z_{i}|^{10}+|w_{i}|^{10}+|y_{i}|^{10}+ \cdots. 
%\]
%Then by \eqref{formula of general Jacobian}
%it will be apparent
%that on each open coordinate neighborhood
%\[
%{\widetilde{\Omega_{\mathbb{O}}}}_{|O_{z_{i}}}=\frac{1}{{z_{i}}^5}dz_{\alpha_{i}}\wedge\cdots\wedge dx_{i_{2}}.
%\]
%Let 
%\[
%\mathfrak{n}=\sum z_{i}^{5}dz_{\gamma_{i}'}\wedge\cdots\wedge d\xi_{o_{1}'}
%\]
%where the coordinates appearing are the complement for defining the
%forms $\sigma_{z_{i}}$, then
%\begin{proposition}
%\[
%\widetilde{\Omega}_{\mathbb{O}}\wedge \mathfrak{n}=dz_{1}\wedge\cdots\wedge d\xi_{3}.
%\]
%\end{proposition}

We mention that
since the transition function of the canonical line bundle on
$\mathbb{X}_{\mathbb{O}}$ is invariant under the multiplication
by non-zero complex numbers, it is a pull-back of a complex line bundle
on the quotient space $\overline{\mathbb{X}}_{\mathbb{O}}:=\mathbb{C}^{*}\backslash\mathbb{X}_{\mathbb{O}}$. More precisely

\begin{proposition}
{\em (1)} Interpreting the  
calculations above in terms of the homogeneous coordinates we see that
the canonical line bundle 
$\mathcal{K}^{\overline{\mathbb{X}}_{\mathbb{O}}}=
\stackrel{15}{\bigwedge}T^{* '}(\overline{\mathbb{X}}_{\mathbb{O}})^{\mathbb{C}}$
of the quotient space $\overline{\mathbb{X}}_{\mathbb{O}}$
is isomorphic to
$\stackrel{5}\otimes{\mathcal{L}^{*}}_{\big|\overline{\mathbb{X}}_{\mathbb{O}}}$,
where
$\mathcal{L}$ is the tautological line bundle on the projective
space $P^{26}\mathbb{C}$, $\mathcal{L}\subset P^{26}\mathbb{C}\times\mathbb{C}^{27}$. 

{\em (2)} Let $\mathcal{V}$ be the kernel of the projection map 
$\pi:\mathbb{X}_{\mathbb{O}}\longrightarrow \overline{\mathbb{X}}_{\mathbb{O}}$,
\[
\mathcal{V}:=\ker\,d\pi\subset T(\mathbb{X}_{\mathbb{O}}),
\]
which can be seen naturally as a complex line bundle trivialized by
the holomorphic vector field corresponding to the dilation action 
\[
\mathbb{X}_{\mathbb{O}}\ni A  \longmapsto  t\cdot A\in\mathbb{X}_{\mathbb{O}}.
\]
In this sense we denote it by $\mathcal{V}_{\mathbb{C}}$. 
Then  by the exact sequence
\[
\{0\}\longrightarrow
\pi^{*}(T^{*'}(\overline{\mathbb{X}}_{\mathbb{O}})^\mathbb{C})
\longrightarrow T^{*'}(\mathbb{X}_{\mathbb{O}})^\mathbb{C}
\longrightarrow 
{\mathcal{V}_{\mathbb{C}}}^{*}\longrightarrow\{0\}
\]
we know that the canonical line bundle
$\mathcal{K}^{\mathbb{X}_{\mathbb{O}}}\cong
\pi^{*}(\mathcal{K}^{\overline{\mathbb{X}}_{\mathbb{O}}})
\otimes{\mathcal{V}_{\mathbb{C}}}^{*}$
is holomorphically trivial, since $\pi^{*}(\mathcal{L})$ is
holomorphically trivial.
\end{proposition}
%}

%%% End of Section 4  %%%%%%%%%%%%%%%%%%%%%%%%%%%%%%%%%%%%%%%%%%%%%%%%%%%%%%%%%%%%%%%%%%%%%%%%%%%
%%% Section 5 %%%%%%%%%%%%%%%%%%%%%%%5

\section{Symplectic manifolds and polarizations}

In this section we review an aspect of a geometric quantization theory
in a restricted framework fitting only to our purpose. 
In the subsections $\S 5. 3$ and $\S 5. 4$ and  
in section $\S 6$ we explain how the framework is adapted to our case.

\subsection{Integral symplectic manifold}\label{integral symplectic manifold}
Let $(M,\omega^{M})$ be a symplectic manifold with the symplectic
form $\omega^{M}$. In this paper we assume that 
\smallskip

\quad [In1] the map $H^{2}(M,\mathbb{Z})\to H^{2}_{dR}(M,\mathbb{R})$ is
injective, or the group $H^{2}(M,\mathbb{Z})$ has no torsion and, 
\smallskip
 
\quad [In2] the de Rham cohomology class $[\omega^{M}]$ of the
symplectic form $\omega^{M}$ is in this image.
\smallskip

Then the complex line bundle $\mathbb{L}_{\omega^{M}}:=\mathbb{L}\in H^{1}(M,\mathbb{C}^{*})\cong
H^{2}(M,\mathbb{Z})$ corresponding to the cohomology class
$[\omega^{M}]$ is unique(of course, up to isomorphism).
The first condition is satisfied, for example if $M$ is simply connected
and our case $M=\mathbb{X}_{\mathbb{O}}$ satisfies both of these
conditions trivially, since $H^{2}(\mathbb{X}_{\mathbb{O}},\mathbb{Z})=\{0\}$.

Under these assumptions, 
the unique complex line bundle $\mathbb{L}$ has the canonically defined
connection $\nabla$, which is defined as follows:

Let $\{U_{i}\}$ be an open covering of $M$ with several ``good'' properties required in the arguments below
(it is always possible for manifolds). 
Then there are one-forms $\{f_{i}\}$, each of
which is defined on $U_{i}$ and $df_{i}=\omega^{M}$.
Then the correction of smooth functions $\{c_{ij}\}$ defined 
by $dc_{ij}=f_{j}-f_{i}$ on
$U_{i}\cap U_{j}$ satisfy that 
$c_{jk}-c_{ik}+c_{ij}$ takes integers on $U_{i}\cap U_{j}\cap U_{k}$, 
and the transition functions $\{g_{ij}=e^{2\pi\sqrt{-1}c_{ij}}\}$ defines 
the line bundle $\pi:\mathbb{L}\to M$.

The connection $\nabla$ on $\mathbb{L}$ is defined as
\[
\nabla_{X}(s_{i})=2\pi\sqrt{-1}<f_{i},X>s_{i}~\text{on $U_{i}$},
~(X~\text{is a vector field})
\]
where $s_{i}$ is a nowhere vanishing section on $U_{i}$ identifying
$U_{i}\times \mathbb{C}$ and $\pi^{-1}(U_{i})\subset \mathbb{L}$ in
such a way that 
\[
U_{i}\times\mathbb{C}\ni(x,z)\mapsto z\cdot s_{i}(x)\in\pi^{-1}(U_{i}).
\]
{\it Here $<f_{i}\,,X>$ denotes the pairing of a one-form $f_{i}$ and a tangent vector $X$.}

%Using the relation $s_{j}=g_{ij}s_{i}$ we can prove
%the well-definedness of this connection.

If we choose all the functions $c_{ij}$ being real valued,
we may regard that the line bundle $\mathbb{L}$ is equipped with
an Hermitian inner product, which we denote by
$<\cdot\,,\,\cdot>^{\mathbb{L}}_{x}$ at $x\in M$. {\it Hereafter we assume
that the line
bundle $\mathbb{L}$ is equipped with such an Hermitian inner product.}

We may regard that the space $C^{\infty}(M)$ is a Lie algebra by the Poisson
bracket $\{f,g\}:=\omega^{M}(H_{f},H_{g})$, where $H_{f}$ denotes the
Hamilton vector field with the Hamiltonian $f$ defined by $<df,\,\bullet> =\omega^{M}(H_{f},\bullet)$.
The space $\Gamma(\mathbb{L},M)$ is a central object in the
quantization theory.
There is a basic fact
that the correspondence from $g\in C^{\infty}(M)$ to the operator $T_{g}$,
assignment of a function to an operator,
\[
T_{g}:\Gamma(\mathbb{L},M)\ni s
\longmapsto \nabla_{H_{g}}(s)+2\pi{\sqrt{-1}}g\cdot s
\]
is a Lie algebra homomorphism, $[T_{g}\,,\,T_{h}]=T_{\{g\,,\,h\}}$,
and it is the main theme in the
quantization theory how to assign 
a function on a phase
space to an operator on the configuration space.

\subsection{Real and complex polarizations}
Let $(M,\omega^{M})$ be a symplectic manifold with the symplectic
form $\omega^{M}$ ($\dim M=2n$). The skew-symmetric bi-linear form
${\omega_{p}^{M}}$
%: T_{p}(M)\times T_{p}(M) \to \mathbb{R}$
at each point $p\in M$ is naturally extended to the complexification 
$T(M)\otimes\mathbb{C}:=T(M)^{\mathbb{C}}$
%${\omega_{p}^{M}}: T_{p}(M)^{\mathbb{C}}\times T_{p}(M)^{\mathbb{C}} \to \mathbb{C}$,
as the skew-symmetric complex bi-linear form which we denote with the same
notation.

Let $F$ be a subbundle of the complex
fiber dimension $n$ 
in $T(M)^{\mathbb{C}}$ satisfying the properties that
\begin{align*}
&(1)~F~\text{is maximal isotropic with respect to the skew-symmetric
  bi-linear form}~\omega^{M},\\
&(2)~F~\text{is integral, that is}~F\cap \overline{F}~\text{has constant
  rank and}~F,~F+\overline{F}~\text{is closed}\\ 
&\text{under bracket operation of vector fields taking values in these subbundles}.\\
\intertext{In this paper 
we only treat two extreme cases,}
&(P1)\qquad F=\overline{F},~\text{and}\\
&(P2)\qquad F+\overline{F}=T(M)^{\mathbb{C}}.
\end{align*}
First one is the complexification of a Lagrangian foliation $L\subset T(M)$,
$F=L\otimes\mathbb{C}$ and
we call it a real polarization. 
The second case is called a complex polarization.

If there is a polarization satisfying the second condition
$F+\overline{F}=T(M)^{\mathbb{C}}$, then $M$ has a
almost complex structure $J$ and 
the subbundle $F$ is identified with $(0,1)$-vectors in $T(M)^{\mathbb{C}}$
(anti-complex subbundle). The integrability condition implies that 
$M$ becomes a complex manifold.
When we put 
\[
g(\alpha,\beta):=\omega^{M}(J(\alpha),\beta),~\alpha,\beta~\text{vector
fields on $M$},
\]
then $g$ is a non-singular symmetric bi-linear form on $T(M)$ and
moreover it defines an Hermitian form on $T(M)^{\mathbb{C}}$.
Under the condition that the form $g$ is positive definite, then it is
equivalent that $M$ has a K\"ahler structure.
%This is equivalent to the condition that
%\[
% -\sqrt{-1}\omega^{M}(\xi,\overline{\xi})\geq 0~\text{for any $\xi\in\Gamma(F)$}.
%\]
We call such a polarization a {\it positive polarization}. 

Hence it is equivalent that if there is a positive complex polarization on the symplectic
manifold $M$, then $M$ is a K\"ahler manifold and the symplectic form $\omega^{M}$ is a
K\"ahler form. Also real polarization is always positive.

{\it In this paper we consider 
two polarizations on the space $\mathbb{X}_{\mathbb{O}}$, 
one is the real
polarization $\mathcal{F}$ naturally defined on the cotangent bundle
%as the kernel of the 
%map ${\bf q}\circ (\tau_{\mathbb{O}})^{-1}:\mathbb{X}_{\mathbb{O}}\to P^2\mathbb{O}$ and 
and a K\"ahler polarization {\em (}= positive complex polarization{\em)} $\mathcal{G}$ described in 
\eqref{condition Cayley cotangent} and Theorem \eqref{Kaehler form}.}

\subsection{Hilbert space structure on the spaces of polarized sections}

Now let $M$ be a symplectic manifold 
satisfying the conditions [In1] and [In2]
as in the subsection $\S$ \ref{integral symplectic manifold} 
and fix a line bundle
$\mathbb{L}$ corresponding to the cohomology
class $[\omega^{M}]$ with  
the connection $\nabla$ and
the Hermitian inner product 
explained in the above subsections
and assume that there is a polarization $F$ on $M$. 

Let $U$ be an open subset in $M$.
We introduce a space $C_{F}(U)\subset C^{\infty}(U)$
by
\[
C_{F}(U)=\{h\in C^{\infty}(U)~|~X(h)=0,~^{\forall} X\in \Gamma{(F,U),
  ~\text{vector fields taking values in $F$}}\}
\]
and a subspace $\Gamma_{F}(\mathbb{L},U)$ of smooth sections in
$\Gamma(\mathbb{L},U)$
by
\[
\Gamma_{F}(\mathbb{L},U)=\{s\in \Gamma(\mathbb{L},U)~|
~\nabla_{X}(s)=0,~^{\forall}X\in \Gamma(F,U)\}.
\]
Let $U$ be an open subset such that there is an one-form
$\theta$ on $U$ satisfying 
\[
d\theta=\omega^{M},~\text{and}~<\theta,X>=0~\text{for vectors $X\in F$}.
\]
Although it is not canonical,
we may locally identify the spaces 
$C_{F}(U)$ and $\Gamma_{F}(\mathbb{L},U)$
by fixing a nowhere vanishing section $s:U\to \mathbb{L}$ with the property
that $\nabla_{X}(s)=0$ for $X\in F$ 
in such a way that
\[
C_{F}(U)\ni \varphi\mapsto \varphi\cdot s \in\Gamma_{F}(\mathbb{L},U).
\]
Then under this identification, the connection $\nabla$ is 
\[
\nabla_{X}(\varphi\cdot s)=X(\varphi)\cdot
s+2\pi\sqrt{-1}\varphi<\theta, X>\cdot s
=X(\varphi)\cdot s,
\]
for vector field $X$ taking values in $F$. 

If $F$ is a real polarization, 
then the function space $C_{F}(U)$ consists of such functions that are
constant along each leaf~$\cap~U$ of the Lagrangian foliation, and if
$F$ is a complex polarization, then $C_{F}(U)$ consists of holomorphic
functions on $U$.

We call these sections $\in \Gamma_{F}(\mathbb{L},\,{U})$ ``{\it polarized sections}'' 
(with respect to a polarization $F$) and are the main objects in the geometric quantization theory.
We may regard, according to the polarization,
that they express
quantum states in the real polarization case
and that they express good classical observables in the complex polarization. 
The above identification indicates the local nature of the polarized
sections according to the polarization.

{\it One basic problem is to introduce an inner product on the space
$\Gamma_{F}(\mathbb{L},M)$ of $\mathbb{L}$-valued polarized sections
and a related space (which will be explained later) in
a reasonable way (or without additional assumptions) to make it a (pre-)Hilbert space
and the most interesting problem is to see a transformation from 
one space of polarized sections $\Gamma_{G}(\mathbb{L},M)$ (by a polarization $G$)
to another space $\Gamma_{F}(\mathbb{L},M)$ of polarized sections by
another polarization $F$.}
 
We discuss two cases according to the polarizations (real and positive complex) 
how we introduce an inner product below in [RP] (real polarization) and
in [CP] (complex polarization). 

Under our assumptions we work only on density, (partial) half density, or
(partial)$1/4$-density spaces. The meaning of ``partial'' will be
explained in Remark \ref{meaning of partial density}.
\smallskip

[RP]~\quad  Let $F$ be a {\it real polarization}. In this paper, for avoiding unnecessary
generality, we  assume more strongly that 

\qquad (RP1) there is a submersion to an orientable manifold $N$, 
\[
\Phi:M\longrightarrow N
\]
whose fibers are connected Lagrangian
submanifolds.

So, the real polarization $F$ is defined as the kernel
$F=\text{Ker}\, d\Phi$
of a surjective submersion $\Phi:M\longrightarrow N$ and the functions in
$C_{F}(M)$
are naturally descended to the base space $N$, that is $\Phi^{*}(C^{\infty}(N))=C_{F}(M)$.

Let $\alpha,~\beta\in \Gamma_{F}(\mathbb{L},M)$, then
by the equality
\[
0=<\nabla_{X}(\alpha),\beta>^{\mathbb{L}}+<\alpha,\nabla_{X}(\beta)>^{\mathbb{L}}=X(<\alpha,\beta>^{\mathbb{L}}),
~\text{for $X\in F$},
\]
the function $<\,\alpha\,,\,\beta\,>^{\mathbb{L}}$ is constant on each fiber.
Hence it can be naturally identified with a function on
the base manifold $N$. For such functions we need not integrate along the leaves and 
it will be enough to consider the integration to the transversal
direction of the leaves. This is realized by the integration on the base space
to make the space $\Gamma_{F}(\mathbb{L},M)$ into a (pre) Hilbert space. 
There are many way to choose a measure on $N$, say a Riemann volume
form to integrate it.

Instead of the space $\Gamma_{F}(\mathbb{L},M)$, we consider 
{\it $\mathbb{L}$-valued polarized 
{\em (or exactly to say, we call horizontal and partial)} half-densities $\varphi\in
\Gamma_{F}\Big(\mathbb{L}\otimes\big|\stackrel{max}\bigwedge F^{0}\big|,M\Big)$  
and/or {\em horizontal and partial} $\frac{1}{4}$-densities} 
~$\varphi\in \Gamma_{F}\Big (\mathbb{L}\otimes\sqrt{\big|\stackrel{max}\bigwedge F^{0}\big|},M\Big)$,
where $F^{0}$ is the annihilator of $F$,
\[
F^{0}=\{\xi\in T^{*}(M)~|~<\xi,X>=0,~^{\forall}X\in F~\}.
\]

We can introduce a (partial) connection  
$\nabla{\hspace{-0.24cm}\slash}_{X}(\xi):=i_{X}(d\xi)$ 
on $\stackrel{max}\bigwedge F^{0}=\stackrel{\max}\bigwedge(d\Phi)^{*}(\Phi^{*}(T^{*}(N)))$, 
where $\xi$ is a differential form
$\in\Gamma\Big(\stackrel{max}\bigwedge F^{0}, M\Big)$, 
$X\in F$ and $i_{X}$ denotes the interior product with a tangent vector
$X\in F$. 

Note that
\[
i_{X}(d\xi)=i_{X}\,\circ\, d\xi\in 
\Gamma\Big(\stackrel{max}\bigwedge F^{0},M\Big), \quad\text{for $X\in F$ and
  $\xi\in\Gamma\Big(\stackrel{max}\bigwedge F^{0},M\Big)$}.
\]
Since $i_{X}(\xi)=0$ 
for $\xi\in \Gamma\Big(\stackrel{max}\bigwedge F^{0},M\Big)$ by $X\in F$  
\begin{align*}
\nabla\hspace{-0.24cm}\slash_{X}(f\cdot\xi)
&=i_{X}\circ d(f\cdot\xi)=i_{X}\circ (df\wedge \xi+f\cdot d\xi)=X(f)\cdot \xi- df\wedge i_{X}(\xi)+f\cdot i_{X}(d\xi)\\
&=X(f)\xi+f\cdot \nabla\hspace{-0.24cm}\slash_{X}(f\cdot\xi),~\text{for
  $X\in F$ and $f\in C^{\infty}(M)$},
\end{align*}
the vector fields taking values in $F$ work as a differentiation 
on the space of the differential forms $\Gamma\Big(\stackrel{max}\bigwedge F^{0},M\Big)$. 
Hence we can consider the differentiation
$\nabla\hspace{-0.23cm}\slash_{X}$ along the
polarization $F$ for the sections
$\in~\Gamma\Big(\stackrel{max}\bigwedge F^{0},M\Big)$ and also
sections $\in~\Gamma\Big(\big|\stackrel{max}\bigwedge F^{0}\big|,M\Big)$ too.

Then under our assumption (RP1) % $\sim$ (RP3) 
and according to the definition of the partial connection,
the sections $\xi\in \Gamma_{F}\Big(\stackrel{max}\bigwedge F^{0},M\Big)$ can
be descended to the sections $\in \Gamma(\stackrel{max}\bigwedge T^*(N),N)$, 
hence it holds
\begin{equation}\label{descended measure}
\Phi^{*}\left(\Gamma\Big(\stackrel{max}\bigwedge T^{*}(N),N\Big)\right)\cong
\Gamma_{F}\Big(\stackrel{max}\bigwedge F^{0},M\Big).
\end{equation}
We may regard a differential form in
$\Gamma_{F}\Big(\stackrel{max}\bigwedge F^{0},M\Big)$  
a {polarized (or horizontal) ``partial'' half density} (or half degree form) on $M$.

\begin{remark}\label{meaning of partial density}
By our assumption {\em (RP1)}, % $\sim$ {\em (RP3)}, 
there is an exact sequence
\begin{equation}\label{exact sequence}
\{0\}\longrightarrow F^{0}\longrightarrow T^{*}(M)\longrightarrow
F^{*}\longrightarrow\{0\},
\end{equation}
and the injective bundle map on $M$,~$(d\Phi)^{*}:\Phi^{*}(T^{*}(N))\to T^{*}(M)$, ~which is
the dual of the differential $d\Phi$. 
Since the polarization $F$ coincides with the vertical subbundle of
the projection map 
$\Phi$, the image $(d\Phi)^{*}(\Phi^{*}(T^{*}(N)))=F^{0}$.

By the assumption {\em (RP1)} we regard that
$\stackrel{max}\bigwedge\hspace{-0.18cm}T^{*}(N)\cong |\stackrel{max}\bigwedge T^{*}(N)|$ 
{\em(line bundles of the highest degree differential form and density
(volume form) line
bundle)} and we consider the square root bundle
$\sqrt{\stackrel{max}\bigwedge F^{0}}$.

Sections in 
$\Gamma_{F}({\stackrel{max}\bigwedge F^{0}},M)$
or $\Gamma_{F}(\sqrt{\stackrel{max}\bigwedge F^{0}},M)$ 
are not the half densities or $1/4$-densities,
since $\stackrel{max}\bigwedge\hspace{-0.2cm}T^{*}(M)\cong\stackrel{max}\bigwedge\hspace{-0.2cm} F^{*}
\otimes\hspace{-0.18cm}\stackrel{max}\bigwedge\hspace{-0.15cm} F^{0}=\text{trivial bundle given by the Liouville volume form}$.
So we should call the sections in $\Gamma_{F}(\stackrel{max}\bigwedge\hspace{-0.2cm}F^{0},M)$  
or in $\Gamma_{F}(\sqrt{\stackrel{max}\bigwedge\hspace{-0.2cm} F^{0}},M)$ 
polarized ``partial'' half density or ``partial'' $1/4$-density.
\end{remark}

Differential forms $\mu\in \Gamma_{F}\Big(\stackrel{max}\bigwedge
\hspace{-0.2cm}F^{0},M\Big)$ is descended
to densities ${\mu}_{*}\in \Gamma(\stackrel{max}\bigwedge\hspace{-0.2cm}
 T^{*}(N),N)$ (highest degree differential form) on the
base manifold $N$, that is there is a unique highest degree
differential form 
$\mu_{*}\in \Gamma(\stackrel{max}\bigwedge\hspace{-0.2cm}T^{*}(N),N)$
such that $\Phi^{*}({\mu}_{*})=\mu$ by the isomorphism \eqref{descended measure}, 
and then we can integrate $\mu_{*}$ on $N$. Hence we have a natural linear form
\[
{\text I}_{N}:\Gamma_{F}\Big(\,\stackrel{max}\bigwedge\hspace{-0.2cm}F^{0},M\,\Big)\ni\mu \longmapsto 
{\text I}_{N}(\mu):=
\int_{N}\,{\mu}_{*}\in \mathbb{C}.
\]
If we denote the inverse map of $\Phi^{*}$ of
\eqref{descended measure} by $\Phi_{*}$, then
\[
\int_{N}\,{\mu}_{*}=\int_{N}\Phi_{*}(\mu).
\] 

In turn, we consider the square root bundle 
$\sqrt{\stackrel{max}\bigwedge\hspace{-0.18cm}F^{0}}$, which can be seen as a partial $1/4$-density bundle on $M$. 
Then we can also introduce a partial connection  
${\nabla\hspace{-0.3cm}\slash_{X}}^{1/2}$ on the line bundle
$\sqrt{\stackrel{max}\bigwedge\hspace{-0.18cm}F^{0}}$ and as well it is defined also on the line
bundle $\mathbb{L}\otimes \sqrt{\stackrel{max}\bigwedge\hspace{-0.18cm}F^{0}}$. 
Hence we consider 
{\it ``\,$\mathbb{L}$-valued polarized {\em(}or horizontal{\em)} partial
  $\frac{1}{4}$-densities\,''} 
$\alpha\otimes{\eta}\in \Gamma_{F}\Big(\mathbb{L}\otimes \sqrt{\stackrel{max}\bigwedge\hspace{-0.18cm}F^{0}},M\,\Big)$
and define their product by making use of the Hermitian inner
product on $\mathbb{L}$ with the formula
\begin{align}
\Gamma_{F}\Big(\mathbb{L}\otimes \sqrt{\stackrel{max}\bigwedge\hspace{-0.18cm}F^{0}},M\,\Big)
&\times \Gamma_{F}\Big(\mathbb{L}\otimes \sqrt{\stackrel{max}\bigwedge\hspace{-0.18cm}F^{0}},M\,\Big) 
\longrightarrow \Gamma_{F}\Big(\stackrel{max}\bigwedge\hspace{-0.18cm}F^{0},M\,\Big)\notag\\
&{\text{\rotatebox[origin=c]{90}{$\in$}}}
{\hspace{6.5cm}{\text{\rotatebox[origin=c]{90}{$\in$}}}}\notag\\
(\alpha\otimes{\mu},\,&\beta\otimes{\nu})\qquad\longmapsto\qquad
<\,\alpha\,,\,\beta\,>^{\mathbb{L}}\cdot\,\mu\otimes \nu\in\Gamma_{F}
\Big(\stackrel{max}\bigwedge\hspace{-0.18cm}F^{0},M\,\Big)
\label{product of 1/4 sections}.
\end{align}
The resulting horizontal partial
half density $<\,\alpha\,,\,\beta\,>^{\mathbb{L}}\cdot\mu\otimes
\nu\in\Gamma_{F}(\stackrel{max}\bigwedge\hspace{-0.18cm}F^{0},M)$, 
is identified with a density on $N$.
Hence we can define a pairing (or an inner product) for the sections in
$\Gamma_{F}\Big(\mathbb{L}\otimes \sqrt{\stackrel{max}\bigwedge\hspace{-0.18cm}F^{0}},M\,\Big)$
by the integration of the corresponding density on $N$ in a natural
way,
\begin{align*}
\Gamma_{F}\Big(\mathbb{L}\otimes &\sqrt{\stackrel{max}\bigwedge\hspace{-0.18cm}F^{0}},M\,\Big)
\times \Gamma_{F}\Big(\mathbb{L}\otimes \sqrt{\stackrel{max}\bigwedge\hspace{-0.18cm}F^{0}},M\,\Big)\longrightarrow \mathbb{C},\\
&{\text{\rotatebox[origin=c]{90}{$\in$}}}           \\
(a\otimes{\mu},&\,b\otimes{\nu})\longmapsto 
{\text I}_{N}(\Phi_{*}(<a,b>^{\mathbb{L}}\cdot \mu\otimes \nu))
=\int_{N}\,\Phi_{*}(<a,b>^{\mathbb{L}}\cdot \mu\otimes\nu).
\end{align*}

{\it
For the real polarization $\mathcal{F}$ on our space ~$\mathbb{X}_{\mathbb{O}}$, first 
we trivialize the line bundle $\mathbb{L}$ by a nowhere vanishing polarized 
section ${\bf s}_{0}\in \Gamma_{\mathcal{F}}(\mathbb{L},\mathbb{X}_{\mathbb{O}})$
with $<{\bf s}_{0},{\bf s}_{0}>^{\mathbb{L}}\equiv 1$.
We call this trivialization of the line bundle $\mathbb{L}$ a ''unitary trivialization". 

Next,  let
$dv_{P^2\mathbb{O}}$ be the Riemann volume form on $P^{2}\mathbb{O}$.
We consider the   square root 
\[
\sqrt{\{{\bf q}\circ(\tau_{\mathbb{O}})^{-1}\}^{*}(dv_{P^2\mathbb{O}})}
={\{{\bf q}\circ(\tau_{\mathbb{O}})^{-1}\}^{*}(\sqrt{dv_{P^2\mathbb{O}}})}
\in \Gamma_{\mathcal{F}}\Big(\sqrt{\stackrel{max}\bigwedge \mathcal{F}^{0}},\mathbb{X}_{\mathbb{O}}\,\Big),
\]
and identify a $\mathbb{L}$-valued polarized partial 
${1}/{4}$-density $\xi\otimes\mu\in
\Gamma_{\mathcal{F}}\Big(\mathbb{L}\otimes\sqrt{\stackrel{max}\bigwedge\hspace{-0.18cm}
\mathcal{F}^{0}},\mathbb{X}_{\mathbb{O}}\Big)$ 
with
$f\cdot {\bf s}_{0}$ $\otimes$
$\sqrt{\{{\bf q}\circ(\tau_{\mathbb{O}})^{-1}\}^{*}({dv_{P^2{\mathbb{O}}}})}$, where 
the function $f$ can been as a pull-back of a function $g\in C^{\infty}(P^{2}\mathbb{O})$, $f={\bf q}^{*}(g)$.
Then we  may identify it with a half density on $N$ of the form $g\cdot
\sqrt{dv_{P^2\mathbb{O}}}$.
Hence we identify the $L_{2}$-space 
with respect to the Riemann volume
form $dv_{P^2\mathbb{O}}$ {\em(we denote it by $L_{2}(P^2\mathbb{O},dv_{P^2\mathbb{O}})$)} 
and the space of ~$~\mathbb{L}$-valued polarized partial $1/4$-densities
$\Gamma_{\mathcal{F}}\Big(\mathbb{L}\otimes\sqrt{\stackrel{max}\bigwedge}\mathcal{F}^{0},~\mathbb{X}_{\mathbb{O}}\Big)$
{\em(}after taking completion{\em)}}.
\bigskip

[CP]~\quad Let $G$ be a positive complex polarization on $M$ whose  
symplectic form $\omega^{M}$ is expressed as a K\"ahler form:
\[
{\sqrt{-1}\,\,\overline{\partial}\,\partial \phi}=\omega^{M}.
\]
The line bundle $\mathbb{L}$ corresponding to the cohomology class $[\omega^{M}]$
is equipped with an Hermitian inner product $<\cdot,\cdot>^{\mathbb{L}}$ 
as was explained in \ref{integral symplectic manifold}. 

The inner product $<a,b>^{\mathbb{L}}$ of two sections $a,b\in \Gamma_{G}(\mathbb{L},M)$
is a function on $M$ and can be integrated with respect to the
Liouville volume form
$dV_{M}:=\frac{(-1)^{n(n-1)/2}}{n!}\left\{\omega^{M}\right\}^{n}$
($\dim M=2n$).
Hence we can introduce an inner product on the space
$\Gamma_{G}(\mathbb{L},M)$ intrinsically, since we do not
depend on any other additional assumptions. 

We can also introduce an inner product on the space of
$\mathbb{L}$-valued ``polarized'' sections of the canonical line bundle $K^{G}$ for
the complex polarization $G$.

The canonical line bundle $K^{G}=\stackrel{max}\bigwedge\hspace{-0.15cm}
T^{*'}(M)^{\mathbb{C}}$ is the line bundle of the highest degree exterior product of the 
holomorphic part $T^{*'}(M)^{\mathbb{C}}$ of the complexified cotangent
bundle $T^{*}(M)^{\mathbb{C}}$ (\,$(1,0)$ type
cotangent vectors\,), which is the annihilator of the complex
polarization $G$ \,(\,(0,1) tangent vectors), like $F^{0}$ for the real polarization $F$. 
The sections of the canonical line bundle can be thought as half
densities (or complex valued half density) by the isomorphism $K^{G}\otimes \overline{K^{G}}=\stackrel{max}\bigwedge\hspace{-0.15cm}T^{*}(M)^{\mathbb{C}}$. We can introduce a partial connection
${\nabla\hspace{-0.24cm}\slash_{X}}^{G}$ $(X\in G)$ along the complex 
polarization $G$ in the similar way as for the real polarization.
Then we consider the space $\Gamma_{G}(\mathbb{L}\otimes K^{G},M)$ of 
``{\it $\mathbb{L}$-{valued} {\it polarized sections of} {\it the canonical line bundle}'' }
and using the
Hermitian inner product on $\mathbb{L}$
we have a highest degree differential form
\[
<a\otimes \mu,b\otimes \nu>=<a,b>^{\mathbb{L}}\cdot \,\mu\wedge \,\overline{\nu}\in
\Gamma\Big(\stackrel{max}\bigwedge\hspace{-0.15cm}T^{*}(M)^{\mathbb{C}},M\,\Big),
\]
where $a,b\in\Gamma_{G}(\mathbb{L},M)$ and $\mu,\nu\in\Gamma_{G}(K^{G},M)$. 
The quantity $\mu\wedge \overline{\nu}$ %$\in \Gamma(\stackrel{max}\bigwedge T^{*}(M)\otimes\mathbb{C},M)$ 
can be seen as a (complex valued) density on $M$.
Hence we have an intrinsic (pre-)Hilbert space structure on the space
$\Gamma_{G}(\mathbb{L}\otimes K^{G},M)$. 

{\it For the complex polarization $\mathcal{G}$ on our space ~$\mathbb{X}_{\mathbb{O}}$, 
we use a structure so called Calabi-Yau structure on
$\mathbb{X}_{\mathbb{O}}$ to identify 
the space $\Gamma_{\mathcal{G}}(\mathbb{L}\otimes K^{\mathcal{G}},\mathbb{X}_{\mathbb{O}})$
and the space $C_{\mathcal{G}}(\mathbb{X}_{\mathbb{O}})$ of holomorphic
functions on $\mathbb{X}_{\mathbb{O}}$
by the correspondence
\begin{equation}\label{holomorphic function to L-valued 16 form}
\gamma:C_{\mathcal{G}}(\mathbb{X}_{\mathbb{O}})\ni h\longmapsto  \gamma(h)=h\cdot {\bf t}_{0}\otimes 
\Omega_{\mathbb{O}}\in \Gamma_{\mathcal{G}}(\mathbb{L}\otimes K^{\mathcal{G}},\mathbb{X}_{\mathbb{O}}).
\end{equation}
The existence of the nowhere vanishing holomorphic $16$-form $\Omega_{\mathbb{O}}$
on $\mathbb{X}_{\mathbb{O}}$ was proved in
Proposition \eqref{holomorphic global section}
and ${\bf t}_{0}$ is taken 
for trivializing the line bundle $\mathbb{L}$ satisfying the property
${\nabla{\hspace{-0.24cm}}{/}_{X}}^{\mathcal{G}}({\bf t}_{0})=0$.
%, where the partial covariant
%differential $\nabla{\hspace{-0.23cm}}{/}_{X}$
%is defined  for the complex polarization $\mathcal{G}$.

We call a trivialization of the line bundle $\mathbb{L}$ 
by the section ${\bf t}_{0}$ a "holomorphic trivialization". We will
determine the relation of the sections ${\bf s}_{0}$ and ${\bf
  t}_{0}$, ${\bf t}_{0}=g_{0}{\bf s}_{0}$
in the subsection 
$\S$ \ref{holomorphic and unitary trivialization}.
}

\subsection{Pairing of polarizations and a Bargmann type
  transformation}\label{pairing of polarizations}

First, we recall the fiber integration. 
Let $\phi:M\to N$ be a differentiable map between two manifolds. 

Let $\sigma\in\Gamma\Big(\stackrel{p}\bigwedge\hspace{-0.1cm}T^{*}(M), M\Big)$ be a
differential form with the degree $p\geq \dim M-\dim N:=d$.
For $g\in\Gamma\Big(\stackrel{q}\bigwedge\hspace{-0.1cm}T^{*}(N),N\Big)$ with compact
support satisfying $q=m-p=\dim M-p\geq 0$ (we denote the space of
sections with compact support by $\Gamma_{0}(\,*\,,\,*\,)$). 
We assume
\[
\int_{M}\,|\sigma\wedge \phi^{*}(g)|<\,+\infty
\] 
for any $g\in \Gamma_{0}\Big(\stackrel{q}\bigwedge\hspace{-0.1cm}T^{*}(N)\Big)$   
and define a linear functional
\[
g\longmapsto \int_{M}\,\sigma\wedge \phi^{*}(g),
\]
which is understood as a distribution on the space
$\Gamma_{0}\Big(\stackrel{q}\bigwedge\hspace{-0.1cm}T^{*}(N)\Big)$. {\it We denote this
distribution
by $\phi_{*}(\sigma)$ and express  as
\begin{equation}\label{fiber integral}
\phi_{*}(\sigma)(g):=\int_{N}\phi_{*}(\sigma)\wedge g =\int_{M}\,\sigma\wedge \phi^{*}(g).
\end{equation}}
If $\phi$ is a submersion, then
$\phi_{*}(\sigma)$ is a smooth differential form of degree $p-d$.

In the last subsection we introduced inner products on the
spaces
$\Gamma_{F}\left(\mathbb{L}\otimes\sqrt{\stackrel{max}\bigwedge\hspace{-0.15cm}F^{0}},M\right)$
and $\Gamma_{G}$
$\left(\mathbb{L}\otimes {\hspace{-0.2cm}\stackrel{max}\bigwedge {\hspace{-0.15cm}T^{*'}(M)^{\mathbb{C}}}},M\right)
=\Gamma_{G}\left(\mathbb{L}\otimes K^{G},M\right)$
for a real polarization $F$ satisfying the condition (RP1) %$\sim$ (RP3) 
and a positive complex polarization $G$ on an integral symplectic manifold $M$.
Our main purpose is to construct a transformation
\begin{equation}\label{Bargmann transformation}
\mathfrak{B}:\Gamma_{G}\left(\mathbb{L}\otimes K^{G},M\right)\longrightarrow
\Gamma_{F}\Big(\mathbb{L}\otimes\sqrt{\stackrel{max}\bigwedge F^{0}},M\Big)
\end{equation}
or it may be understood as a sesqui-linear form on
\begin{align}
&\Gamma_{F}\Big(\mathbb{L}\otimes\sqrt{\stackrel{max}\bigwedge F^{0}},M\Big)\times\Gamma_{G}\Big(\mathbb{L}\otimes K^{G},M\Big)\stackrel{Id\times\mathfrak{B}}\longrightarrow \\
&\qquad\qquad\qquad\qquad
\Gamma_{F}\Big(\mathbb{L}\otimes\sqrt{\stackrel{max}\bigwedge  F^{0}},M\Big)\times
\Gamma_{F}\Big(\mathbb{L}\otimes\sqrt{\stackrel{max}\bigwedge F^{0}},M\Big)\longrightarrow 
\mathbb{C}.\notag
\end{align}
For the sections 
$(\alpha\otimes \mu\,,\,\beta\otimes\nu) \in
\Gamma_{F}\Big(\mathbb{L}\otimes\sqrt{\stackrel{max}\bigwedge\hspace{-0.1cm} F^{0}},M\Big)\times
\Gamma_{G}\Big(\mathbb{L}\otimes K^{G},\,M\Big)
$
their product
\[
<\alpha,\beta>^{\mathbb{L}}\cdot\, |\mu\otimes \nu|
\]
($|\,*\,|$ means a section of $\big|K^{G}\otimes \sqrt{\stackrel{max}\bigwedge}\hspace{-0.1cm}F^{0}\big|$)
is understood as a partial $3/4$-density on $M$ and so we need some
modification to integrate it, since there are no manifold of the
dimension $3/4\times \dim M$.

Since we identify the half density space
$\Gamma\Big(\sqrt{\stackrel{max}\bigwedge T^{*}(N)},N\Big)$ 
with a $L_{2}$-space by fixing a Riemann volume form $dv_{N}$, 
we define a sequi-linear form 
$$\Gamma_{F}\Big(\mathbb{L}\otimes\sqrt{\stackrel{max}\bigwedge F^{0}},M\Big)
\times\Gamma_{G}\Big(\mathbb{L}\otimes K^{G},M\Big)\longrightarrow\mathbb{C}
$$
by 
\begin{align}
\Gamma_{F}\Big(\mathbb{L}\otimes\sqrt{\stackrel{max}\bigwedge
  F^{0}},M\Big)
&\times \Gamma_{G}\Big(\mathbb{L}\otimes K^{G},M\Big)
\ni
(\alpha\otimes \mu\,,\,\beta\otimes\nu)\label{basic sesquiliner form} \\
&\longmapsto
\int_{M}\,<\alpha,\beta>^{\mathbb{L}}\cdot \,\Phi^{*}(f_{\mu}dv_{N})\wedge \overline{\nu},\notag
\end{align}
where we can put $\mu =\Phi^{*}(f_{\mu})\sqrt{\Phi^{*}(dv_{N})}$ with 
a function $f_{\mu}\in C^{\infty}(N)$, that is
we multiply the partial $1/4$-density $\sqrt{\Phi^{*}(dv_{N})}$  
to the partial $1/4$-density $\nu=\Phi^{*}(f_{\nu})\sqrt{\Phi^{*}(dv_{N})}$, then  
$\sqrt{\Phi^{*}(dv_{N})}\otimes \sqrt{\Phi^{*}(dv_{N})}\otimes {\mu}$
is a (complex valued) highest degree differential form 
(or can be thought as a density) on $M$ and we can define a sesqui-linear form. 

Once we have a sesqui-linear form
\[
P:\Gamma_{F}\Big(\mathbb{L}\otimes\sqrt{\stackrel{max}\bigwedge F^{0}},M\Big)\times\Gamma_{G}\Big(\mathbb{L}\otimes K^{G},M\Big)\longrightarrow \mathbb{C}
\]
it is rewritten as
\[
P(\alpha\otimes \mu,\beta\otimes\nu)=
\sum\,{\text{I}_{N}(\Phi_{*}<\alpha,\alpha_{i}>^{\mathbb{L}}\cdot \mu\otimes\mu_{i})},
\]
where we put $\mathfrak{B}(\beta\otimes\nu)=\sum\,\alpha_{i}\otimes\mu_{i}$.
\smallskip

{\it 
In our space $T^{*}_{0}(P^2\mathbb{O})\stackrel{\tau_{\mathbb{O}}}\cong\mathbb{X}_{\mathbb{O}}$ we have two
polarizations $\mathcal{F}$ {\em(real)} and $\mathcal{G}$ {\em(}K\"ahler{\em)} and 
we identify the spaces $C_{\mathcal{G}}(\mathbb{X}_{\mathbb{O}})$
and $\Gamma_{\mathcal{G}}(\mathbb{L}\otimes K^{\mathcal{G}},\mathbb{X}_{\mathbb{O}})$
by \eqref{holomorphic function to L-valued 16 form}.
The inner product on the space
$C_{\mathcal{G}}(\mathbb{X}_{\mathbb{O}})$ induced by the map $\gamma$ will be explicitly
described in \eqref{parameter family of inner products} 
%{inner product on Fock space in terms of Liouville volume form} 
 at the end of $\S$ \ref{Fock space}
in terms of the Liouville volume form. There we will introduce a parameter family of
inner products on the space $\Gamma_{\mathcal{G}}(\mathbb{L}\otimes K^{\mathcal{G}},\mathbb{X}_{\mathbb{O}})$.

We recall the sections ${\bf s}_{0}$ and ${\bf t}_{0}$ and describe
our Bargmann type transformation including the 
quantity $<{\bf s}_{0},{\bf t}_{0}>^{\mathbb{L}}$. 

Let $\theta^{P^2\mathbb{O}}$ be the canonical one-form on the
cotangent bundle $T^{*}(P^2\mathbb{O})$,
then
$d\theta^{P^2\mathbb{O}}=\omega^{P^2\mathbb{O}}$ and for any 
$X\in\mathcal{F}$,~ $<\theta^{P^2\mathbb{O}},X>=0$.
So let ${\bf s}_{0}$ be a nowhere vanishing polarized {\em(}with respect to
the real polarization $\mathcal{F}${\em)} global section of
$\mathbb{L}$ defining
a trivialization $\mathbb{X}_{\mathbb{O}}\times \mathbb{C}\cong \mathbb{L}$
by the correspondence
\begin{equation}\label{trivialzation:real case}
\mathbb{X}_{\mathbb{O}}\times\mathbb{C}\ni(A,z)\longleftrightarrow z\cdot {\bf s}_{0}(A)\in \mathbb{L},
\end{equation}
with $<{\bf s}_{0}\,,\,{\bf s}_{0}>^{\mathbb{L}}\equiv 1$. 

Also by the relation 
\[
{\tau_{\mathbb{O}}}^{*}\left(\sqrt{-2}\,\,\overline{\partial}\partial \,||A||^{1/2}\right)={\omega^{P^2\mathbb{O}}},
\]
given in Theorem \eqref{Kaehler form}, we take a {\em(}complex{\em)}one-form
\[
\theta_{\mathcal{G}}=\sqrt{-2}\,\partial \,||A||^{1/2},
\]
then $d{\tau_{\mathbb{O}}}^{*}(\theta_{\mathcal{G}})=\omega^{P^2\mathbb{O}}$ and
$\theta_{\mathcal{G}}(X)=0$ for $X\in\mathcal{G}$, since $X$ is
a $(0,1)$ tangent vector.

Then we can trivialize the line bundle $\mathbb{L}$ by making use of a nowhere
vanishing
global section ${\bf t}_{0}$ in a similar way to \eqref{trivialzation:real case}.

Using the identifications \eqref{holomorphic function to L-valued 16 form} %.{Identification by Calabi-Yau}
and the correspondence
\begin{align*}
C^{\infty}(P^2\mathbb{O})\ni g\,\longmapsto \,&\{{\bf q}\circ
{\tau_{\mathbb{O}}}^{-1}\}^{*}(g)\cdot {\bf s}_{0}
\otimes\sqrt{\{{\bf q}\circ {\tau_{\mathbb{O}}}^{-1}\}^{*}(dv_{P^2\mathbb{O}})}\\
&\longmapsto \{{\bf q}\circ {\tau_{\mathbb{O}}}^{-1}\}^{*}(g)\cdot {\bf s}_{0}
\otimes{\{{\bf q}\circ {\tau_{\mathbb{O}}}^{-1}\}^{*}(dv_{P^2\mathbb{O}})}
\end{align*}
the integral \eqref{basic sesquiliner form} is rewritten as
\begin{equation}
\int_{\mathbb{X}_{\mathbb{O}}}\,\{{\bf q}\circ {\tau_{\mathbb{O}}}^{-1}\}^{*}(g)\cdot
\overline{h}
\cdot <{\bf s}_{0},{\bf t}_{0}>^{\mathbb{L}}
\cdot \{{\bf q}\circ {\tau_{\mathbb{O}}}^{-1}\}^{*}(dv_{P^2\mathbb{O}})\wedge \overline{\Omega_{\mathbb{O}}},
\end{equation}
and it is also expressed in terms of the fiber integration as follows:
\begin{align}
&\int_{\mathbb{X}_{\mathbb{O}}}\,\{{\bf q}\circ {\tau_{\mathbb{O}}}^{-1}\}^{*}(g)
\cdot \overline{h}
\cdot <{\bf s}_{0},{\bf t}_{0}>^{\mathbb{L}}
\cdot \{{\bf q}\circ {\tau_{\mathbb{O}}}^{-1}\}^{*}(dv_{P^2\mathbb{O}})
\wedge 
\overline{\Omega_{\mathbb{O}}}\notag\\
&=\int_{P^2\mathbb{O}}\,g\cdot \{{\bf q}\circ {\tau_{\mathbb{O}}}^{-1}\}_{*}(<{\bf s}_{0},{\bf
  t}_{0}>^{\mathbb{L}}\cdot
  \overline{h}\cdot\overline{\Omega_{\mathbb{O}}})\,dv_{P^2\mathbb{O}}.\label{fiber integration form}
\end{align}
Then the Bargmann type transformation
\begin{align}
&\mathfrak{B}:C_{\mathcal{G}}(\mathbb{X}_{\mathbb{O}})\to
C^{\infty}(P^2\mathbb{O}),~~C_{\mathcal{G}}\ni h\longmapsto \mathfrak{B}(h),\notag
\intertext{is defined as}
&\mathfrak{B}(h)
=\{{\bf q}\circ{\tau_{\mathbb{O}}}^{-1}\}_{*}(h\cdot<{\bf t}_{0},{\bf
  s}_{{0}}>^{\mathbb{L}}\Omega_{\mathbb{O}}).\label{Bargmann type transform}
\end{align}
Hence we can express the integral  \eqref{fiber integration form} as 
\[
\int_{P^2\mathbb{O}}\,g\cdot \overline{\mathfrak{B}(h)}\,dv_{P^2\mathbb{O}}.
\]
\begin{remark}\label{ambiguity by constant}
The section ${\bf s}_{0}$ is free of $U(1)$-multiple and ${\bf t}_{0}$
is of free from a constant $\in\mathbb{C}^{*}$.
\end{remark}
}

\section{Bargmann type transformation}
For expressing the Bargmann type transformation explicitly and
to determine its $L_{2}$ continuity, 
we need to know the function 
$<{\bf s}_{0},{\bf t}_{0}>^{\mathbb{L}}=g_{0}$,
and relations of
$\Omega_{\mathbb{O}}\wedge\overline{\Omega_{\mathbb{O}}}$ 
and $\{{\bf q}\circ {\tau_{\mathbb{O}}}^{-1}\}^{*}(dv_{P^2\mathbb{O}})\wedge\overline{\Omega_{\mathbb{O}}}$
with the Liouville volume form $dV_{T^{*}(P^2\mathbb{O})}$ explicitly. In this section we determine them.

\subsection{Holomorphic trivialization and unitary trivialization}
\label{holomorphic and unitary trivialization}

The relation of the sections ${\bf s}_{0}$ and ${\bf t}_{0}$ is given by
a function $g_{0}=<{\bf s}_{0},{\bf t}_{0}>^{\mathbb{L}}$, that is
\begin{equation}\label{basic relation of two trivialization by polarized sections}
{\bf t}_{0}=g_{0}\cdot{\bf s}_{0}.
\end{equation}
The function $g_{0}$ satisfies an equation
\begin{align*}
&\nabla_{X}({\bf t}_{0})=2\pi{\sqrt{-1}}<\sqrt{-2}\,\partial\,||A||^{1/2},X>g_{0}\cdot {\bf s}_{0}\\
&\qquad\qquad=\nabla_{X}(g_{0}{\bf s}_{0})=X(g_{0}){\bf s}_{0}+2\pi{\sqrt{-1}}g_{0}\cdot
<\theta^{P^2\mathbb{O}},X>{\bf s}_{0},
\end{align*}
%since
%$\mathcal{F}+\overline{\mathcal{G}}=\mathcal{F}+{\mathcal{G}}=T(\mathbb{X}_{\mathbb{O}})^{\mathbb{C}}$,
and we have an equation for the function $g_{0}$:
\begin{equation}\label{relation between s0 and t0}
2\pi \sqrt{-1}\cdot\left({\tau_{\mathbb{O}}}^{*}\left(\sqrt{-1}\sqrt{2}\,\partial ||A||^{1/2}\right)- \theta^{P^2\mathbb{O}}\right)g_{0}=dg_{0}.
\end{equation}
Put $g_{0}=e^{2\pi\sqrt{-1}\lambda}$, then
the equation \eqref{relation between s0 and t0} reduces to the equation
\begin{equation}\label{basic equation}
d\lambda={\tau_{\mathbb{O}}}^{*}\left(\sqrt{2}\sqrt{-1}\partial\,||A||^{1/2}\right)-\theta^{P^2\mathbb{O}}.
\end{equation} 
To get a solution $\lambda$ we need to consider 
the real and imaginary parts
in the formula
\[
\sqrt{2}\sqrt{-1}\partial\,||A||^{1/2}
\]
separately. So, put
\[
{\tau_{\mathbb{O}}}^{*}\big(\sqrt{2}\sqrt{-1}\partial\,||A||^{1/2}\big)
:=a+\sqrt{-1}b
\]
with real and imaginary parts of the one-form %\break % \break
${\tau_{\mathbb{O}}}^{*}\big(\sqrt{2}\sqrt{-1}\partial\,||A||^{1/2})$ on
$\mathcal{J}(3)\times \mathcal{J}(3)$. Then
\[
d\,(d\lambda)=d\big({\tau_{\mathbb{O}}}^{*}\big(\sqrt{2}\sqrt{-1}\partial\,||A||^{1/2}\big)-\theta^{P^2\mathbb{O}}\big)=0
\]
implies that there are  real valued functions $\lambda_{Re}$ and
$\lambda_{Im}$ such that 
\begin{align*}
&a-\theta^{P^2\mathbb{O}}=d\lambda_{Re},~\text{and}\\
&d\lambda_{Im}=b.
\end{align*}
The problem to solve the equation \eqref{basic equation}
reduces
to find explicitly the functions $\lambda_{Re}$ and $\lambda_{Im}$.

Let $(X,Y)\in T^{*}_{0}(P^2\mathbb{O})\subset
\mathcal{J}(3)\times\mathcal{J}(3)$. 
Here again we remark that
we are identifying the cotangent space $T^{*}_{X}(P^2\mathbb{O})$ 
and the tangent space $T_{X}(P^2\mathbb{O})$ by the Riemannian metric defined 
by $(Y_{1},Y_{2})_{X}^{P^2\mathbb{O}}:=\text{tr}\,(Y_{1}\circ Y_{2})$ for $Y_{i}\in
T_{X}(P^2\mathbb{O})\cong \mathcal{J}(3)$, that is
for $Y_{i}\in T_{X}(P^2\mathbb{O})\subset \mathcal{J}(3),~i=1,2,$ 
\begin{align*}
&Y_{1}=\begin{pmatrix}\epsilon_{1}&u_{3}&\theta(u_{2})\\\theta(u_{3})&\epsilon_{2}&u_{1}\\u_{2}&\theta(u_{1})&\epsilon_{3}\end{pmatrix},
\quad Y_{2}=\begin{pmatrix}\eta_{1}&v_{3}&\theta(v_{2})\\\theta(v_{3})&\eta_{2}&v_{1}\\v_{2}
&\theta(v_{1})&\eta_{3}\end{pmatrix},\quad\epsilon_{i},\eta_{i}\in\mathbb{R},~u_{i}, v_{i}\in\mathbb{O}\cong\mathbb{R}^{8},
\end{align*} 
\begin{equation}\label{inner product}
(Y_{1},Y_{2})_{X}^{P^2\mathbb{O}}:=\text{tr}\,(Y_{1}\circ Y_{2})=\sum \epsilon_{i}\eta_{i}+2\sum (u_{i},v_{i})^{\mathbb{R}^8}.
\end{equation}
Based on this expression, by the notation $(Y,dX)$ for 
\begin{align*}
&X=\begin{pmatrix}\xi_{1}&x_{3}&\theta(x_{2})\\
\theta(x_{3})&\xi_{2}&x_{1}\\x_{2}&\theta(x_{1})&\xi_{3}\end{pmatrix}\in \mathcal{J}(3),\quad~Y=\begin{pmatrix}\epsilon_{1}&u_{3}&\theta(u_{2})\\\theta(u_{3})&\epsilon_{2}&u_{1}\\u_{2}&\theta(u_{1})&\epsilon_{3}\end{pmatrix}\in\mathcal{J}(3),
\end{align*}
we mean the canonical one-form 
\[
(Y, dX):=\sum \epsilon_{i}d\xi_{i}+2\sum \{u_1\}_{i}d\{x_{1}\}_{i}+\{u_2\}_{i}d\{x_{2}\}_{i}+ \{u_3\}_{i}d\{x_{3}\}_{i} 
\]
on
$T^{*}(\mathcal{J}(3))\cong\mathcal{J}(3)\times(\mathcal{J}(3))^{*}\cong\mathcal{J}(3)\times(\mathcal{J}(3))$,
or its restriction to $T^{*}(P^2\mathbb{O})$, that is,
in the inner product expression \eqref{inner product}
we understand as 
$\eta_{i}$ and $\{v_{k}\}_{i}$ ($k=1,2,3, i=0,\cdots,7$) are replaced 
by the differentials $d\xi_{i}$ and 
$d\{x_{k}\}_{i}$ of the corresponding components in
$X\in\mathcal{J}(3)$, respectively.

Also for $A\in\mathcal{J}(3)^{\mathbb{C}}$ and an one form $B$ on $\mathcal{J}(3)^{\mathbb{C}}$
we express the complex one form
$(A,dB)$ in the same way. 

Let $(X,Y)\in T(P^{2}\mathbb{O})\cong T^{*}(P^{2}\mathbb{O})$ and put $A=\tau_{\mathbb{O}}(X,Y)
=1\otimes\left(||Y||^2X-Y^{2}\right)+\sqrt{-1}\otimes\frac{||Y||Y}{\sqrt{2}}$, then
%
%
%Now we recall the map $\tau_{\mathbb{O}}$:
%\begin{align*}\label{complex structure map}
%&\tau_{\mathbb{O}}:
%\mathcal{J}(3)\times \mathcal{J}(3)\supset
 % T^{*}_{0}(P^{2}\mathbb{O})\ni (X,Y)\longmapsto
%A:=\tau_{\mathbb{O}}(X,Y)\in \mathbb{C}\otimes_{\mathbb{R}}\mathcal{J}(3)=\mathcal{J}%(3)^{\mathbb{C}}\\
%&\tau_{\mathbb{O}}(X,Y)=1\otimes\left(||Y||^2X-Y^{2}\right)+\sqrt{-1}\otimes\frac{||Y||Y}{\sqrt{2}}.
%\end{align*}
%Based on this expression of the matrix $A=\tau_{\mathbb{O}}(X,Y)\in\mathcal{J}(3)^{\mathbb{C}}$
%we have
\begin{proposition}{\em(}see \cite{Fu2}{\em)}
\begin{align*}
&\frac{1}{2}||Y||^{4}=||a||^2=||b||^2,\quad
  ||A||^2=||a||^2+||b||^2=||Y||^4,~\text{and}\quad
(da,a)=||Y||^{2}(Y,dY)=(db,b).
\end{align*}
\end{proposition}
In the expression
\begin{align*}
&\tau^{*}(dA,\overline{A})=(\tau^{*}(dA),\tau^{*}(\overline{A})
=(a-\sqrt{-1}b,da+\sqrt{-1}db)=d||A||^2\\
&=(a,da)+(b,db)+\sqrt{-1}\big((a,db)-(b,da)\big), \quad\text{and}\\
&(a,db)-(b,da)=2(db,a)
=\frac{2}{\sqrt{2}}\cdot \left\{||Y||^3(dY,X)-||Y||(dY,Y\circ Y)\right\}
\end{align*}
and it is proved in \cite{Fu2} (page 179) that
\[
(dY,Y\circ Y)=0.
\]
Hence
\begin{align*}
\tau_{\mathbb{O}}^{*}(\sqrt{2}\sqrt{-1}\partial\,||A||^{1/2})-\theta^{P^2\mathbb{O}}%\\
%&
&=\sqrt{2}\sqrt{-1}\frac{1}{4||Y||^3}\{\sqrt{-2}\cdot (dY,X)-2||Y||^2(Y,dY)-\theta^{P^2\mathbb{O}}\\
&=-(dY,X)-(Y,dX)+\frac{\sqrt{-1}}{\sqrt{2}||Y||}(Y,dY)
=\frac{\sqrt{-1}}{\sqrt{2}}d||Y||,
\end{align*}
since $d(X,Y)=(dX,Y)+(Y,dX)=d\,\text{tr}(X\circ Y)=0$ for $(X,Y)\in T^{*}(P^2\mathbb{O})$.
Hence finally we may choose the solutions $\lambda_{Re}$ and
$\lambda_{Im}$ with
\[
\lambda_{Re}\equiv 0~\text{and}~\lambda_{Im}=\frac{1}{\sqrt{2}}||Y||.
\]
Hence
\begin{proposition}\label{solution of the basic equation}
\[
g_{0}=e^{-\sqrt{2}\pi\,||Y||}, \text{or it is expressed on $\mathbb{X}_{\mathbb{O}}$ as}~ g_{0}=e^{-\sqrt{2}\pi\,||A||^{1/2}}.
\]
\end{proposition}
Now we have
\begin{align}
&\mathfrak{B}:C_{\mathcal{G}}(\mathbb{X}_{\mathbb{O}})\to C^{\infty}(P^2\mathbb{O}),\notag\\
&\mathfrak{B}(h)=\{{\bf q}\circ (\tau_{\mathbb{O}})^{-1}\}_{*}(h\cdot<{\bf t}_{0},{\bf s}_{{0}}>^{\mathbb{L}}\Omega_{\mathbb{O}})
=\{{\bf q}\circ (\tau_{\mathbb{O}})^{-1}\}_{*}(h\cdot e^{-\sqrt{2}\pi||A||^{1/2}}\Omega_{\mathbb{O}}).\label{final form of Bargmann type transformation}
\end{align}

\begin{remark}
The solution $\lambda_{Re}$ can be an arbitrary real constant. However
the absolute value $|g_{0}|$ does not depend on the chosen constant $\lambda_{Re}$.
\end{remark}

\subsection{Fock-like space}\label{Fock space}
We show
\begin{proposition}\label{invariance under F4}
The nowhere vanishing global holomorphic section $\Omega_{\mathbb{O}}$ of the
canonical line bundle $K^{\mathcal{G}}$ is $F_{4}$-invariant.
\end{proposition}
\begin{proof}
Let $\alpha\in F_{4}$. The action of $\alpha$ on $\mathbb{X}_{\mathbb{O}}$ is
naturally defined from the action on $P^2\mathbb{O}$ 
and the action is holomorphic. We denote it with the
same notation $\alpha:\mathbb{X}_{\mathbb{O}}\to \mathbb{X}_{\mathbb{O}}$.

We can put $\alpha^{*}(\Omega_{\mathbb{O}})=K_{\alpha}\cdot\Omega_{\mathbb{O}}$ with a nowhere
vanishing holomorphic function $K_{\alpha}=K_{\alpha}(A)$.

Then 
$$
\alpha^{*}(\Omega_{\mathbb{O}})\bigwedge\overline{\alpha^{*}(\Omega_{\mathbb{O}})}=
\alpha^{*}\big(\Omega_{\mathbb{O}}\bigwedge\overline{\Omega_{\mathbb{O}}}\big)=
|K_{\alpha}|^2\cdot\Omega_{\mathbb{O}}\bigwedge\overline{\Omega_{\mathbb{O}}}.
$$
We can express 
\[
\Omega_{\mathbb{O}}\bigwedge\overline{\Omega_{\mathbb{O}}}
=D\cdot\frac{1}{16!}\{{\tau_{\mathbb{O}}}^{-1}\}^{*}\Big((\omega^{P^2\mathbb{O}})^{16}\Big)
\]
by the Liouville volume form  
$\displaystyle{\frac{1}{16!}(\omega^{P^2\mathbb{O}})^{16}}$ and a function $D=D(A)$
on $\mathbb{X}_{\mathbb{O}}$. Hence
\[
\alpha^{*}\big(\Omega_{\mathbb{O}}\bigwedge\overline{\Omega_{\mathbb{O}}}\big)=\alpha^{*}(D)\cdot
\frac{1}{16!}\{{\tau_{\mathbb{O}}}^{-1}\}^{*}\Big((\omega^{P^2\mathbb{O}})^{16}\Big),
\]
since the action by $\alpha$ on $\mathbb{X}_{\mathbb{O}}$ is symplectic.
Hence
\[
\alpha^{*}(D)=|K_{\alpha}|^2\cdot D.
\]
By comparing
the behaviours of $\Omega_{\mathbb{O}}$ and the Liouville volume form  $dV_{P^{2}\mathbb{O}}$
under the dilation action by positive numbers:
\[
T_{t}:\mathbb{X}_{\mathbb{O}}\to\mathbb{X}_{\mathbb{O}}, ~A\to t\cdot A,
\]
we can see on the coordinate neighborhood $O_{z_{1}}$
\begin{align*}
T_{t}^{*}\big(\Omega_{\mathbb{O}}\bigwedge\overline{\Omega_{\mathbb{O}}}\big)&=
\frac{1}{(t z_{1})^5}d(tz_{1})\wedge\cdots\wedge d(t\xi_{2})\,
\bigwedge\,
\frac{1}{(t \overline{z_{1}})^{5}}
d(t\overline{z_{1}})\wedge \cdots\wedge d(t\overline{\xi_{2}})\\
&={t^{22}}\frac{1}{{z_{1}}^{5}}dz_{1}\wedge\cdots\wedge d\xi_{2}\,
\bigwedge\,\frac{1}{{\overline{z_{1}}}^{\,5}}
d\overline{z_{1}}\wedge \cdots\wedge d\overline{\xi_{2}}\\
&=t^{22}\cdot D(A) \cdot \frac{1}{16!}\{{\tau_{\mathbb{O}}}^{-1}\}^{*}\Big((\omega^{P^2\mathbb{O}})^{16}\Big)
= D(t\cdot A)\cdot t^{8}\frac{1}{16!}\{{\tau_{\mathbb{O}}}^{-1}\}^{*}\Big((\omega^{P^2\mathbb{O}})^{16}\Big).
\end{align*}
Hence 
\[
D(t\cdot A)=t^{14}\cdot D(A).
\]
Note that the action $T_{t}$ on $T^{*}_{0}(P^{2}\mathbb{O})$ defined
via the map $\tau_{\mathbb{O}}$ is
\begin{equation}\label{dilation order}
{\tau_{\mathbb{O}}}^{-1}\circ T_{t}\circ 
\tau_{\mathbb{O}}:T^{*}_{0}(P^{2}\mathbb{O})\ni (X,Y)\longmapsto
(X,\sqrt{t}Y)\in T^{*}_{0}(P^{2}\mathbb{O}). 
\end{equation}
Then since $||\alpha(A)||=||A||$
\begin{align*}
&\alpha^{*}(D)(A)=D\left(\alpha(A)\right)=
D\left(||\alpha(A)||\cdot\frac{\alpha(A)}{||\alpha(A)||}\right)\\
&=||A||^{14}\cdot D\left(\frac{\alpha(A)}{||\alpha(A)||}\right)
=||A||^{14}\cdot |K_{\alpha}(A)|^2D\left(\frac{A}{||A||}\right),
\end{align*}
hence
\begin{equation}\label{boundedness}
D\left(\frac{\alpha(A)}{||\alpha(A)||}\right)=|K_{\alpha}(A)|^2\,D\left(\frac{A}{||A||}\right).
\end{equation}
This equality implies that the function $K_{\alpha}$ is bounded on
$\mathbb{X}_{\mathbb{O}}$.
Especially, if we consider it on the coordinate open
subset $O_{z_{1}}\cong \mathbb{C}^{*}\times \mathbb{C}^{15}$
($z_{1}\not=0$), then it can be extended to a holomorphic function on 
$\mathbb{C}\times \mathbb{C}^{15}\supset O_{z_{1}}$ and is bounded there. Hence 
the function $K_{\alpha}$ is
a constant function on the whole space $\mathbb{X}_{\mathbb{O}}$.

Then by the property
\[
K_{\alpha\cdot \beta}=K_{\alpha}\cdot K_{\beta}, \alpha,\beta\in F_{4},
\]
$F_{4}\ni \alpha\longmapsto K_{\alpha}$ is a one-dimensional
representation of the compact simply connected group $F_{4}$, so that
we have not only $|K_{\alpha}|\equiv 1$ for any $\alpha\in F_{4}$, but
also it must hold always $K_{\alpha}\equiv 1$. This implies 
$\Omega_{\mathbb{O}}$ is $F_{4}$-invariant. 
\end{proof}

\begin{corollary}
Since the action of $F_{4}$
on $S(\mathbb{X}_{\mathbb{O}})=\{\,A\in\mathbb{X}~|~||A||=1\,\}$ is transitive, 
the function is of the form $D(A)={C_{1}}\times ||A||^{14}$ with the
constant $C_{1}={2^{26}}$. 
Especially we have
\begin{equation}\label{Calabi-Yau and Liouville}
{\tau_{\mathbb{O}}}^{*}\left(\Omega_{\mathbb{O}}\wedge \overline\Omega_{\mathbb{O}}\right)(X,Y)
=2^{26}||Y||^{28}\frac{1}{16!}\left(\omega^{P^2\mathbb{O}}\right)^{16}.
\end{equation}
\end{corollary}
\begin{proof}
It is enough to determine the constant $C_{1}$.

Following the expression \eqref{vector expression of A}  of the matrix $A\in\mathcal{J}(3)^{\mathbb{C}}$
we denote 
%The components of the the matrix
\[A=\begin{pmatrix}\xi_{1}&c'+c''{\bf e}_{4}&\theta(b'+b''{\bf e}_{4})\\
\theta(c'+c''{\bf e}_{4})&\xi_{2}&a'+a''{\bf e}_{4}\\
b'+b''{\bf e}_{4}&\theta(a'+a''{\bf e}_{4})&\xi_{3}\end{pmatrix}
\in\mathcal{J}(3)^{\mathbb{C}}
\]
by
\begin{align*}
&(\xi_{1},\xi_{2},\xi_{3},z_{1},z_{2},z_{3},z_{4},w_{1},w_{2},w_{3},w_{4},y_{1},y_{2},y_{3},y_{4},
x_{1},x_{2},x_{3},x_{4},u_{1},u_{2},u_{3},u_{4}),
%&(x_{1},\cdots,x_{4},u_{1},\cdots,u_{4},y_{1},\cdots,y_{4},v_{1},\cdots,v_{4},
%z_{1},\cdots,z_{4},w_{1},\cdots,w_{4},\xi_{1},\xi_{2},\xi_{3})\\
%&=\left(
%\eta_{1},\eta_{2},\eta_{3}
%\begin{pmatrix}z_{1}&z_{2}\\z_{3}&z_{4}\end{pmatrix},
%\begin{pmatrix}w_{1}&w_{2}\\w_{3}&w_{4}\end{pmatrix},
%\begin{pmatrix}y_{1}&y_{2}\\y_{3}&y_{4}\end{pmatrix},
%\begin{pmatrix}v_{1}&v_{2}\\v_{3}&v_{4}\end{pmatrix},
%\begin{pmatrix}x_{1}&x_{2}\\x_{3}&x_{4}\end{pmatrix},
%\begin{pmatrix}u_{1}&u_{2}\\u_{3}&u_{4}\end{pmatrix}\right)
\end{align*}
where 
\begin{align*}
&\rho_{\mathbb{H}}(c')+\rho_{\mathbb{H}}(c''){\bf e}_{4}
=\begin{pmatrix}z_{1}&z_{2}\\z_{3}&z_{4}\end{pmatrix}+\begin{pmatrix}w_{1}&w_{2}\\w_{3}&w_{4}\end{pmatrix}{\bf e}_{4},\\
&\rho_{\mathbb{H}}(b')+\rho_{\mathbb{H}}(b''){\bf e}_{4}
=\begin{pmatrix}y_{1}&y_{2}\\y_{3}&y_{4}\end{pmatrix}+\begin{pmatrix}v_{1}&v_{2}\\v_{3}&v_{4}\end{pmatrix}{\bf e}_{4},\\
&\rho_{\mathbb{H}}(a')+\rho_{\mathbb{H}}(a''){\bf e}_{4}
=\begin{pmatrix}x_{1}&x_{2}\\x_{3}&x_{4}\end{pmatrix}+\begin{pmatrix}u_{1}&u_{2}\\u_{3}&u_{4}\end{pmatrix}{\bf e}_{4}
\end{align*}
The correspondence between $c=c'+c''{\bf e}_{4}=\sum\,\{c\}_{i}{\bf e}_{i}$ (and also 
$b=b'+b''{\bf e}_{4}=\sum\,\{b\}_{i}{\bf e}_{i}, a=a'+a''{\bf e}_{4}=\sum\,\{a\}_{i}{\bf e}_{i}$),
and the components $\{z_{i},w_{i}\}$ is given in \eqref{matrix representation} and \eqref{matrix to vector}.

By a simple calculation we have
\begin{align*}
||A||^2&=\sum\,|\xi_{i}|^2+2\sum\,|a'|^2+|a''|^2+|b'|^2+|b''|^2+|c'|^2+|c''|^2\\
&=\,\sum_{i=1}^{3} |\xi_{i}|^2+\sum_{i=1}^{4} \,|z_{i}|^2+|w_{i}|^2+|y_{i}|^2+|v_{i}|^2+|x_{i}|^2+|u_{i}|^2.
\end{align*}
Then we rewrite for $A\in O_{z_{1}}$
\[
A\longleftrightarrow (\xi_{1},\xi_{2},\xi_{3},z_{1},z_{2},z_{3},z_{4},w_{1},w_{2},w_{3},w_{4},y_{1},y_{2},y_{3},y_{4},
x_{1},x_{2},x_{3},x_{4},u_{1},u_{2},u_{3},u_{4}))
\]
as
\begin{align*}
&(x_{3},x_{4},u_{2},u_{4},y_{2},y_{4},
v_{1},v_{3},z_{1},z_{2},z_{3},
w_{1},w_{2},w_{3},w_{4},\xi_{2}~;~x_{1},x_{2},u_{1},u_{3},y_{1},y_{3},v_{2},v_{4},z_{4},\xi_{1},\xi_{3})\\
&=(s_{1},\cdots\cdots,s_{16}~;~s_{17},\cdots,s_{27}),
\end{align*}
%then $||A||=\sqrt{\sum\limits_{i=1}^{27}|s_{i}|^2}$,
that is, the first $16$
coordinates give local coordinates on ${O}_{z_{1}}$ and the remaining 
coordinates $(s_{17},\cdots,s_{27})$ are rational functions of the coordinates
$(s_{1},\cdots,s_{16})$, that is $s_{j}=s_{j}(s_{1},\cdots,s_{16})$,
$j\geq 17$ (especially $s_{9}=z_{1}$ and the explicit form of each $s_{j}$
for $j>16$ is given in \eqref{explicit expression of dependent variables}).

In particular, we see from the explicit form of these
functions at the point
$A=A(z_{1})=A(0,\cdots,z_{1},\cdots, 0,;0,\cdots,0)$
\begin{align*}
&s_{j}(A(z_{1}))=s_{j}(0,\cdots,0,
  z_{1},0,\cdots,0)=s_{j}(0,\cdots,0,s_{9},0,\cdots,0)=0,~\text{$17\leq
  j\leq 27$},\\
\intertext{and for $i\leq 16$, $j\geq 17$}
&\frac{\partial s_{j}}{\partial s_{i}}(A(z_{1}))=0.
\end{align*}
On $O_{z_{1}}$ it holds
\begin{align*}
\Omega_{\mathbb{O}}\bigwedge \overline{\Omega_{\mathbb{O}}}&=\frac{1}{|z_{1}|^{10}}\,dx_{3}\wedge
dx_{4}\wedge\cdots\wedge dw_{4}\wedge d\xi_{2}\,
\bigwedge\,
d\overline{x_{3}}\wedge
  d\overline{x_{4}}\wedge\cdots\wedge d{\overline{w_{4}}} \wedge d\overline{\xi_{2}}\\
&=\frac{1}{|s_{9}|^{10}}ds_{1}\wedge\cdots \wedge ds_{16}\bigwedge d\overline{s_{1}}\wedge\cdots \wedge d\overline{s_{16}}
=D(A)\cdot \frac{1}{16!}\{{\tau_{\mathbb{O}}}^{-1}\}^{*}\Big((\omega^{P^2\mathbb{O}})^{16}\Big).
\end{align*}
We calculate the right hand side at the point
$A(z_{1})=(0,\cdots,0,z_{1},0,\ldots;0,\cdots,0)\in O_{z_{1}}$ using the expression of
$\omega^{P^2\mathbb{O}}$ ({see Theorem 3.2}):
\[
\omega^{P^2\mathbb{O}}=\{\tau_{\mathbb{O}}\}^{*}\Big(\sqrt{-2}\overline{\partial}\partial ||A||^{1/2}\Big).
\]
Then   
\begin{align*}
\overline{\partial}\partial ||A||^{1/2}
&=\frac{1}{4}\cdot
  \overline{\partial}\left(\left(\sum_{i=1}^{27}|s_{i}|^2\right)^{-3/4}\cdot\sum\,\overline{s_{i}}ds_{i}\right)\\
&=\frac{-3}{16}\left(\sum_{i=1}^{27}|s_{i}|^2\right)^{-7/4}\sum_{j=1}^{27} \,s_{j}d\overline{s_{j}}\bigwedge
  \sum_{i=1}^{27}\,\overline{s_{i}}d s_{i}+\frac{1}{4}\left(\sum_{i=1}^{27}|s_{i}|^2\right)^{-3/4}\sum\limits_{i=1}^{27} \,d\overline{s_{i}}\wedge d s_{i}.\\
\intertext{Here we evaluate it at the point $A(z_{1})$, then it is given by}
&\frac{-3}{16}|s_{9}|^{-7/2}\cdot|s_{9}|^2\,d\overline{s_{9}}{\wedge d s_{9}}
+\frac{1}{4}|s_{9}|^{-3/2}\sum\limits_{i=1}^{16} \,d\overline{s_{i}}{\wedge d s_{i}}\\
&=\frac{1}{16}|s_{9}|^{-3/2}d\overline{s_{9}}{\wedge ds_{9}}
+\frac{1}{4}|s_{9}|^{-3/2}\sum\limits_{1\leq i\leq 8~\text{and}~10\leq i\leq 16}
d\overline{s_{i}}{\wedge d s_{1}}.
\end{align*}
Hence
\begin{align*}
(\omega^{P^2\mathbb{O}})^{16}&=\left(\sqrt{-2}\,\overline{\partial}\partial
  ||A||^{1/2}\right)^{16}%&\\
%&
=16!\cdot \frac{1}{16}\cdot\frac{1}{4^{15}}(\sqrt{-2})^{16}\cdot\frac{1}{|s_{9}|^{3/2\times
  16}}d\overline{s_{1}}{\wedge ds_{1}}{\wedge\cdots} {\cdots\wedge d\overline{s_{16}}}
  {\wedge ds_{16}}\\
&=16!\cdot\frac{1}{2^{26}\cdot |s_{9}|^{24}}
d\overline{s_{1}}\wedge ds_{1}\wedge\cdots\cdots \wedge
  d\overline{s_{16}}\wedge d s_{16}
\end{align*}
at the point $A(z_{1})=A(0,\cdots,0,s_{9},0,\cdots,0;0,\cdots,0)$ ($s_{9}=z_{1}$).
Consequently we have
\[
{\Omega_{\mathbb{O}}}\wedge\overline{\Omega_{\mathbb{O}}}_{\big|A(z_{1})}
=\overline{\Omega_{\mathbb{O}}}\wedge{\Omega_{\mathbb{O}}}_{\big|A(z_{1})}
=D\left(\frac{s_{9}}{|s_{9}|}\right)\cdot|s_{9}|^{14}\cdot \frac{1}{2^{26}}|s_{9}|^{-24}
d\overline{s_{1}}\wedge ds_{1}\wedge\cdots\cdots \wedge
d\overline{s_{16}}\wedge ds_{16}
\]
and the constant $C_{1}$ is
\[
C_{1}={2^{26}}, ~D(A)=2^{26}||A||^{14},~
\Omega_{\mathbb{O}}\bigwedge \overline{\Omega_{\mathbb{O}}}=2^{26}||A||^{14}\frac{1}{16!}
\{{\tau_{\mathbb{O}}}^{-1}\}^{*}\Big((\omega^{P^2\mathbb{O}})^{16}\Big).
\]
\end{proof}
%In the following, we denote 
%the space consisting of the restrictions of
%polynomials $\mathbb{C}[\,\mathcal{J}(3)^{\mathbb{C}}\,]\cong 
%\mathbb{C}[\xi_{1},\xi_{2},\xi_{3}, z_{1},\ldots,w_{4},y_{1},\ldots,v_{4},x_{1},\ldots,u_{4}]$ 
%$(=\sum\mathcal{P}_{k}[\xi_{1},\xi_{2},\xi_{3}, z_{1},\ldots,w_{4},y_{1},\ldots,v_{4},x_{1},\ldots,u_{4}]
%:=\sum\,\mathcal{P}_{k}\,:\text{sum of subspaces of homogeneous polynomials})$ 
%of 
%$27$ complex variables $(\xi_{1},\xi_{2},\xi_{3},z_{i},w_{i},y_{i},v_{i},x_{i},u_{i})$
%($i=1$,$\cdots$,$4$),
%~see $\S 4$ and precisely to say 
%modulo defining ideal of  
%the space $\mathbb{X}_{\mathbb{O}}$) to the
%subspace $\mathbb{X}_{\mathbb{O}}$ is denoted by
%$\mathcal{P}[\mathbb{X}_{\mathbb{O}}]=\sum\,\mathcal{P}_{k}[\mathbb{X}_{\mathbb{O}}]$ 

Let $\mathbb{C}[\mathcal{J}(3)^{\mathbb{C}}=\sum\,\mathcal{P}_{k}[\mathcal{J}(3)^{\mathbb{C}}]$
be the algebra 
of polynomials (and of polynomial functions) on $\mathcal{J}(3)^{\mathbb{C}}$
with the $27$ complex variables $(\xi_{1},\xi_{2},\xi_{3},z_{i},w_{i},y_{i},v_{i},x_{i},u_{i})$
($i=1$,$\cdots$,$4$ and $\mathcal{P}_{k}$ is a subspace of degree $k$ homogeneous polynomials) 
and denote their restrictions to $\mathbb{X}_{\mathbb{O}}$
by $\mathbb{C}[\mathbb{X}_{\mathbb{O}}]=\sum\,\mathcal{P}_{k}[\mathbb{X}_{\mathbb{O}}]$.

Recall the correspondence 
\[
\gamma:\mathbb{C}[\mathbb{X}_{\mathbb{O}}]=\sum\,\mathcal{P}_{k}[\mathbb{X}_{\mathbb{O}}]\ni 
p\longmapsto \gamma(p)=p\cdot{\bf t}_{0}\otimes\Omega_{\mathbb{O}}\in \Gamma_{\mathcal{G}}(\mathbb{L}\otimes
K^{\mathbb{G}},\mathbb{X}_{\mathbb{O}}).
\]
We define a parameter family of inner products 
$\{\boldsymbol{(}*\,,\,*\boldsymbol{)}_{\varepsilon}\}_{\varepsilon\in\mathbb{R}}$ on the space 
$\Gamma_{\mathcal{G}}(\mathbb{L}\otimes K^{\mathcal{G}},\mathbb{X}_{\mathbb{O}})$ by the following way that
\begin{align}
\Gamma_{\mathcal{G}}(\mathbb{L}\otimes
K^{\mathcal{G}},\mathbb{X}_{\mathbb{O}})&\times \Gamma_{\mathcal{G}}(\mathbb{L}\otimes K^{\mathcal{G}},\mathbb{X}_{\mathbb{O}})
\ni (h\cdot{\bf t}_{0}\otimes\Omega_{\mathbb{O}}\,,\,g\cdot {\bf t}_{0}\otimes\Omega_{\mathbb{O}})\notag\\
\longmapsto
&\int_{\mathbb{X}_{\mathbb{O}}}h\cdot\overline{g}<{\bf t}_{0},{\bf t}_{0}>^{\mathbb{L}}\cdot||A||^{\varepsilon}\cdot\Omega_{\mathbb{O}}\bigwedge\overline{\Omega_{\mathbb{O}}}\notag\\
&=2^{26}\int_{\mathbb{X}_{\mathbb{O}}}h\cdot\overline{g}\cdot e^{-2\sqrt{2}\pi||A||^{1/2}}\cdot ||A||^{14+\varepsilon}\cdot
\{{\tau_{\mathbb{O}}}^{-1}\}^{*}(dV_{P^2\mathbb{O}})=\boldsymbol{(}h\,,\,g\boldsymbol{)}_{\varepsilon}
\label{parameter family of inner products},%\label{inner products on holomorphic function space}
\end{align}
then through the map $\gamma$ we also consider a parameter family of inner products
on the space $\mathbb{C}[\mathbb{X}_{\mathbb{O}}]$. 
\begin{remark}\label{Fock-like space note}
According to the value of $\varepsilon$, 
the integral \eqref{parameter family of inner products} for functions
$f,g\in\mathcal{P}_{k}[\mathbb{X}_{\mathbb{O}}]$ need not be finite.
In fact, for $k>-11-\varepsilon/2$  
the integral \eqref{parameter family of inner products} converges. We
denote 
by $\mathfrak{F}_{\varepsilon}$
the completion of the space $\sum_{k\,>\,-11-\varepsilon/2}\,\mathcal{P}_{k}[\mathbb{X}_{\mathbb{O}}]$
with respect to the integral 
\eqref{parameter family of inner products}
%{inner products on holomorphic function space} 
and the  
the remaining finite dimensional space $\sum_{k\leq-11-\varepsilon/2}\mathcal{P}[\mathbb{X}_{\mathbb{O}}]$
with a suitable inner product.  
\end{remark}
%%%%%%%%%%%%%%%%%%%%%%%%%%%%%%%%%%%%%%%%%%%%%%%%%%%%%%%%%%%%%%%%%%%%%%%%%%%%
\section{Pairing with the Riemann volume form}

Let $dv_{P^2\mathbb{O}}$ be the Riemann volume form on
$P^2\mathbb{O}$. The purpose in this section is to show
\begin{proposition}
\begin{align}
&\{{\bf q\circ {\tau_{\mathbb{O}}}^{-1}}\}^{*}(dv_{P^2\mathbb{O}})(A)\bigwedge\,
\overline{\Omega_{\mathbb{O}}}\,(A)\notag\\
&=C_{RC}(A)\cdot \{{\tau_{\mathbb{O}}}^{-1}\}^{*}\left(\frac{1}{16!}\Big(\omega^{P^2\mathbb{O}}\Big)^{16}\right)(A)
=2^{26}||A||^{3}
\{{\tau_{\mathbb{O}}}^{-1}\}^{*}\left(\frac{1}{16!}\Big(\omega^{P^2\mathbb{O}}\Big)^{16}\right)(A),\\
&A=\tau_{\mathbb{O}}(X,Y)\in
  \tau_{\mathbb{O}}(T^{*}_{0}(P^{2}\mathbb{O}))=\mathbb{X}_{\mathbb{O}}.
\label{paring with Riemann volume form}
\end{align}
\end{proposition}

The homogeneity order is determined by comparing their orders 
in the both sides (see the relation \eqref{dilation order}).
\subsection{A local coordinates}
For the determination of the constant $C_{RC}(A/||A||)$ 
we choose a local coordinates around the
point $X_{1}=\begin{pmatrix}1&~0&~0\\0&~0&~0\\0&~0&~0\end{pmatrix}\in P^2\mathbb{O}$.

The condition 
$X^2=X= \begin{pmatrix}t_{1}&c&\theta(b)\\\theta(b)&t_{2}&a\\b&\theta(a)&t_{3}\end{pmatrix}$ 
for $X\in P^2\mathbb{O}\subset \mathcal{J}(3)$ is 
expressed as 
\begin{equation}\label{condition for Cayley projective plane}
\left\{
\begin{array}{l}
(t_{3}+t_{2})a+\theta(bc)=a,(t_{1}+t_{3})b+\theta(ca)=b,(t_{2}+t_{1})c+\theta(ab)=c,\\
{t_{1}}^2+c\theta(c)+\theta(b)b=t_{1},{t_{2}}^2+\theta(c)c+a\theta(a)=t_{2},
{t_{3}}^2+\theta(a)a+b\theta(b)=t_{3}~\text{and}\\
\text{tr}\,X\,=\,t_{1}+t_{2}+t_{3}=1.
\end{array}\right.
\end{equation}
where $a,b,c\in\mathbb{O},t_{i}\in\mathbb{R}$.
Using the last equation in \eqref{condition for Cayley projective
  plane}, first $6$ conditions are rewritten in the forms of  
\begin{equation}\label{condition for Cayley projective plane final}
\left\{
\begin{array}{l}
t_{1} a=\theta(bc),\quad t_{2} b=\theta(ca),\quad t_{3} c=\theta(ab),\\
(t_{1}-1/2)^2+c\theta(c)+\theta(b)b=(t_{1}-1/2)^2+|c|^2+|b|^2=1/4,\\
(t_{2}-1/2)^2+\theta(c)c+a\theta(a)=(t_{2}-1/2)^2+|c|^2+|a|^2=1/4,\\
(t_{3}-1/2)^2+\theta(x)x+b\theta(b)=(t_{3}-1/2)^2+|a|^2+|b|^2=1/4.
\end{array}\right.
\end{equation}

Let
\begin{equation}\label{coordinate  X1}
\mathcal{W}_{1}=\Big\{\,(b,c)\in\mathbb{O}^{2}~\Big|~\,|c|^2+|b|^2\,<\,\frac{1}{8}\Big\}.
\end{equation}

Then we can solve the equations \eqref{condition for Cayley projective plane final}
in the following order:

First, we solve the fourth equation 
in \eqref{condition for Cayley projective plane final} with respect to 
$t_{1}$ under the condition $|c|^2+|b|^2\,<\,\frac{1}{8}$
with the solution 
\[
t_{1}=\frac{1}{2}+\sqrt{\frac{1}{4}-|b|^2-|c|^2}>\frac{1}{2}.
\]
Then the component $a$ is given by $(b,c)$
by the first equation in \eqref{condition for Cayley projective plane final} as
\[
a=\frac{\theta(bc)}{t_{1}}.
\]
This solution $a$ satisfies the inequality:
\[
|a|=\frac{|bc|}{t_{1}}<2\cdot \frac{|b|^2+|c|^2}{2}< \frac{1}{8}.
\]
With these we can solve the variable $t_{2}$ in
the fifth equation in \eqref{condition for Cayley projective plane final} with the solution
\[
t_{2}=\frac{1}{2}-\sqrt{\frac{1}{4}-|c|^2-|a|^2},% <\frac{1}{2}.
\]
where $|c|^2+|a|^2<\frac{1}{8}+\frac{1}{64}<\frac{1}{4}$ implies that $t_{2}<\frac{1}{2}$.

Now, with these solutions expressed in terms of the variables
$(b,c)\in\mathcal{W}_{1}$ we define a map
\begin{equation}\label{local coordinate X1}
\mathcal{M}:\mathcal{W}_{1}\ni (b,c)\longmapsto X
=\begin{pmatrix}t_{1}&c&\theta(b)\\
\theta(c)&t_{2}&a\\
b&\theta(a)&1-t_{1}-t_{2}\end{pmatrix}
\in P^2\mathbb{O}.
\end{equation}

Then the matrix $\mathcal{M}(b,c)$ satisfies the condition \eqref{condition for Cayley projective plane final},
so that we can choose components $(b,c)$ as a local coordinates around the point $X_{1}$. We
denote by $\widetilde{\mathcal{W}}_{1}=\mathcal{M}(\mathcal{W}_{1})$. The point $X_{1}$ corresponds to
$(0,0)\in\mathcal{W}_{1}$.

\begin{lemma}\label{Riemann volume form at X1}
In terms of the local coordinates 
\[
(b,c)=\left(\sum_{i=0}^{7}\,\{b\}_{i}{\bf e}_{i},\sum_{i=0}^{7}\,\{c\}_{i}{\bf e}_{i}\right)
\]
introduced above around the point $X_{1}$,
the Riemann volume form $dv_{P^2\mathbb{O}}$ at the point $X_{1}$ is
\begin{equation}
dv_{P^2\mathbb{O}}(0,0)=
d{\{b\}_{0}}\wedge \cdots\wedge d\{b\}_{7}\wedge d\{c\}_{0}\wedge\cdots\wedge d\{c\}_{7}.
\label{Riemann volume form expression at X1}
\end{equation}
\end{lemma}
\begin{proof}
We can see this by 
\begin{align*}
d\mathcal{M}_{(0,0)}\left(\frac{\partial}{\partial \{b\}_{0}}\right)&=\left(\frac{\partial}{\partial \{b\}_{0}}\right)_{X_{1}}
+\sum_{i=1}^{3}\frac{\partial t_{i}(0,0)}{\partial \{b\}_{0}}\left(\frac{\partial}{\partial t_{i}}\right)_{X_{1}}
+\sum_{i=0}^{7}\,\frac{\partial \{a\}_{i}(0,0)}{\partial\{b\}_{0}}\left(\frac{\partial}{\partial \{a\}_{i}}\right)_{X_{1}}
=\left(\frac{\partial}{\partial \{b\}_{0}}\right)_{X_{1}},
\end{align*}
where we know
\begin{align*}
&\frac{\partial t_{1}(0,0)}{\partial
  \{b\}_{0}}=\frac{-\{b\}_{0}}{\sqrt{1/4-|b|^2-|c|^{2}}}_{\big|b=0,c=0}=0,\quad
  \frac{\partial t_{2}(0,0)}{\partial \{b\}_{0}}
=\frac{-2b_{0}-2\sum_{i=0}^{7}\{a\}_{i}\frac{\partial
  \{a\}_{i}}{\partial
  \{b\}_{0}}}{2\sqrt{1/4-|b|^2-|a|^{2}}}_{\big|b=0,c=0}=0, \,\,etc.,
\end{align*}
since $a(0,0)=\sum\,\{a\}_{i}{\bf e}_{i}=0$.
Other derivatives are also
\begin{align*}
&\displaystyle{\frac{\partial t_{i}}{\partial \{b\}_{j}}}_{\big|(0,0)}=0,
~\displaystyle{\frac{\partial t_{i}}{\partial \{c\}_{j}}}_{\big|(0,0)}=0,~\displaystyle{\frac{\partial \{a\}_{j}\{a\}_{k}}{\partial \{b\}_{i}}}_{\big|(0,0)}=0,
~\displaystyle{\frac{\partial \{a\}_{j}\{a\}_{k}}{\partial \{c\}_{i}}}_{\big|(0,0)}=0.
\end{align*}
Hence 
\begin{align*}
&d\mathcal{M}_{(0,0)}\left(\frac{\partial}{\partial
  \{b\}_{i}}\right)=\left(\frac{\partial}{\partial
  \{b\}_{i}}\right)_{X_{1}},\quad d\mathcal{M}_{(0,0)}\left(\frac{\partial}{\partial \{c\}_{i}}\right)=\left(\frac{\partial}{\partial \{c\}_{i}}\right)_{X_{1}}.
\end{align*}
Then 
the metric tensor $g_{ij}$ with respect to the coordinates $(b,c)$ at the point $(b,c)=(0,0)$
is $g_{ij}=\delta_{ij}$.
\end{proof}

\subsection{Explicit determination of the 
pairing with the Riemann volume form}\label{pairing with Riemann volume form}

Let  
\[A=\begin{pmatrix}\xi_{1}&z&\theta(y)\\\theta(z)&\xi_{2}&x\\y&\theta(x)&\xi_{3}\end{pmatrix}
\in\mathbb{X}_{\mathbb{O}}, \text{where $\xi_{i}\in\mathbb{C}$,
$z,y,x\in\mathbb{C}\otimes_{\mathbb{R}}\mathbb{O}$}.
\]
Put ${\tau_{\mathbb{O}}}^{-1}(A)=(X(A),Y(A))$, then
\begin{align*}
&\qquad X(A)=\frac{A+\overline{A}}{2||A||}+\frac{A\circ
  \overline{A}}{||A||^2}~(\text{see \eqref{inverse of tau-O}})\\
&=\begin{pmatrix}\frac{\xi_{1}+\overline{\xi_{1}}}{2||A||}+\frac{|\xi_{1}|^{2}+|z|^2+|y|^2}{||A||^2}&
\frac{z+\overline{z}}{2||A||}+\frac{-\xi_{3}\overline{z}-\overline{\xi_{3}}z
+\theta(\overline{x}y+x\overline{y})}{2||A||^2}&
\frac{\theta(y+\overline{y})}{2||A||}
+\frac{-\theta(\xi_{2}\overline{y}+\overline{\xi_{2}}y)+z\overline{x}+\overline{z}x}{2||A||^2}\\
\frac{\theta(z+\overline{z})}{2||A||}
+\frac{-\theta(\xi_{3}\overline{z}+\overline{\xi_{3}}z)+\overline{x}y+x\overline{y}}{2||A||^2}&
\frac{\xi_{2}+\overline{\xi_{2}}}{2||A||}+\frac{|\xi_{2}|^2+|z|^2+|x|^2}{||A||^2}&
\frac{x+\overline{x}}{2||A||}+\frac{-\xi_{1}\overline{x}-\overline{\xi_{1}}x
+\theta(\overline{y}z+y\overline{z})}{2||A||^2}\\
\frac{y+\overline{y}}{2||A||}+\frac{-\xi_{2}\overline{y}-\overline{\xi_{2}}y
+\theta(z\overline{x}+\overline{z}x)}{2||A||^2}&
\frac{\theta(x+\overline{x})}{2||A||}
+\frac{-\theta(\xi_{1}\overline{x}+\overline{\xi_{1}}x)+\overline{y}z+y\overline{z}}{2||A||^2}&
\frac{\xi_{3}+\overline{\xi_{3}}}{2||A||}+\frac{|\xi_{3}|^2+|x|^2+|y|^2}{||A||^2}
\end{pmatrix}.
\end{align*}
From the above expression of ${\tau_{\mathbb{O}}}^{-1}(A)=(X(A),Y(A))$ we consider
two components of the matrix $X(A)\in P^2\mathbb{O}$ for $A\in {U}_{z_{1}}$:
\begin{align*}
&c=\frac{z+\overline{z}}{2||A||}+\frac{-\xi_{3}\overline{z}-\overline{\xi_{3}}z
+\theta(\overline{x}y+x\overline{y})}{2||A||^2},\quad b=\frac{y+\overline{y}}{2||A||}+\frac{-(\xi_{2}\overline{y}+\overline{\xi_{2}}y)
+\theta(z\overline{x}+\overline{z}x)}{2||A||^2}.
\end{align*}
Take a point $A_{1}=\begin{pmatrix}1&{\sqrt{-1}}{\bf
    e}_{0}&0\\\sqrt{-1}{\bf e}_{0}&-1&0\\0&0&0\end{pmatrix}\in
O_{z_{1}}$, then
${\bf q}\circ{\tau_{\mathbb{O}}}^{-1}(A_{1})=X_{1}$.
On the
other hand
the point $A_{1}\in\mathbb{X}_{\mathbb{O}}$ corresponds to the
matrices 
\begin{align*}
&A_{1}\longleftrightarrow\left(\xi_{1},\xi_{2},\xi_{3},Z,W,Y,V,X,U\right)
%\left(X,U,Y,V,Z, W,\xi_{1},\xi_{2},\xi_{3}\right)\\
=\left(1,-1,0,
\begin{pmatrix}\sqrt{-1}&0\\0&\sqrt{-1}\end{pmatrix},
\begin{pmatrix}0&0\\0&0\end{pmatrix},
\begin{pmatrix}0&0\\0&0\end{pmatrix},
\begin{pmatrix}0&0\\0&0\end{pmatrix},
\begin{pmatrix}0&0\\0&0\end{pmatrix},
\begin{pmatrix}0&0\\0&0\end{pmatrix}\right).
%\begin{pmatrix}\sqrt{-1}&0\\0&\sqrt{-1}\end{pmatrix},\begin{pmatrix}0&0\\0&0\end{pmatrix},
%\in \mathbb{C}(2)^{6}\times \mathbb{C}^{3}
\end{align*}
(see the matrix representation \eqref{rep of complex octanion by  2 by 2 complex matrix 1} of the octanions and
and vector representation \eqref{vector expression of A} 
of elements in $\mathcal{J}(3)^{\mathbb{C}}$).

So we consider points $A\in U_{z_{1}}$ 
around a point 
\begin{align*}
P_{z_{1}}(A)&=(\xi_{2},z_{1},z_{2},z_{3},w_{1},w_{2},w_{3},w_{4},y_{2},y_{4},v_{1},v_{3},x_{3},x_{4},u_{2},u_{4})\\
&=
(-1,\sqrt{-1},0,0,0,0,0,0,0,0,0,0,0,0,0,0).
\end{align*}
By the explicit expression 
\eqref{explicit expression of dependent variables} of the other dependent
variables 
$(\xi_{1},\xi_{3},z_{4},y_{1},y_{3},v_{2},v_{4},x_{1},x_{2},u_{1},u_{2})$
%(x_{1},x_{2},u_{1},u_{2},y_{1},y_{3},v_{2},v_{4},z_{4},\xi_{1},\xi_{3})$
in the matrix representation 
$(\xi_{1},\xi_{2},\xi_{3},Z,W,Y,V,X,U)$ of
$A_{1}$ is
$
(0,0,0,0,0,0,0,0,\sqrt{-1},1,0).
$

For avoiding the confusion of the expression of octanion and its matrix expression by the mp $\rho_{O}$,
recall the correspondence 
\eqref{matrix representation} and \eqref{matrix to vector}).
%\eqref{matrix representation} and \eqref{matrix to vector}.
%that we distinguish the notations for the coefficients of the (complex) octanions
%$z=\sum\,\{z\}_{i}{\bf e}_{i},y=\sum\,\{y\}_{i}{\bf e}_{i},x=\sum\,\{x\}_{i}{\bf e}_{i}$  
%appeared in
%the expression of the matrix $A$ and the matrix components in the
%$2\times 2$ matrix expression throught the map $\rho_{\mathbb{O}}$ 
%as $A\longleftrightarrow(\xi_{1},\xi_{2},\xi_{3},Z,W,Y,V,X,U)$. 
%For example, for $Z=\begin{pmatrix}z_{1}&z_{2}\\z_{3}&z_{4}\end{pmatrix}$,
%$\displaystyle{\{z\}_{0}=\frac{z_{1}+z_{4}}{2}}$ and
%$\displaystyle{\{z\}_{1}= \frac{z_{1}-z_{4}}{2\sqrt{-1}}}$
%(see \eqref{matrix representation} and \eqref{matrix to vector}).
%\eqref{2 by 2 complex matrix to complex quaternion}).

Now we determine the differentials {\it modulo anti-holomorphic differentials} 
\begin{align*}
&\{{\bf q}\circ\tau_{\mathbb{O}}^{-1}\}^{*}(dc)
=\sum_{i=0}^{7}\,\{{\bf q}\circ\tau_{\mathbb{O}}^{-1}\}^{*}(d\{c\}_{i})\otimes{\bf e}_{i}
=\sum_{i=0}^{7}\,d\,\big(\{{\bf q}\circ\tau_{\mathbb{O}}^{-1}\}^{*}(\{c\}_{i})\big)\otimes{\bf e}_{i},~\text{and}\\
&\{{\bf q}\circ\tau_{\mathbb{O}}^{-1}\}^{*}(db)%=%\sum_{i=0}^{7}\,d\{b\}_{i}\otimes{\bf e}_{i}
=\sum_{i=0}^{7}\,\{{\bf
  q}\circ\tau_{\mathbb{O}}^{-1}\}^{*}(d\{b\}_{i})\otimes{\bf e}_{i}
=\sum_{i=0}^{7}\,d\,\big(\{{\bf q}\circ\tau_{\mathbb{O}}^{-1}\}^{*}(\{b\}_{i})\big)\otimes{\bf e}_{i},
\end{align*}
at the point $A_{1}$.

Each component of $b$ and $c$ is given by
\begin{align*}
&\{c\}_{i}=\frac{\{z\}_{i}+\{\overline{z}\}_{i}}{2||A||}+\frac{-\xi_{3}\{\overline{z}\}_{i}-\overline{\xi_{3}}\{z\}_{i}
+\{\theta(\overline{x}y+x\overline{y})\}_{i}}{2||A||^2},~\text{and}\\
&\{b\}_{i}=\frac{\{y\}_{i}+\{\overline{y}\}_{i}}{2||A||}
+\frac{-\xi_{2}\{\overline{y}\}_{i}-\overline{\xi_{2}}\{y\}_{i}
+\{\theta(\overline{z}x+z\overline{x})\}_{i}}{2||A||^2}.
\end{align*}

The pull-back $\{{\bf q}\circ{\tau_{\mathbb{O}}}^{-1}\}^{*}(dv_{P^2\mathbb{O}})$ is expressed as
\[
\{{\bf q}\circ{\tau_{\mathbb{O}}}^{-1}\}^{*}(dv_{P^2\mathbb{O}})=\sum_{i=0}^{16}\Sigma_{i},~\,\,\,
\text{with}~\,\,\,
%\]
%where
%\[
\Sigma_{i}\in \Gamma\left(\stackrel{16-i}\bigwedge
  T^{*'}(\mathbb{X}_{\mathbb{O}})^{\mathbb{C}}\otimes \stackrel{i}\bigwedge  T^{*''}(\mathbb{X}_{\mathbb{O}})^{\mathbb{C}}\right).
\]
In particular,
\[
\Sigma_{i}\wedge\overline{\Omega}=0~\text{for} ~i\geq 1,
~\text{and}~\overline{\Sigma}_{j}=\Sigma_{16-j}.
\]
Hence for the determination of the constant $C_{RC}(Y/||Y||)$, it is enough to consider 
the terms consisting of holomorphic differentials 
\[
d\xi_{2},dz_{1},dz_{2},dz_{3},dw_{1},dw_{2},dw_{3},dw_{4},dy_{2},dy_{4},dv_{1},dv_{3},
dx_{3},dx_{4},du_{2},du_{4}
\]
and may ignore the anti-holomorphic differentials
$d\overline{x}_{3},d\overline{x}_{4}, ~etc,$ so that 
in the expression of equalities below we denote them as
$*\,\equiv\, *$, which means both sides coincide {\it modulo anti-holomorphic
differentials}.

Here we gather up relations of the holomorphic differentials of
dependent variables by independent variables at the point $A_{1}$.
See \eqref{explicit expression of dependent variables} for the explicit expression of each
variable $\xi_{1},\xi_{3},\cdots\cdots, x_{1},x_{2},u_{1},u_{2}$ 
in terms of independent variables $\xi_{2},z_{1},\cdots, x_{3},x_{4},u_{2},u_{4}$.

{\it All the equalities in the Lemmas blow hold at the point $A_{1}$.}

\begin{lemma}
\allowdisplaybreaks{
\begin{align*}
&||A_{1}||=2,~d||A||^2_{|A_{1}}\,\equiv\,\{\overline{z_{1}}dz_{1}+\overline{z}_{4}dz_{4}+\overline{\xi_{1}}d\xi_{1} 
+\overline{\xi_{2}}d\xi_{2}\}_{|A_{1}}=-2d\xi_{2},\\
&z_{4}(A_{1})=\sqrt{-1},~{dz_{4}}_{|A_{1}}=-dz_{1}-2\sqrt{-1}d\xi_{2},~\xi_{3}(A_{1})=0,{d\xi_{3}}_{|A_{1}}=0,~{d\xi_{1}}_{|A_{1}}=-d\xi_{2},\\
&{dx_{1}}_{|A_{1}}=\sqrt{-1}dy_{4},~{dy_{1}}_{|A_{1}}=-\sqrt{-1}dx_{4},{dy_{3}}_{|A_{1}}=\sqrt{-1}dx_{3},~{dx_{2}}_{|A_{1}}=-\sqrt{-1}dy_{2},\\
&{dv_{2}}_{|A_{1}}=\sqrt{-1}du_{2},~{dv_{4}}_{|A_{1}}=\sqrt{-1}du_{4},
~{du_{1}}_{|A_{1}}=-\sqrt{-1}dv_{1},~{du_{3}}_{|A_{1}}=-\sqrt{-1}dv_{3}.
\end{align*}
}
\end{lemma}

\begin{lemma}
\allowdisplaybreaks{
\begin{align*}
&{d\{c\}_{i}}_{|A_{1}}\equiv
  \frac{d\{z\}_{i}}{2||A_{1}||}-\frac{\{z\}_{i}+\{\overline{z}_{i}\}}{||A_{1}||^3}\cdot d\xi_{2},\\
\intertext{and for each $i=0,\cdots,7$}
&{d\{c\}_{0}}_{|A_{1}}\equiv \frac{-\sqrt{-1}d\xi_{2}}{2^2},~{d\{c\}_{1}}_{|A_{1}}\equiv \frac{d\xi_{2}-\sqrt{-1}dz_{1}}{2^2},
~{d\{c\}_{2}}_{|A_{1}}\equiv \frac{dz_{2}-dz_{3}}{2^3},~{d\{c\}_{3}}_{|A_{1}}\equiv \frac{dz_{2}+dz_{3}}{2^3\sqrt{-1}},\\
&{d\{c\}_{4}}_{|A_{1}}\equiv \frac{dw_{1}+dw_{4}}{2^3},~{d\{c\}_{5}}_{|A_{1}}\equiv \frac{dw_{1}-dw_{4}}{2^3\sqrt{-1}},
~{d\{c\}_{6}}_{|A_{1}}\equiv \frac{dw_{2}-dw_{3}}{2^3},~{d\{c\}_{7}}_{|A_{1}}\equiv \frac{dw_{2}+dw_{3}}{2^3\sqrt{-1}},\\
&{d\{b\}_{0}}_{|A_{1}}\equiv \frac{dy_{4}-\sqrt{-1}dx_{4}}{2^2},~{d\{b\}_{i}}_{|A_{1}}\equiv
  \frac{d\{y\}_{i}}{2^3}+\sqrt{-1}\frac{d\{x\}_{i}}{2^3},\\
\intertext{where we can ignore the term $\{z\overline{x}\}$, since
  ${\{x\}_{i}}_{|A_{1}}=0$ and
  ${d\{\overline{z}x\}_{i}}_{|A_{1}}\equiv\sum\limits_{{\bf e}_{\alpha}{\bf
  e}_{\beta}={\bf e}_{i}} 
\{\overline{z}\}_{\alpha}d\{x\}_{\beta}={\{\overline{z}\}_{0}d\{x\}_{i}}_{|A_{1}}=-\sqrt{-1}d\{x\}_{i}$
  and for $i=1,\cdots,7$,}
&{d\{b\}_{1}}_{|A_{1}}\equiv \frac{dx_{4}-\sqrt{-1}dy_{4}}{2^2},~{d\{b\}_{2}}_{|A_{1}}\equiv \frac{dy_{2}-\sqrt{-1}dx_{3}}{2^2},
~{d\{b\}_{3}}_{|A_{1}}\equiv \frac{dx_{3}-\sqrt{-1}dy_{2}}{2^2},\\
&
{d\{b\}_{4}}_{|A_{1}}\equiv \frac{dv_{1}+\sqrt{-1}du_{4}}{2^2},~
{d\{b\}_{5}}_{|A_{1}}\equiv -\frac{du_{4}+\sqrt{-1}dv_{1}}{2^2},\\
&
~{d\{b\}_{6}}_{|A_{1}}\equiv \frac{-dv_{3}+\sqrt{-1}du_{2}}{2^2},
~{d\{b\}_{7}}_{|A_{1}}\equiv \frac{du_{2}-\sqrt{-1}dv_{3}}{2^2}.
\end{align*}
}
\end{lemma}
Based on these data
\begin{proposition}
At the point $A_{1}$, the holomorphic component
of the pull-back $\{{\bf q}\circ
{\tau_{\mathbb{O}}}^{-1}\}^{*}(dv_{P^2\mathbb{O}})$ is equal to
\begin{align*}
&\{{\bf q}\circ {\tau_{\mathbb{O}}}^{-1}\}^{*}(dv_{P^2\mathbb{O}})_{|A_{1}}=
\{{\bf q}\circ {\tau_{\mathbb{O}}}^{-1}\}^{*}(d\{c\}_{0}\wedge\cdots\wedge d\{c\}_{7}\wedge
  d\{b\}_{0}\wedge\cdots\wedge d\{b\}_{7})_{|A_{1}}\\
&\equiv \frac{-\sqrt{-1}d\xi_{2}}{2^2}\wedge
\frac{d\xi_{2}-\sqrt{-1}dz_{1}}{2^2}\wedge \frac{dz_{2}-dz_{3}}{2^3}
\wedge\frac{dz_{2}+dz_{3}}{2^3\sqrt{-1}}\wedge\frac{dw_{1}+dw_{4}}{2^3}
\wedge\frac{dw_{1}-dw_{4}}{2^3\sqrt{-1}}\\
&\qquad \wedge\frac{dw_{2}-dw_{3}}{2^3}
\wedge\frac{dw_{2}+dw_{3}}{2^3\sqrt{-1}}
\wedge\frac{dy_{4}-\sqrt{-1}dx_{4}}{2^2}
\wedge\frac{dx_{4}-\sqrt{-1}dy_{4}}{2^2}\\
&\qquad\quad \wedge\frac{dy_{2}-\sqrt{-1}dx_{3}}{2^2}
\wedge\frac{dx_{3}-\sqrt{-1}dy_{2}}{2^2}
\wedge\frac{dv_{1}+\sqrt{-1}du_{4}}{2^2}
\wedge -\frac{du_{4}+\sqrt{-1}dv_{1}}{2^2}\\
&\qquad\qquad\quad \wedge \frac{-dv_{3}+\sqrt{-1}du_{2}}{2^2}
\wedge \frac{du_{2}-\sqrt{-1}dv_{3}}{2^2}\\
&=\frac{1}{2^{31}\sqrt{-1}}\cdot 
dx_{3}\wedge dx_{4}\wedge du_{2}\wedge du_{4}\wedge dy_{2}\wedge dy_{4}\wedge 
dv_{1}\wedge dv_{3}\wedge dz_{1}\wedge dz_{2}\wedge dz_{3}\\
&\qquad\qquad\qquad\qquad\qquad\wedge dw_{1}\wedge dw_{2}\wedge dw_{3}\wedge dw_{4}\wedge d\xi_{2}.
\end{align*}
\end{proposition}
Hence
\begin{corollary}
\begin{align}
&\{{\bf q}\circ
{\tau_{\mathbb{O}}}^{-1}\}^{*}(dv_{P^2\mathbb{O}})\bigwedge 
\overline{\Omega}\,(A)\,
=\,C_{RC}\frac{1}{16!}\{{\tau_{\mathbb{O}}}^{-1}\}^{*}\left(\Big(\omega^{P^2\mathbb{O}}\Big)^{16}\right)\,(A)\notag\\
&=2^{6}\cdot||A||^{3}\cdot\frac{1}{16!}\{{\tau_{\mathbb{O}}}^{-1}\}^{*}\left(\Big(\omega^{P^2\mathbb{O}}\Big)^{16}\right)\,(A).
\label{explicit form of pairing with Riemann volume form}
\end{align}
\end{corollary}

By this formula \eqref{explicit form of pairing with Riemann volume form}
we have an expression of the Bargmann type transformation \eqref{Bargmann type transform}.
\begin{corollary}
\begin{align}
&\mathfrak{B}(h)(X)\cdot dv_{P^2\mathbb{O}}(X)
=\{{\bf q}\circ{\tau_{\mathbb{O}}}^{-1}\}_{*}
\left(
h\cdot g_{0}\cdot 
\{ {\bf q}\circ {\tau_{\mathbb{O}}}^{-1} \}^{*}(dv_{P^2\mathbb{O}})
\bigwedge \overline{\Omega_{\mathbb{O}}}
\right)\notag\\
&=\{{\bf q}\circ{\tau_{\mathbb{O}}}^{-1}\}_{*}
\left(\,h\cdot g_{0}\cdot 2^{6}\cdot
  ||A||^{3}\cdot \frac{1}{16!} \{{\tau_{\mathbb{O}}}^{-1}\}^{*}\left(\big(\omega^{P^2\mathbb{O}}\big)^{16}\right)\right)\notag\\
&={2^6}\cdot {\bf q}_{*}\,
\left(
h(\tau_{\mathbb{O}}(X,*))\cdot e^{-\sqrt{2}\pi||*||} ||*||^{6}\cdot
  dV_{T^{*}(P^2\mathbb{O})}(X,*)
\right).\label{Bargmann type transform by Liouville volume form}
\end{align}
\end{corollary}

\section{{Invariant polynomials and harmonic polynomials on the Jordan algebra $\mathcal{J}(3)$}}

In this section we describe invariant polynomials 
on $\mathcal{J}(3)$ and commuting differential operators with constant
coefficients under the action by the automorphism group $F_{4}$ of the
Jordan algebra $\mathcal{J}(3)$ (see \cite{He} and \cite{HL} for the framework here
and \cite{Yo} for necessary properties of $F_{4}$ in relation with $P^{2}\mathbb{O}$).
% Mostly we follow the framework
% given in \cite{HL} and \cite{He} based on the detailed and explicit determinations
% given in \cite{Yo} for various properties of the group $F_{4}$ and the
% Cayley projective plane $P^2\mathbb{O}$. 

\subsection{Correspondence between polynomials and differential
  operators with constant coefficients}\label{Poly and Diff}

Let 
\[\mathbb{R}^{N}\times \mathbb{R}^{N}\ni (x,\xi)=(x_{1},\ldots,x_{N},\xi_{1}\ldots,\xi_{N})\longmapsto
<x,\,\xi>=\sum \,x_{i}\xi_{i}\in \mathbb{R},
\] 
be the standard non-degenerate symmetric bi-linear form. 
We also use the same notation for its extension to the complex
bi-linear form defined on $\mathbb{C}^{N}\times\mathbb{C}^N$.

Differential operators $D_{x}$  with constant (complex) coefficients
are expressed in the form 
\[
D=D_{x}=\sum\limits_{|\alpha|\leq k}\,
a_{\alpha}\,\frac{\partial^{|\alpha|}}{\partial x^{\,\alpha}}=\sum a_{\alpha}\,D^{\,\alpha}_{x},
\]
where
$a_{\alpha}\in\mathbb{C}$ and 
%$\alpha=(\alpha_{1},\ldots,\alpha_{N})$ being multi-indices, 
$\displaystyle{
D^{\,\alpha}_{x}:=\frac{\partial^{\,|\alpha|}}
{\partial x^{\,\alpha}}=\frac{\partial^{\,\alpha_{1}+\cdots+\alpha_{N}}\qquad}{\partial{x_{1}}^{\alpha_{1}}\cdots\partial{x_{N}}^{\alpha_{N}}}}.
$

Let $D=\sum a_{\alpha}\,D^{\,\alpha}_{x}$ be a constant coefficient
partial differential operator defined on $\mathbb{R}^{N}$, then
by the relation
\begin{equation}\label{differential operator and polynomial}
e^{-<x,\xi>}D_{x}(e^{<\bullet,\xi>})(x)=\sum\, a_{\alpha}\,\xi^{\alpha}:=Q^{D}(\xi),
\end{equation}
where $\xi^{\alpha}={\xi_{1}}^{\alpha_{1}}\cdots{\xi_{N}}^{\alpha_{N}}$, 
the correspondence $D\longleftrightarrow Q^{D}(\xi)$
is bijective, that is, 
the algebra
$\mathbb{C}[\mathbb{R}^{N}]=\mathbb{C}[x_{1},\ldots,x_{N}]=\sum\limits_{k=0}^{\infty} \mathcal{P}_{k}[x_{1},\ldots,x_{N}]$ 
of (complex coefficient) polynomials on $\mathbb{R}^{N}$
and the algebra
$\mathcal{D}[x_{1},\ldots,x_{N}]\break=\sum\limits_{k=0}^{\infty}\,\mathcal{D}_{k}[x_{1},\ldots,x_{N}]$ 
of linear differential operators with constant (complex) coefficients
are isomorphic. Here we denote by $\mathcal{P}_{k}[x_{1},\ldots,x_{N}]$ the subspace of homogeneous
polynomials of degree $k$ and
by $\mathcal{D}_{k}=\mathcal{D}_{k}[x_{1},\ldots,x_{N}]$ the subspace consisting of
homogeneous differential operators with constant coefficients of order $k$.

{\it We will denote the differential operator corresponding
to a polynomial $Q\in\mathbb{C}[x_{1},\ldots,x_{N}]$ by $D^{Q}$. }
%The correspondence from a polynomial
%$Q=Q(x)=\sum a_{\alpha}x^{\alpha}$ to a differential operator $D=D^{Q}$ is understood that
%in the expression of the polynomial $Q$
%we replace the monomial 
%$x^{\alpha}$ by the differentiation 
%$\displaystyle{\frac{\partial^{|\alpha|}}{\partial x^{\alpha}}}=D^{\,\alpha}_{x}$.
%\smallskip

Let $g\in \GL(N,\mathbb{R})$ and define 
$\mathcal{P}_{g}:\mathbb{C}[x_{1},\ldots,x_{N}]\longrightarrow\mathbb{C}[x_{1},\ldots,x_{N}]$
an algebra isomorphism in the natural way: 
\begin{align*}
&Q=Q(x)=\sum\, a_{\alpha}\cdot x^{\alpha}\,\longmapsto
\mathcal{P}_{g}(Q)(x)=Q(g^{-1}(x))=
  \sum \,a_{\alpha}\cdot (g^{-1}(x))^{\alpha},~~\text{where}\\%\in\mathbb{C}\,[x_{1},\cdots,x_{N}],~\text{where}\\
&g^{-1}=\Big(\,\, \{g^{-1}\}_{i,j}\,\Big),~\text{and}~\,\,
(g^{-1}(x))^{\alpha}=\left(\sum_{i} \,\{g^{-1}\}_{1,i}\,x_{i}\right)^{\alpha_{1}}\cdots\left(\sum_{i}\,\{g^{-1}\}_{N,i}\,x_{i}\right)^{\alpha_{N}}.
\end{align*}
The following relation will be seen easily.
\begin{lemma}\label{commutativity with linear transformation}
Let $g\in \GL(N,\mathbb{R})$ and $D$ a linear differential operator with constant coefficients. 
Then, $\mathcal{P}_{g}\circ D=D\circ \mathcal{P}_{g}$ on the space of the 
whole polynomial functions, if and only if $Q^{D}(\xi)=Q^{D}({^{t}g}^{-1}(\xi))$.
\end{lemma}

%\begin{proof}
%If $\mathcal{P}_{g}\circ D=D\circ\mathcal{P}_{g}$ on the space of the whole polynomial
%functions, then by the Weierstrass polynomial approximation theorem
%this commutativity holds on the space 
%$C^{\infty}(\mathbb{R}^{N})$ of all smooth functions.%
%
%Then
%\begin{align*}
%&\mathcal{P}_{g}\circ D(e^{<\bullet,\xi>})(x)=D(e^{<\bullet,\xi>})(g^{-1}(x))=(e^{<g^{-1}(x),\xi>})Q_{D}(\xi),\\
%\intertext{and}
%&\left(D\circ \mathcal{P}_{g}(e^{<\bullet,\xi>})\right)(x)=
%D(e^{<g^{-1}(\bullet),\xi>})(x)=D(e^{<\bullet,^{t}{g^{-1}}(\xi)>})(x)=e^{<x,^t{g^{-1}}(\xi)>}Q_{D}(^{t}g^{-1}(\xi)).
%\end{align*}
%Hence
%\[
%e^{<g^{-1}(x),\xi>}Q_{D}(\xi)=e^{<x,^{t}g^{-1}(\xi)>}Q_{D}({^t{g^{-1}}}(\xi))
%\]
%and
%we have
%the desired relation $Q_{D}(\xi)=Q_{D}(^{t}g^{-1}(\xi))$. 
%
%Conversely, from the relation $Q_{D}(\xi)=Q_{D}(^t{g^{-1}}(\xi))$
%we can also prove the commutativity $\mathcal{P}_{g}\circ D=D\circ\mathcal{P}_{g}$ by
%going back the argument above.
%\end{proof}

Next, we introduce an Hermitian inner product $\ll\cdot\,,\,\cdot\gg$ on the space of polynomials
$\mathbb{C}[x_{1},\ldots,x_{N}]$ by the following way: 

we fix the  coordinates $(x_{1},\cdots,x_{N})\in\mathbb{R}^{N}$
and define the inner product between 
monomials $x^{\alpha}$ and $x^{\beta}$ by
\begin{equation}\label{inner product in polynomial space}
\ll\,x^{\alpha}\,,\,x^{\beta}\,\gg :=\alpha_{1}!\cdots\alpha_{N}!\cdot
\delta_{\alpha_{1},\beta_{1}}\cdots
\delta_{\alpha_{N},\beta_{N}}:=\alpha !\cdot\delta_{\alpha,\beta}.
\end{equation}
%Hence if $\alpha\not=\beta$, then monomials $x^{\alpha}$ and $x^{\beta}$ are orthogonal.
%So two monomials of different degrees are also orthogonal.
%\smallskip

{\it The inner product on the space $\mathbb{C}[x_{1},\ldots,x_{N}]$ introduced above
 is used only in this section.}
\smallskip

Then the following properties will be seen easily too.
\begin{lemma}\label{properties of the inner product in the polynomial algebra}
Let $D=\sum a_{\alpha}D^{\,\alpha}_{x}$
be a differential operator with constant coefficients 
and $P$ the polynomial corresponding to $D$ according to the
correspondence \eqref{differential operator and polynomial}, that is
\[
e^{-<x,\,\xi>}D^{P}(e^{<\bullet,\,\xi>})(x)=P(\xi)=\sum \,a_{\alpha}\,{\xi}^{\alpha}.
\] 
Then for any polynomial $Q=\sum\,C_{\beta}x^{\beta}$
\begin{align}\label{duality of inner product}
&\ll\,P,\, Q\,\gg =\ll\,\sum a_{\alpha}\,x^{\alpha}\,,\, \sum\,{C}_{\beta}\,x^{\beta}\,\gg
=\sum_{\alpha}\,\alpha! \cdot\delta_{\alpha,\beta} \cdot
  a_{\alpha}\cdot \overline{C_{\beta}}\\
&=D^{P}(\overline{Q})(0)=:(\!(D^{P},\,Q)\!)=\overline{(\!( D^{Q}\,,\,P )\!)},
\end{align}
where we replace the variables $\xi_{i}$ to $x_{i}$.
In particular, if the order of the differential operator $D$ and the
degree of a
polynomial $Q$ coincides, then
\[
D(\overline{Q})(x)\equiv D(\overline{Q})(0),
\]
and the spaces $\mathcal{P}_{k}$ and $\mathcal{P}_{\ell}$ are always orthogonal, if
$k\not=\ell$.

It holds a kind of associativity:
\begin{align}
(\!(D_{1}\circ D_{2},\,Q)\!)=D_{1}\circ D_{2}(Q)(0)
=\ll P_{1}\cdot P_{2},\,Q\gg=\ll\,P_{1}, D_{2}(Q)\,\gg
=(\!(D_{1},\,D_{2}(Q))\!).\label{associativity of inner product}
\end{align}
\end{lemma}

The equation \eqref{duality of inner product}
can be understood that the Hermitian inner product we introduced is a
pairing between the space $\mathcal{D}[x_{1},\ldots,x_{N}]$ of differential
operators with constant coefficients and the space of polynomials,
especially by this pairing 
the space $\mathcal{D}(\mathbb{R}^{N})$ is identified with the
(restricted) dual space $\mathcal{D}[x_{1},\ldots,x_{N}]\cong 
\sum\limits_{k=0}^{\infty}\,{\mathcal{P}_{k}[x_{1},\ldots,x_{N}]}^{*}$ of $\mathbb{C}[x_{1},\ldots, x_{N}]$. With
respect to the action of $g\in \GL(N,\mathbb{R})$ on $\mathbb{C}[x_{1},\ldots,x_{N}]$
the dual action of $g$ on $\mathcal{D}(\mathbb{R}^{N})$ is 
\begin{equation*}
\mathcal{D}(\mathbb{R}^{N})\ni D\longmapsto 
{\mathcal{P}_{g}}^{-1}\circ D\circ \mathcal{P}_{g} =:\mathcal{P}_{g}^{*}(D)
\end{equation*}
and satisfies the relation
\begin{equation}\label{dual action on differential operators}
(\!(\,\mathcal{P}_{g}^{*}(D),\, f\, ){\!)}=({\!(}\,D,\,\mathcal{P}_{g}(f) ){\!)}, 
~D\in\mathcal{D}[x_{1},\ldots,x_{N}],~f\in\mathbb{C}[x_{1},\ldots,x_{N}].
\end{equation}

\subsection{Trace function and invariant polynomials}

We recall two
important properties Theorems \ref{basic theorem 8.1} and \ref{irreducible decomposition}
on the action of the group $F_{4}$ on $\mathcal{J}(3)$.
Also the properties \eqref{invariance of trace}, \eqref{invariance of inner product}
should be reminded in this section (see \cite{SV} and \cite{Yo}).

\begin{theorem}\label{basic theorem 8.1}
For any $A\in\mathcal{J}(3)$, there exists an element $\alpha\in
F_{4}$ such that
\begin{equation}\label{diagonalization by F4}
\alpha(A)=\begin{pmatrix}\xi_{1}&0&0\\0&\xi_{2}&0\\0&0&\xi_{3}\end{pmatrix},
\end{equation}
where the triple of quantities $\{\xi_{i}\}$ depends only on $A$ and
does not depend on such an element $\alpha\in F_{4}$ which send $A$ to
a diagonal matrix in $\mathcal{J}(3)$. 
\end{theorem}
\begin{theorem}\label{irreducible decomposition}
The representation of $F_{4}$ to $\mathcal{J}(3)$
is decomposed into two mutually orthogonal irreducible subspaces, that is
\[
\mathcal{J}(3)=\mathcal{J}_{0}(3)\oplus \mathbb{R}\cdot Id,
\]
where $\mathcal{J}_{0}(3)=\{~A\in\mathcal{J}(3)\,|~{\text{\em tr}\,(A)=0}~\}$ and
$Id$ is the $3\times 3$ identity matrix which 
is the fixed point in $\mathcal{J}(3)$ under the action of $F_{4}$.

It holds the same decomposition in the complexified Jordan
algebra $\mathcal{J}(3)^{\mathbb{C}}$ by the action of the complex
group ${F_{4}}^{\mathbb{C}}$.
\end{theorem}

{\it In this section we express 
\[
\mathcal{J}(3)\ni A=\begin{pmatrix}
\xi_{1}&z&\theta(y)\\
\theta(z)&\xi_{2}&x\\
y&\theta(x)&\xi_{3}
\end{pmatrix}
\longleftrightarrow (z_{0},\cdots,z_{7},y_{0},\cdots,y_{7},x_{0},\cdots,x_{7},\xi_{1},\xi_{2},\xi_{3})
\]
with the coefficients of 
\[
z=\sum \,\{z\}_{i}{\bf e}_{i}=\sum \,z_{i}{\bf e}_{i},~
y=\sum \,\{y\}_{i}{\bf e}_{i}=\sum\, y_{i}{\bf e}_{i},~
x=\sum \,\{x\}_{i}{\bf e}_{i}=\sum\, x_{i}{\bf e}_{i}
\]
and do not use the notation $\{z\}_{i}$, $\{y\}_{i}$, $\{x\}_{i}$
 {\em(see \eqref{vector expression of A})}.
We also denote these coordinates as
\begin{align}
&(z_{0},\cdots,z_{7},y_{0},\cdots,y_{7},x_{0},\cdots,x_{7},\xi_{1},\xi_{2},\xi_{3})
=(s_{1},\cdots\cdots,s_{24},s_{25},s_{26},s_{27})\label{coordinate 1}\\
\intertext{or}
&
(z_{0},\cdots,z_{7},y_{0},\cdots,y_{7},x_{0},\cdots,x_{7},\xi_{1},\xi_{2},\xi_{3})
=(s_{1},\cdots\cdots,s_{24},\xi_{1},\xi_{2},\xi_{3}).
\label{coodinates 2}
\end{align}
}
 
We denote the (complex valued) polynomial algebra over 
$\mathcal{J}(3)$ by $\mathbb{C}[\,\mathcal{J}(3)\,]$ with the independent variables
$\{z_{i},\ldots, x_{i},\xi_{1},\xi_{2},\xi_{3}\}$ and also regard it as the
the algebra of polynomial functions.
 %on the Jordan algebra $\mathcal{J}(3)$ too. 
 It is equipped
with an Hermitian inner product explained 
in the preceding subsection $\S$ \ref{Poly and Diff}. 

Then, we can identify by the isometric way the space
$\mathcal{P}_{1}[\mathcal{J}(3)]\cong\big(\mathcal{J}(3)^{\mathbb{C}}\big)^{*}$ with the space
$\mathcal{J}(3)^{\mathbb{C}}$
through the correspondence
\begin{equation}\label{identification with the dual space}
\mathcal{J}(3)^{\mathbb{C}}\ni A\longleftrightarrow
h_{A}\in \mathcal{P}_{1}[\mathcal{J}(3)],~h_{A}(X)=\text{tr}\,(X\circ A) = <X\,,A>^{\mathcal{J}(3)^{\mathbb{C}}}.
\end{equation}

The action of the group $F_{4}$ is extended to the space
$\mathbb{C}[\mathcal{J}(3)]$ as denoted in $\S 8.1$ :
\[\mathbb{C}[\mathcal{J}(3)]\ni Q\longmapsto (\mathcal{P}_{g}(Q)(X):=Q(g^{-1}(X))
\]
and the extended action leaves the degree of the polynomials and the inner product. 

\begin{definition}\label{invariant polynomial space}
We denote a subspace in each $\mathcal{P}_{k}[\mathcal{J}(3)]$ by $I_{k}$
consisting of invariant polynomials 
under the extended action of the group $F_{4}$ and put
$I=I_{F_{4}}=\sum_{k\geq 0}\,I_{k}$,
the algebra of invariant polynomials under the action of the Lie group
$F_{4}$ on $\mathcal{J}(3)$. 
\end{definition}
By the property \eqref{selfadjointness of Jordan product in the inner product},
the functions $\mathcal{J}(3)\ni A\longmapsto \text{tr}\,(A^{k}):=T_{k}(A)$
is well-defined and are invariant polynomials
({\it off course, these are also well defined on $\mathcal{J}(3)^{\mathbb{C}}$}).
Then,
\begin{proposition}
All the invariant polynomials in $\mathcal{P}_{k}[\mathcal{J}(3)]$ 
are given by the linear sums of polynomials of the products
\[
{T_{1}}^{i_{1}}\cdot {T_{2}}^{i_{2}}\cdot {T_{3}}^{i_{3}}
\] 
under the condition that $i_{1}+2i_{2}+3i_{3}=k$ {\em ($0\leq
i_{1}, i_{2},i_{3}\leq k$)} and 
\begin{align}
&\dim_{\mathbb{C}}\,I_{k}=\{\text{number of the
  solutions}~(i_{1},i_{2}, i_{3})~\text{under the condition}~ i_{1}+2i_{2}+3i_{3}=k\}\notag\\
&=\sum_{\ell=0}^{[k/3]}\,\left(\left[\frac{k-3\ell}{2}\right]+1\right).\label{dim of Ik}
\end{align}
\end{proposition}
\begin{proof}
Let $f\in I_{k}$ be an invariant polynomial. Then
by the property \eqref{diagonalization by F4} in Theorem \ref{basic theorem 8.1} 
and the invariance of
the trace function \eqref{invariance of trace}, 
$f(A)=f(\alpha(A))=f\left(\begin{pmatrix}\xi_{1}&0&0\\0&\xi_{2}&0\\0&0&\xi_{3}\end{pmatrix}\right)$ 
depend only on the triple $\{\xi_{i}\}_{i=1}^{3}$ which appears when it is 
expressed as a diagonal matrix
given in the above Theorem \ref{basic theorem 8.1}.

Let $\sigma_{1}:\mathcal{J}(3)\to\mathcal{J}(3)$ be a permutation defined by
\begin{equation}
\sigma_{1}:\mathcal{J}(3)\ni A\longmapsto 
\begin{pmatrix}0&~1&~0\\1&~0&~0\\0&~0&~1\end{pmatrix}
A\begin{pmatrix}0&~1&~0\\1&~0&~0\\0&~0&~1\end{pmatrix}\in\mathcal{J}(3).
\end{equation}
Likewise we can define other two permutations $\sigma_{2}$ and $\sigma_{3}$ among
the quantities $\{\xi_{i}\}$ 
by the matrices 
$\begin{pmatrix}0&~0&~1\\0&~1&~0\\1&~0&~0\end{pmatrix}$
and $\begin{pmatrix}1&~0&~0\\0&~0&~1\\0&~1&~0\end{pmatrix}$, respectively. These are
elements in $F_{4}$ and satisfy $f(\sigma_{i}(A))=f(A)$. 

Hence the invariant polynomial ring 
$I=I_{F_{4}}=\sum\limits_{k\geq 0}\,I_{k}$ in $\mathbb{C}[\mathcal{J}(3)]$
is generated by three elementary symmetric polynomials
\[
\xi_{1}+\xi_{2}+\xi_{3},
\quad \xi_{1}\xi_{2}+\xi_{2}\xi_{3}+\xi_{3}\xi_{1}\quad \text{and}\quad \xi_{1}\xi_{2}\xi_{3}.
\]
This is equivalent to say that the subalgebra of invariant polynomials of positive degree
$I_{+}:=\sum_{k\geq 1}\,I_{k}$ (i.e., {\it without constant terms}) is generated by three
invariant polynomials
\begin{align*}
&T_{1}(A)=\text{tr}\,(A)=\sum_{i=1}^{3}\,\xi_{i}, \qquad
  T_{2}(A)=\text{tr}\,(A^2)=2\sum_{i=0}^{7}\,
\left({z_{i}}^2+{y_{i}}^2+{x_{i}}^2\right)+\sum_{i=1}^{3}\,{\xi_{i}}^2=||A||^2,\\
&T_{3}(A)=\text{tr}\,(A\circ (A\circ A))=<A,A\circ
  A>^{\mathcal{J}(3)}=<A\circ A\,, A>^{\mathcal{J}(3)}\\
&\quad
=\sum_{i=1}^{3}\,{\xi_{i}}^3+3\left(|z|^2(\xi_{1}+\xi_{2})+|y|^2(\xi_{3}+\xi_{1})+|x|^2(\xi_{2}+\xi_{3})\right)\\
&\qquad\qquad\qquad
+\frac{zx\cdot y
+\theta(zx\cdot y)
+xy\cdot z
+\theta(xy\cdot z)
+yz\cdot x
+\theta(yz\cdot x)}{2}\\
&\qquad\qquad\qquad\quad+\frac{x\cdot yz
+\theta(x\cdot yz)
+y\cdot zx
+\theta(y\cdot zx)
+z\cdot xy
+\theta(z\cdot xy)}{2}\\
&=\sum_{i=1}^{3}\,{\xi_{i}}^3+3\left(|z|^2(\xi_{1}+\xi_{2})+|y|^2(\xi_{3}+\xi_{1})+|x|^2(\xi_{2}+\xi_{3})\right)
+6\cdot \mathfrak{Re}(x\cdot yz),\\
&\text{($\mathfrak{Re}(x\cdot yz)=\{x\cdot yz\}_{0}$ is the real
  part of the octanion $x\cdot yz$)}.
\end{align*}
The last formula \eqref{dim of Ik} is given by solving the
equation $i_{1}+2\cdot i_{2}+3\cdot i_{3}=k$ (see Appendix).
\end{proof}

In the proof above we used a property of the multiplication low in 
the octanion $\mathbb{O}$: 
\[
\mathfrak{Re}(x\cdot yz)=\mathfrak{Re}(y\cdot zx)=\mathfrak{Re}(z\cdot xy),~
\mathfrak{Re}(zx\cdot y)=\mathfrak{Re}(\theta(y)\cdot\theta(zx))=\mathfrak{Re}(y\cdot zx)
~\text{and similar identities}.
\]
%For example a basis of  $I_{4}$ is given by %lower degrees, we list a basis of $I_{k}$:
%\begin{align*}
%&I_{1}=[\{T_{1}\}],~\dim\,I_{1}=1,\\
%&I_{2}=[\{T_{2}, ~{T_{1}}^2\}],~\dim\,I_{2}=2,\\
%&I_{3}=[\{T_{3}, ~T_{2}T_{1}, ~{T_{1}}^{3}\}],~\dim\,I_{}=3,\\
%&I_{4}=[\{T_{3}T_{1},~{T_{2}}^2, ~T_{2}{T_{1}}^2, ~{T_{1}}^{4}\}],~\dim\,I_{4}=4.%\\
%&I_{5}=[\{T_{3}T_{2}, ~T_{3}{T_{1}}^2, ~{T_{2}}^2T_{1},~T_{2}{T_{1}}^3,~{T_{1}}^{5}\}],~\dim\,I_{5}=5,\\
%&I_{6}=[\{{T_{3}}^2, ~T_{3}T_{2}T_{1}, ~T_{3}{T_{1}}^{3},
%  {T_{2}}^{3},~{T_{2}}^{2}{T_{1}}^2, ~T_{2}{T_{1}}^{4}, ~{T_{1}}^{6}\}],~\dim\,I_{6}=7.
%\end{align*}

Next, we mention (see 
Lemma \ref{commutativity with linear transformation} and Theorem \ref{basic theorem 8.1} ) 
%a relation between the invariant polynomials and
%commuting differential operators with respect to  the $F_{4}$ action. That is,
%by the lemma \ref{commutativity with linear transformation} and the fact that the group
%$F_{4}$ is closed under the transpose operation with respect to the
%inner product in $\mathcal{J}(3)$, we have
\begin{proposition}
The invariant polynomial ring $I=\sum I_{k}$ and differential
operators with constant coefficients commuting with the $F_{4}$ action
are isomorphic. Especially,
the differential operators corresponding to the generators $T_{1}, T_{2}$
and $T_{3}$ of the invariant polynomial ring are
\begin{align*}
&T_{1}=e^{-<x,\,\xi>}L(e^{<\bullet,\,\xi>})(x)\longleftrightarrow 
L=L(z,y,x,\xi_{1},\xi_{2},\xi_{3}):=\frac{\partial}{\partial\xi_{1}}+\frac{\partial}{\partial\xi_{2}}
+\frac{\partial}{\partial\xi_{3}},\\
&T_{2}\longleftrightarrow 
-\Delta:=2\sum\limits_{i=0}^{7}\,\left(\frac{\partial^2}{\partial{z_{i}}^2}+
+\frac{\partial^2}{\partial{y_{i}}^2}
+\frac{\partial^2}{\partial{x_{i}}^2}\right)
+\sum\limits_{j=1}^{3}
\frac{\partial^2}{\partial {\xi_{j}}^2},\\
&T_{3}\longleftrightarrow 
\Gamma:=\sum\,\frac{\partial^3}{\partial {\xi_{j}}^{3}}
+3\left(\frac{\partial}{\partial\xi_{1}}+\frac{\partial}{\partial\xi_{2}}\right)\circ
\sum\limits_{i=0}^{7}\,\frac{\partial^2}{\partial{z_{i}}^2}
+3\left(\frac{\partial}{\partial\xi_{3}}+\frac{\partial}{\partial\xi_{1}}\right)\circ
\sum\limits_{i=0}^{7}\,\frac{\partial^2}{\partial{y_{i}}^2}\\
&\qquad\qquad\qquad\qquad\qquad +3\left(\frac{\partial}{\partial\xi_{2}}+\frac{\partial}{\partial\xi_{3}}\right)
\circ\sum\limits_{i=0}^{7}\,\frac{\partial^2}{\partial{x_{i}}^2}+
6\sum\limits_{i,j,k=0}^{7}\,\pm\frac{\partial^3\quad\quad}{\partial{x_{i}}\partial{y_{j}}\partial{z_{k}}}.
\end{align*}
%All such operators are generated by these three operators.

The second operator is the Laplacian on the Euclidean space $\mathbb{R}^{27}\cong\mathcal{J}(3)$.

The last term of the operator $\Gamma$ consists of $8^{3}$ partial
differential operators of the form 
$\displaystyle{\frac{\partial^3\quad\quad}{\partial{x_{i}}\partial{y_{j}}\partial{z_{k}}}}$
with suitable signs. 
\end{proposition}

We define an $F_{4}$-invariant subspace $H_{k}$ in $\mathcal{P}_{k}[\mathcal{J}(3)]$ inductively
and call polynomials therein ``{\it Cayley harmonic polynomial}\,''.

%{\large\color{red}From the next definition to the end of the proof of Proposition \ref{direct sum}
%we temporarily denote $\mathcal{P}_{k}$ instead of $\mathcal{P}_{k}[\mathcal{J}(3)]$ for similicity. 
%}
\begin{definition}\label{Cayley harmonic polynomial}
\[
\begin{array}{l}
{\text{\em (0)}}~ H_{0} ~\text{is the space of the constant functions} = \mathcal{P}_{0},\\
{\text{\em (1)}}~ H_{1}=\{\text{the linear functions:}
\sum\limits_{i=0}^{7} (a_{i}z_{i}+b_{i}y_{i}+c_{i}x_{i})+\sum
\limits_{i=1}^{3}d_{i}\xi_{i}~\text{with}~\sum \,d_{i}=0\},\\
\qquad\text{this space is isomorphic to}~\left\{B\in\mathcal{J}(3)^{\mathbb{C}}~|~\text{\em tr}\,(B)=0~\right\},\\
\qquad\text{and we have an orthogonal decomposition}~\mathcal{P}_{1}=H_{1}\oplus_{\perp}H_{0}I_{1},\\
%{\text{\em (2)}}~H_{2}~\text{is the orthogonal complement of}
%~~H_{0}\cdot [\left\{T_{2}+{T_{1}}^2\right\}]+H_{1}\cdot T_{1}=H_{0}\cdot T_{2}+\mathcal{P}_{1}\cdot T_{1},\\
\qquad\qquad\qquad \vdots\\
{\text{\em (k)}}~H_{k}~\text{is the orthogonal complement of the space}
~\,\left(\sum_{i=0}^{k-1}\, H_{i}\cdot I_{k-i}\right)~\text{taken in $\mathcal{P}_{k}$},
\end{array}
\]
\begin{equation}\label{decomposition of Pk}
\mathcal{P}_{k}[\mathcal{J}(3)]=H_{k}\oplus_{\perp}\sum_{i=1}^{k}\, H_{k-i}\cdot I_{i}.
\end{equation}
\end{definition}
The subspace $H_{k}$ can be seen as the space corresponding to the
space of harmonic polynomials
for the case of $SO(n)$ acting on $\mathbb{R}^{n}$. In fact
\begin{proposition}\label{orthogonal complement and harmonic polynomial}
Let $\mathfrak{H}_{k}$ be
\[
\mathfrak{H}_{k}:=\{Q\in \mathcal{P}_{k}[\mathcal{J}(3)]~|~L(Q)=0, \Delta(Q)=0, \Gamma(Q)=0\}.
\]
Then $\mathfrak{H}_{k}=H_{k}$.
\end{proposition}
%This characterization says that our space of Cayley harmonic polynomials is
%a subspace of the space of usual harmonic polynomials that hereafter we just
%call them simply Cayley-harmonic polynomials.
Before proving this Proposition we show a
\begin{lemma}\label{two expressions of the complement of harmonic space}
For each $k$ the space 
\begin{equation}\label{harmonics and polynomial}
I_{k}+{H_{1}\cdot I_{k-1}}+\cdots+ H_{k-1}\cdot I_{1}=I_{k}+\mathcal{P}_{1}\cdot I_{k-1}+\cdots +\mathcal{P}_{k-1}\cdot I_{1}.
\end{equation}
The right hand side need not be a direct sum, while the left hand side
is a direct sum which will be proved later after several preparations.
\end{lemma}
\begin{proof}
It is apparent
\[
\sum_{i=0}^{k-1}\, H_{i}\cdot I_{k-i} \,\,\subset\,\, \sum \mathcal{P}_{i}\cdot I_{k-i}.
\]
Since $I_{2}\supset I_{1}\cdot I_{1}$,
\[
H_{1}\cdot I_{1}+I_{2}\supset \left(H_{1}\cdot I_{1}+I_{1}\cdot I_{1}\right)+I_{2}= \mathcal{P}_{1}\cdot I_{1}+I_{2}.
\]
Hence, $H_{1}\cdot I_{1}+I_{2}=\mathcal{P}_{1}\cdot I_{1}+I_{2}$.

Assume 
\[
\sum_{i=0}^{j-1}\, H_{i}\cdot I_{j-i} \,\,=\,\,\sum\limits_{i=0}^{j-1} \mathcal{P}_{i}\cdot I_{j-i},~\text{for}~j\leq k.
\]
Then using the property $I_{j}\supset I_{a}\cdot I_{b}$ for $a+b=j$,
we can show inductively  
\[
\sum_{i=0}^{j-1}\, H_{i}\cdot I_{j-i} \,\,\supset\,\,\sum\limits_{i=0}^{j-1} \mathcal{P}_{i}\cdot I_{j-i},~\text{for}~j\leq k.
\]
%For example,
%\begin{align*}
%&I_{k+1}+ H_{1}\cdot I_{k} + H_{2}\cdot I_{k-1}+\cdots + H_{k}\cdot
 % I_{1}\\
%&\supset
%I_{k}\cdot I_{1}+ H_{1}\cdot I_{k-1}\cdot I_{1} + H_{2}\cdot
%I_{k-2}\cdot I_{1}+\cdots+H_{k}\cdot I_{1}\\ 
%&\supset 
%\left(I_{k}+H_{1}\cdot I_{k-1}+\cdots+ H_{k}\right)\cdot I_{1}=
%\mathcal{P}_{k}\cdot I_{1}.
%\end{align*}
\end{proof}
\smallskip
{\bf Proof of the Proposition} \ref{orthogonal complement and harmonic polynomial}.

At this stage, it will be apparent that the conditions
\[\Gamma(Q)=0,~\quad \Delta(Q)=0,\qquad\text{and}~\,L(Q)=0\]
are together equivalent to the condition that
a polynomial $Q\in \mathcal{P}_{k}$ is orthogonal to the subspace 
\[
I_{k}+ H_{1}\cdot I_{k-1}+\cdots+H_{k-1}\cdot I_{1}.
\]
\hfill{$\Box$}

\begin{lemma}\label{action of invariant operator to invariant polynomial}
\begin{align*}
&L(T_{1})=3,~\quad L(T_{2})=2T_{1},~\quad L(T_{3})=3T_{2},\\
&\Delta(T_{2})=198,~\quad \Delta(T_{3})=198T_{1},~\qquad \Gamma(T_{3})=562.
\end{align*}
\end{lemma}

\begin{remark}
Invariant polynomials above need not be orthogonal. For example 
\[
\ll\,T_{2}\,,\, {T_{1}}^{2}\,\gg \,=\, L^2(T_{2})(0)=L(L(T_{2}))(0)=2L(T_{1})(0)=6.
\]
\end{remark}
Next we show that the sum \eqref{harmonics and polynomial} is a direct sum
as mentioned in the above Lemma \ref{two expressions of the complement
  of harmonic space}(summed up in Proposition \ref{direct sum}).

By definition it is enough to show the sum
\[
H_{k-1}\cdot I_{1}+\cdots + H_{1}\cdot I_{k-1}+ I_{k}
\]
is a direct sum. For this purpose we prepare several lemmas.

\begin{lemma}
The map $L:I_{k}\longrightarrow I_{k-1}$ is surjective for all
$k=1,~2,~\cdots, $
\end{lemma}
\begin{proof}
Let ${\bf t}:I_{k}\longrightarrow I_{k+1}$ be a map defined by
\[
{\bf t}(T)=T_{1}\cdot T,
\]
then ${\bf t}$ is injective. In fact, 
if there is an element $T\in I_{k}$ satisfying 
\[L\circ {\bf
  t}(T)=L(T_{1}\cdot T)=3T+T_{1}\cdot L(T)=0,
\] 
then again we have 
\[
3L(T)+3L(T)+T_{1}\cdot
L^2(T)=0~\text{and}~3T-\frac{1}{6}T_{1}^{2}\cdot L^{2}(T)=0.
\]
By iterating this procedure we have
\[
T=0.
\]
Hence the map $L\circ {\bf t}$ is injective, which means that the map
$L:I_{k+1}\to I_{k}$ is already surjective (in fact isomorphic) on ${\bf t}(I_{k})$.
\end{proof}

Based on the equality \eqref{harmonics and polynomial}
and the Lemma below \eqref{an orthogonality},
we can construct an orthogonal basis of the space $I_{k}$ inductively 
$\{\varphi_{k}(i)\}_{i=1}^{\dim I_{k}}$ in the following way:
\begin{definition}\label{special basis}
\begin{align*}
I_{1}=&[\{\varphi_{1}(1)=T_1\}],\\
I_{2}=&[\{\varphi_{2}(1)=T_1^{2}=T_{1}\cdot \varphi_{1}(1),~\varphi_{2}(2)=T_{2}-1/3T_{1}^2\}],\\
&\text{where $\varphi_{2}(2)$ is taken to be orthogonal to $\varphi_{2}(1)$ and equivalently $L(\varphi_{2}(2))=0$},\\
I_{3}=&[\{\varphi_{3}(1)=T_1^{3}=T_{1}\varphi_{2}(1),~\varphi_{3}(2)=T_{1}\varphi_{2}(2),
~\varphi_{3}(3)=T_{3}-T_{2}T_{1}+2/9T_{1}^{3}\}],\\
&\text{where $\varphi_{3}(3)$ is taken
  to be orthogonal to $\varphi_{3}(1)$ and $\varphi_{3}(2)$,}
  ~\text{which is also taken}\\ 
&\text{to satisfy $L(\varphi_{3}(3))=0$ and is determined uniquely up to constant multiples}.
%I_{4}=&[\{\varphi_{4}(1)=T_1^{4}=T_{1}\varphi_{3}(1),~\varphi_{4}(2)=T_{1}\varphi_{3}(2),
%~\varphi_{4}(3)
%=T_{1}\varphi_{3}(3)\\
%&\varphi_{4}(4)=T_{2}^2-2/3T_{2}T_{1}+1/9T_{1}^{4}\}],~
%\text{where}~\varphi_{4}(4)~\text{is taken to satisfy
 % $L(\varphi_{4}(4))=0$}\\
%&\text{and is equal to $\varphi_{4}(4)={\varphi_{2}(2)}^{2}$}.
\end{align*}
Likewise we can continue the construction in such a way that
if $\{\varphi_{k}(i)\}_{i=1}^{\dim I_{k}}$ is constructed as above for
$k=1,2,3,4$, then we define for $k\geq 5$
\begin{align*}
&\varphi_{k+1}(i)=T_{1}\varphi_{k}(i)~\text{for $i=1\cdots,\,\dim
  I_{k}$}~\text{and for $j=1,\cdots,\dim I_{k+1}-\dim I_{k}$},\\
&\varphi_{k+1}(\dim I_{k}+j) 
~\text{is chosen as being orthogonal to all $\varphi_{k+1}(i)$, $i=1,\cdots, \,\dim I_{k}+j-1$}.
\end{align*}
The orthogonality condition $\ll\,\varphi_{k+1}(\dim I_{k}+j)\,,\,T_{1}\varphi_{k}(i)\gg=0$
implies that $L(\varphi_{k+1}(\dim I_{k}+j))=0$.
\end{definition}

\begin{lemma}\label{an orthogonality}
The construction is guaranteed by the property that 
if $f$ and $g\in I_{k}$ is orthogonal and $L(g)=0$, then $T_{1}\cdot
f$ and $T_{1}\cdot g$ is orthogonal, since
\[
\ll\,T_{1}f,T_{1}g\gg=\ll\, f,L(T_{1}g)\gg=3\ll\,f\,,\,g\gg=0.
\] 
\end{lemma}

\begin{lemma}
Put $\mathcal{N}_{k}=\{T\in I_{k}~|~L(T)=0\}$, then
\[
\dim ~\mathcal{N}_{k+1}=\dim I_{k+1}-\dim I_{k},
\]
and is equal to
the number of the non-negative integer solutions $(a,b)$ of the equation
\begin{equation}\label{combination of null solution of L}
 2a+3b=k+1.
\end{equation}
\end{lemma}
\begin{proof}
Put $\varphi_{2}(2)=T_{2}-1/3T_{1}^{2}:=\varphi_{2}$ and $\varphi_{3}(3)=T_{3}-T_{2}T_{1}+2/9T_{1}^{3}:=\varphi_{3}$.
Then both of these are irreducible polynomials, 
since they are not
decomposed into lower degree polynomials even on the subspace $z=y=x=0$. 

By $L(\varphi_{2}\cdot\varphi_{3})=0$, products of any powers of these two
polynomials are in the kernel of the map $L$. So corresponding to the
non-negative integer
solutions $(a,b)$ of \eqref{combination of null solution of L}
we have a basis of the kernel $\mathcal{N}_{k+1}$. 
\end{proof}

\begin{lemma}\label{product of dimensions}
For any $j$ and $\ell$
\[\dim H_{j}\cdot I_{\ell}=\dim H_{j}\cdot\dim I_{\ell}.
\]
\end{lemma}
\begin{proof}
It will be apparent if $\dim I_{\ell}=1$.

We prove the property by induction and we show that
the natural map $H_{j}\otimes \mathcal{N}_{\ell}\longrightarrow
H_{j}\cdot \mathcal{N}_{\ell}$ is isomorphic.
So we assume for $^{\forall}\ell\leq k$ and any $j\geq 0$ it holds the isomorphism
\begin{equation}\label{tensor isomorphism}
H_{j}\otimes \mathcal{N}_{\ell}\cong H_{j}\cdot\mathcal{N}_{\ell}. 
\end{equation}
Let 
\begin{equation}\label{null element}
\sum\limits_{(a,b)~\text{run through the solutions of \eqref{combination of null solution of L}}} \,h_{a,b}\cdot{\varphi_{2}}^{a}\cdot{\varphi_{3}}^{b}=0,~h_{a,b}\in H_{j}.
\end{equation}
Let $(a_{1},b_{1}),(a_{2},b_{2}),\cdots, (a_{n},b_{n})$ be all the
solutions of \eqref{combination of null solution of L}:
\[
2a_{i}+3b_{i}=k+1,
\]
Assume $a_{1}>a_{2}>\cdots>a_{n}$, then $b_{1}<b_{2}<\cdots<b_{n}$.
Then the we can assume the expression \eqref{null element} has one of
the following two forms:
\begin{align}
&[1]:~{\text{if}}~ a_{n}>0, {\varphi_{2}}^{a_{n}}\cdot p+{\varphi_{3}}^{b_{n}}\cdot h_{a_{n},b_{n}}=0, ~\text{or}\\  
&[2]:~\text{if}~a_{n}=0,  {\varphi_{2}}^{a_n-1}\cdot p+{\varphi_{3}}^{b_{n}}\cdot h_{a_{n},b_{n}}=0.  
\end{align}
In any case the polynomial $\varphi_{2}$ does not divide the
polynomial $\varphi_{3}$, so that we may put
$h_{a_{n},b_{n}}=\varphi_{2}\cdot Q$ with a polynomial $Q\in
\mathcal{P}_{j-2}$. Then by the equality \eqref{harmonics and polynomial}
the polynomial $Q\cdot \varphi_{2}\in H_{j-1}\cdot I_{1}+H_{j-2}\cdot I_{2}+\cdots+H_{1}\cdot I_{j-1}+I_{j}$. 
On the other hand $Q\cdot\varphi_{2}=h_{a_{n},b_{n}}\in H_{j}$. Hence
by the definition of the space $H_{j}$ which is orthogonal complement
of the space $H_{j-1}\cdot I_{1}+H_{j-2}\cdot I_{2}+\cdots+H_{1}\cdot
I_{j-1}+I_{j}$, hence
$h_{a_{n},b_{n}}=0$ and also $p=0$. By iterating the arguments we see
that in the expression
 \[
\sum\limits_{(a_{i},b_{i})~\text{run through the solutions of
    \eqref{combination of null solution of L}}}
\,h_{a_{i},b_{i}}\cdot{\varphi_{2}}^{a_{i}}
\cdot{\varphi_{3}}^{b_{i}}
\]
all the coefficient polynomials $h_{a_{i},b_{i}}$ must be zero.

Finally we see from the sequences
\[
\begin{CD}
\{0\}@>>> H_{j}\otimes\mathcal{N}_{k+1} @>>> H_{j}\otimes I_{k+1}
@>{Id\otimes L}>>H_{j}\otimes I_{k}@>>>\{0\}\\
@.@VVV@VVV@VVV@.\\
\{0\}@>>> H_{j}\cdot\mathcal{N}_{k+1} @>{\tiny inclusion}>> H_{j}\cdot I_{k+1} @>>> H_{j}\cdot I_{k}@>>>\{0\}
\end{CD}
\]
two spaces
\[
H_{j}\otimes I_{k+1}\cong H_{j}\cdot I_{k+1}
\]
are isomorphic.
\end{proof}

\begin{proposition}\label{direct sum}
For each $k$, the sum $H_{k}+H_{k-1}\cdot I_{1}+\cdots H_{1}\cdot I_{k-1}+I_{k}$ is
a direct sum.
\end{proposition}
\begin{proof}
First we remark that the sums $\mathcal{P}_{1}=H_{1}+I_{1}$ and $\mathcal{P}_{2}=H_{2}+H_{1}\cdot
I_{1}+I_{2}$ are orthogonal sums. 
The first one is included in the definition and the second one is shown as
\[
\ll\,h_{1}T_{1}\,,\,T_{2}\gg=\ll\,h_{1}\,,\,L(T_{2})\,\gg=\ll\,h_{1}\,,\,2T_{1}\,\gg=
\ll\,L(h_{1}),2T_{1}\,\gg=0,~\text{where $h_{1}\in H_{1}$.}
\]
Then we assume that the sum
\[
H_{j}+H_{j-1}\cdot I_{1}+\cdots H_{1}\cdot I_{j-1}+I_{j}
\]
are direct sums for $j\leq k$. 

We express
\begin{align*}
&T\in H_{k}\cdot I_{1}+H_{k-1}\cdot I_{2}+\cdots + H_{1}\cdot I_{k}+I_{k+1}~\text{as}\\
&T=h_{k}\cdot T_{1}+\sum_{i=1}^{2} h_{k-1}(i)\cdot
\varphi_{2}(i)+\cdots+\sum_{i=1}^{\dim I_{k}} h_{1}(i)\varphi_{k}(i)
+\sum_{i=1}^{\dim I_{k+1}} h_{0}(i)\varphi_{k+1}(i)=0,
\end{align*}
where $h_{j}(i)\in H_{j}$ and $\varphi_{j}(i)$ are the basis
polynomials of $I_{j}$ constructed in the Definition \ref{special basis}.
Then by the induction hypothesis, $L(T)=0$ implies
\begin{align*}
&h_{k}=0, h_{k-1}(1)\varphi_{2}=0, h_{k-2}(1)\varphi_{3}(1)+h_{k-2}(2)\varphi_{3}(2)=0, \cdots,
  \sum_{i=1}^{\dim I_{k}}\lambda_{i}\varphi_{k}(i)=0,
\end{align*}
that is, 
the coefficient polynomials $h_{j}(i)$ of the basis included in the
orthogonal complement of $\mathcal{N}_{j}$ are zero. 

Hence it will be enough to show
\begin{equation}\label{equation to be considered for linear independence}
h_{k-1}(2)\varphi_{2}+h_{k-2}(3)\varphi_{3}+\cdots
+\sum_{i=\dim I_{k-1}+1}^{\dim I_{k}}h_{1}(i)\varphi_{k}(i)+\sum_{i=\dim I_{k}+1}^{\dim I_{k+1}}\lambda_{i}\varphi_{k+1}(i)=0 
\end{equation}
implies all the coefficient  polynomials  $h_{j}(i)=0$ and constants $\lambda_{i}=0$.
As in the proof of the Lemma \ref{product of dimensions},
the equation \eqref{equation to be considered for linear independence}
can be rewritten as 
\begin{equation}\label{two expressions of element}
\varphi_{2}\cdot P=-\varphi_{3}\cdot Q
\end{equation}
where the polynomial $P=h_{k-1}(2)+\cdots $~~ is the sum of all the terms including some power
($\geq 0$)
of $\varphi_{2}$ and $Q=g_{1}+ g_{2}\varphi_{3}+\cdots$ (especially
$g_{1}=h_{k-2}(3)\in H_{k-2}$) is a polynomials of the polynomial
$\varphi_{3}$ with the coefficient polynomials $g_{i}\in H_{j}$
with the degree of $g_{i}=k+1-3i$. 
Since $\varphi_{2}$ does not divide $\varphi_{3}$, $Q$ must be divided
by $\varphi_{2}$, 
that is we have
\[
\varphi_{2}\cdot Q_{1}= Q =g_{1}+g_{2}\varphi_{3}+\cdots,
\]
where $Q_{1}\in \mathcal{P}_{k-4}$. 
Hence by 
Lemma \ref{two expressions of the complement of harmonic space}
\[
Q=g_{1}+g_{2}\varphi_{3}+\cdots \in H_{k-3}I_{1}+H_{k-4}I_{2}+\cdots+I_{k-2}, 
\]
which implies that $g_{1}=0$.
Hence we can rewrite \eqref{two expressions of element}
as
\[
\varphi_{2}\cdot P=-{\varphi_{3}}^{2}\cdot Q_{2}.
\]
By iterating the same arguments as above we see that $Q=0$.
Hence $P=0$.

Then we can apply the same argument to the polynomial $P$ by
expressing $P$ as
\[
P=\varphi_{2} P_{1}+R=0,
\]
where $P_{1}$ is the sum of terms in $P$ of the form $h_{a,b}\cdot
\varphi_{2}^{a}{\varphi_{3}}^{b}$, $a>0, ~b\geq 0$, $h_{a,b}\in H_{k-1-2a-3b}$ and $R$ is a polynomial
of $\varphi_{3}$,
\[
R=h'_{0}+h_{1}'\varphi_{3}+h_{2}\varphi_{3}^{2}+\cdots
\] 
with coefficients $h_{a}'\in H_{k-1-3a}$. 

Again by the same argument as above we see that $P=0$, which proves our assertion.
\end{proof}

We put $\mathcal{H}:=\sum_{k\geq 0} ~H_{k}$, and denote by
$I_{+}(\mathcal{J}(3)^{\mathbb{C}})=\sum\limits_{k>0}\, I_{k}(\mathcal{J}(3)^{\mathbb{C}})$
invariant polynomial functions extended to the complexification $\mathcal{J}(3)^{\mathbb{C}}$
in the natural way. 

Since the function taking the trace $A\mapsto
\text{tr}\,(A)$ is linear and $A\mapsto A^{k}$ is an operation inside
the Jordan algebra $\mathcal{J}(3)$ (according to the definition of products)  
and also its complexification
$\mathcal{J}(3)^{\mathbb{C}}$, 
the extensions of the invariant polynomials
$T_{i}$ to $\mathcal{J}(3)^{\mathbb{C}}$ coincide with the 
trace functions on the complexification $\mathcal{J}(3)^{\mathbb{C}}$:
\[
\mathcal{J}(3)^{\mathbb{C}}\ni A\longmapsto
\text{tr}\,(A),~A^{2}\longmapsto \text{tr}\,(A^2)~\text{and}~A^{3}\longmapsto \text{tr}\,(A^{3}).
\]
%which are also extensions to the complexification.

Let $N_{\mathcal{J}(3)^{\mathbb{C}}}$ be the common null set (other than zero) of
the invariant polynomial functions (with respect to $F_{4}$ action) considered on the complexified space
$\mathcal{J}(3)^{\mathbb{C}}$:
\[
N_{\mathcal{J}(3)^{\mathbb{C}}}
:
=\left\{~A=\begin{pmatrix}\xi_{1}&z&\theta(y)\\\theta(z)&\xi_{2}&x\\y&\theta(x)&\xi_{3}\end{pmatrix}\in\mathcal{J}(3)^{\mathbb{C}}
~\Big|~A\not=0,~T_{1}(A)=T_{2}(A)=T_{3}(A)=0~\right\}.
\]
\begin{remark}
Let $A\in N_{\mathcal{J}(3)^{\mathbb{C}}}$.
Then at least one of the three components $z,y,x$ does
not vanish.
Since if $A\in N_{\mathcal{J}(3)^{\mathbb{C}}}$ and assume $z=y=x=0$, then $T_{1}(A)=\sum
\xi_{i}=0,~T_{2}(A)=\sum\,{\xi_{i}}^2=0$
and $T_{3}(A)=\sum\,{\xi_{i}}^3=0$.
Hence these imply that $\xi_{i}=0$ too.
\end{remark}
By the Proposition \ref{trace is zero}
\begin{proposition}\label{non-singular part}
$\mathbb{X}_{\mathbb{O}}=\tau_{\mathbb{O}}(T^{*}_{0}(P^2\mathbb{O}))\subset N_{\mathcal{J}(3)^{\mathbb{C}}}$ 
and the non-singular part of the space $N_{\mathcal{J}(3)^{\mathbb{C}}}$ has $\dim N_{\mathcal{J}(3)^{\mathbb{C}}}=24$.
\end{proposition}
\begin{proof}
Let $A\in\mathbb{X}_{\mathbb{O}}$. Then
$T_{1}(A)=\eta_{1}+\eta_{2}+\eta_{3}=0$ (Proposition \ref{trace is zero}),
and $A^{2}=0$ implies $T_{2}(A)=0$ and $T_{3}(A)=0$ trivially.
Hence
$\mathbb{X}_{\mathbb{O}}=\tau_{\mathbb{O}}(T^{*}_{0}(P^2\mathbb{O}))\subset
N_{\mathcal{J}(3)^{\mathbb{C}}}$.

The second assertion is seen by noting that 
at the points $z=y=x=0$ the three differentials
\[
dT_{1}, ~dT_{2}, ~dT_{3}
\]
are linearly independent.
\end{proof}

Let $A\in \mathcal{J}(3)^{\mathbb{C}}$ and consider the functions of the form 
\begin{equation}\label{basic type function}
\mathcal{J}(3)\ni X\longmapsto \text{tr}\,(X\circ A)=h_{A}(X).
\end{equation}
Since $\mathbb{X}_{\mathbb{O}}$ is $F_{4}$ invariant, the
nontrivial subspace in $H_{1}$ linearly spanned by the functions
\[
\mathcal{J}(3)\ni X\longmapsto \text{tr}\,(X\circ A):=h_{A}(X),
~A\in\mathcal{J}(3)^{\mathbb{C}}, ~\text{tr}\,(A)=0,
\]
is an invariant subspace in $H_{1}$. Here note that
$\text{tr}\,(g(X)\circ A)=\text{tr}\,(X \circ {^tg}(A))$ for ${g}\in F_{4}$
and $\text{tr}\,({^tg}(A))=\text{tr}\,(A)=0$. 

However the representation of the
group $F_{4}$ to $H_{1}$ is irreducible 
(Theorem \ref{irreducible decomposition}),
the space $H_{1}$ must be spanned by these functions. Also the same holds
that 
the subspace in $H_{1}$ linearly spanned by the functions
\[
\left\{\text{tr}\,(X\circ A)=h_{A}(X)~|~A\in N_{\mathcal{J}(3)^{\mathbb{C}}}\right\}
\]
coincides with $H_{1}$ (see Proposition \ref{uniqueness of the zero linear form} and Corollary
\ref{26 dim}).
These facts imply 
\begin{proposition}\label{XO generates VF4}
All the point in $N_{\mathcal{J}(3)^{\mathbb{C}}}$ can be expressed as a linear sum of
points in $\mathbb{X}_{\mathbb{O}}$ and this fact implies 
that  the space $N_{\mathcal{J}(3)^{\mathbb{C}}}$ is path-wise-connected.
\end{proposition}
\begin{proof}
Since any linear function $\mathcal{J}(3)\ni
X\longmapsto\text{tr}\,(X\circ A)=h_{A}(X)$ with $A\in N_{\mathcal{J}(3)^{\mathbb{C}}}$ is a linear sum
of functions of the form $\text{tr}\,(X\circ B_{i})=h_{B_{i}}(X)$ with $B_{i}\in\mathbb{X}_{\mathbb{O}}$,
\[
\text{tr}\,(X\circ A) =\sum \,c_{i}\text{tr}\,(X\circ B_{i})~\text{on}~\mathcal{J}(3),
~\text{where} ~B_{i}\in\mathbb{X}_{\mathbb{O}}, 
\]
$A=\sum\,c_{i}B_{i}$ with these $B_{i}\in\mathbb{X}_{\mathbb{O}}$.

Let $A$ and $A'\in N_{\mathcal{J}(3)^{\mathbb{C}}}$. Assume 
$A=\sum\,c_{i}B_{i}$ and $A'=\sum\,c_{i}'B_{i}'$  where $B_{i},
B_{i}'\in N_{\mathcal{J}(3)^{\mathbb{C}}}$. 
Then
the second assertion is proved  by connecting points $B_{i}$ and $B_{i}'$
suitably in  $\mathbb{X}_{\mathbb{O}}$.
\end{proof}

In general, the space $H_{k}$ of ``Cayley-harmonic polynomials'' 
is an orthogonal sum of two subspaces $H_{k}^{(1)}$ and $H_{k}^{(2)}$,
$H_{k}^{(1)}$ is the subspace linearly spanned by
the powers ${\text{tr}(X\circ A)}^{k}$ with $A\in N_{\mathcal{J}(3)^{\mathbb{C}}}$ 
and $H_{k}^{(2)}$ is the orthogonal complement of $H_{k}^{(1)}$ in
$H_{k}$. The orthogonality is equivalent to the property that
Cayley-harmonic functions in $H_{k}^{(2)}$ are vanishing on the subset
$N_{\mathcal{J}(3)^{\mathbb{C}}}$
(see \cite{He}).

In our case the second subspace $H_{k}^{(2)}$ is always $\{0\}$, that is,
\begin{proposition}\label{vanishing of singular harmonic functions}
\[
H_{k}^{(2)}=H_{k}\bigcap \sqrt{I_{+}(\mathcal{J}(3)^{\mathbb{C}})}=H_{k}\bigcap {I_{+}(\mathcal{J}(3)^{\mathbb{C}})}=\{0\}.
\]
\end{proposition}
\begin{proof}
The first equality is a consequence of Hilbert
Nullstellensatz
and the irreducibility of $N_{\mathcal{J}(3)^{\mathbb{C}}}$ implies the
second equality.

We see the latter one by the following observation that
the equation $T_{1}(A)=0$ is linear so that if we replace the variable
$\xi_{3}$ by $\xi_{3}=-\xi_{1}-\xi_{2}$, then the space
$N_{\mathcal{J}(3)^{\mathbb{C}}}$ 
can be seen as a subset defined by $T_{3}(A)=0$ in the quadrics $Q_{2}=\{A\in
\mathbb{C}^{26}\backslash\{0\}~|~T_{2}(A)=0\}$ 
and the polynomial
$T_{3}$ restricted on the space $z=y=x=0$ is irreducible even modulo $T_{2}$, i.e.,
there are no decomposition such that
$T_{3}(A)={\xi_{1}}^{2}\xi_{2}+\xi_{1}{\xi_{2}}^2=(a\xi_{1}+b\xi_{2})(\alpha{\xi_{1}}^2+\beta\xi_{1}\xi_{2}+\gamma{\xi_{2}}^2)$ 
on ${\xi_{1}}^{2}+{\xi_{2}}^{2}+\xi_{1}\xi_{2}=0$. Hence the space
$N_{\mathcal{J}(3)^{\mathbb{C}}}$ must be irreducible and we have
\[
\sqrt{I_{+}(\mathcal{J}(3)^{\mathbb{C}})}={I_{+}(\mathcal{J}(3)^{\mathbb{C}})}.
\] 
\end{proof}
In fact, our space $N_{\mathcal{J}(3)^{\mathbb{C}}}$ is an irreducible
algebraic manifold and a complete intersection.
In particular, there are points in $N_{\mathcal{J}(3)^{\mathbb{C}}}$ at
which the differentials ~$dT_{1}, ~dT_{2}, ~dT_{3}$ ~are linearly independent
(see the Lemma 4 on page 345 \cite{Ko} for these aspects).

Especially, as a corollary of Proposition \ref{XO generates VF4}
we have
\begin{proposition}\label{irreducibility on Hk}
The representation of $F_{4}$ to the space $H_{k}=H_{k}^{(1)}$ is irreducible for
each $k$.
\end{proposition}
\begin{proof}
Since $\mathbb{X}_{\mathbb{O}}$ is connected, if the space $H_{k}$ is decomposed 
into two invariant subspaces, $H_{k}=G_{1}\oplus G_{2}$, 
then they are orthogonal.%, because $F_{4}$ action is orthogonal. 
Consequently, according to this decomposition the space
$\mathbb{X}_{\mathbb{O}}$ must be separated into two non intersecting
closed subsets and this is a contradiction. 

%$\mathbb{X}_{\mathbb{O}}=F_{1}\bigcup  F_{2}$,
%$F_{1}\bigcap F_{2}=\phi$. 
%This can be seen in a following way:
%we put $F_{i}=\{A\in\mathbb{X}_{\mathbb{O}}~|~\big(\text{tr}\,(A\circ X)\big)^{k}\in G_{i}\}$.          
%If there is an element $A\in F_{1}\bigcap F_{2}$ such that
%$A\not=0$, then this implies that $\{0\}=F_{1}\bigcap F_{2}\ni\,\big(\text{tr}\,(A\circ X)\big)^{k}\not=0$.
%This is a contradiction.
Hence each $H_{k}$ must be irreducible under the
action by the group $F_{4}$.
\end{proof}

Now we sum up a conclusion as 
\begin{theorem}\label{decomposition of L2 space of P2O}
Since the functions in the invariant polynomials $I_{k}$ are constant
on the manifold $P^{2}\mathbb{O}$, by restricting polynomial 
functions in $\mathcal{P}_{k}[\mathcal{J}(3)]$ to $P^{2}\mathbb{O}$ 
the decompositions $\mathcal{P}_{k}[\mathcal{J}(3)]=H_{k}+I_{1}H_{k-1}+\cdots +I_{k}$ for each $k$ give 
totally a decomposition of a subspace in $C^{\infty}(P^2\mathbb{O})$ as 
\[
\sum\limits_{k=0}^{\infty}\,{H_{k}}_{|P^2\mathbb{O}},
\]
which is dense
in $C^{\infty}(P^{2}\mathbb{O})$. 
\end{theorem}
\begin{proof}
Based on the preceding arguments it will be enough to remark the last
assertion, which is a standard argument.

Since any smooth function on
$P^{2}\mathbb{O}$ can be extended to a smooth function on an open neighborhood of
$P^2\mathbb{O}$
and the Weierstrass approximation theorem guarantees
that any smooth function can be approximated in the
$C^{\infty}-$topology by polynomials.
%and the approximation holds when it is restricted to
%$P^{2}\mathbb{O}$.
Hence the space $\sum\,{H_{k}}_{|P^2\mathbb{O}}$ is dense in $C^{\infty}(P^2\mathbb{O})$. 
\end{proof}

Before interpreting the decomposition stated in Theorem 
\ref{decomposition of L2 space of P2O} 
in the framework of the Peter-Weyl theorem for a symmetric space of our case $P^{2}\mathbb{O}$
we remark about the Riemannian metric on $P^{2}\mathbb{O}$.
\begin{proposition}\label{Killing form metric} 
The Cayley projective plane
$P^{2}(\mathbb{O})\cong F_{4}/Spin(9)$
is an irreducible Riemannian symmetric space,
that is,
the stationary subgroup $Spin(9)$ acts irreducibly
on the tangent space $T_{X_{1}}P^{2}(\mathbb{O})$.
By Schur's lemma
this implies that
$P^{2}(\mathbb{O})$ has an essentially unique $F_{4}$-invariant
Riemannian metric.
Thus, $(\,\cdot\,,\,\cdot\,)^{P^2\mathbb{O}}$
coincides with the metric on $P^{2}(\mathbb{O})$ induced from the Killing form of 
the Lie algebra of $F_{4}$
up to a constant factor.
\end{proposition}
Let $\Phi_{k}:H_{k}\otimes {H_{k}}^{*}\longrightarrow C^{\infty}(F_{4})$
be a map defined by
\[
H_{k}\otimes {H_{k}}^{*}\ni h\otimes \varphi \longmapsto \Phi_{k}(h\otimes \varphi)(g)=\varphi(\mathcal{P}_{g^{-1}}(h)), ~g\in F_{4},
\]
then the Peter-Weyl theorem says that
the image of the map $\Phi_{k}$
is a subspace consisting of the $\dim H_{k}$ number of the spaces,
all of which are
isomorphic to $H_{k}$.

Recall we explained the
identification \eqref{isomorphism between P2O and F4/Spin(9)} of the 
quotient space $F_{4}/\text{Spin(9)}$ with $P^{2}\mathbb{O}$ 
through the correspondence $F_{4}\ni g\longmapsto g(X_{1})\in P^{2}\mathbb{O}$.

If we consider a subspace ${{H_{k}}^{*}}_{|\text{Spin}(9)}$ consisting of linear
forms in ${H_{k}}^{*}$ which are invariant under the action by $\text{Spin}(9)$,
then the functions in
\[
\Phi_{k}(H_{k}\otimes {{H_{k}}^{*}}_{|\text{Spin(9)}})
\]
are $\text{Spin}(9)$ invariant, so that it can be descended naturally
to
functions on 
$F_{4}/\text{Spin(9)}\cong P^{2}\mathbb{O}\subset \mathcal{J}(3)$. 

For $X\in\mathcal{J}(3)$ we denote the linear form $J_{X}\in {H_{k}}^{*}$
\[
H_{k}\ni h\longmapsto J_{X}(h)=h(X),
\]
that is, this is an evaluation at $X\in\mathcal{J}(3)$.
In particular, we take a linear form $J_{X_{1}}\in
{{H_{k}}^{*}}_{|\text{Spin}(9) }$, then
it can be written as 
\[
J_{X_{1}}(\mathcal{P}_{g^{-1}}(h))=\mathcal{P}_{g^{-1}}(h)(X_{1})=
h(g(X_{1})).
\]
Hence through the identification $F_{4}/\text{Spin}(9)\cong P^{2}\mathbb{O}$
the function
$J_{X_{1}}(\mathcal{P}_{g^{-1}}(h))$ 
is the restriction of the original polynomial function $h\in H_{k}$
to $P^2\mathbb{O}$.
Then we have
\[
\sum\limits_{k=0}^{\infty}\,\Phi_{k}(H_{k}\otimes \{J_{X_{1}}\})
=\sum\limits_{k=0}^{\infty}\,{H_{k}}_{|P^2\mathbb{O}}.
\]

Since $\dim H_{k+1}>\dim H_{k}$ {(see Appendix)} and
the space $\sum\limits_{k=0}^{\infty}\,{H_{k}}_{|P^2\mathbb{O}}$ is
already dense in $C^{\infty}(P^2\mathbb{O})$, a fundamental theorem on
compact symmetric spaces gives us 
\begin{proposition}\label{irreducible rep}
Each irreducible representation of the group $F_{4}$ appears in $C^{\infty}(P^{2}\mathbb{O})$
with multiplicity one as in the above way and incidentally
$\dim {{H_{k}}^{*}}_{|\text{Spin}(9)}=1$. Moreover
by the Proposition \ref{Killing form metric} we can see that
this decomposition is the eigenspace decomposition of the Laplacian on
$P^{2}\mathbb{O}$.

The dimension of the space ${{H_{k}}^{*}}_{|\text{Spin}(9)}$ is always
one and
the linear form $J_{X_{1}}$
can be seen as a base vector of the space ${{H_{k}}^{*}}_{|\text{Spin}(9)}$ for any $k$.
\end{proposition}

\section{Inverse of Bargmann type transformation}

In this section, based on the data obtained until $\S 8$ 
we consider our Bargmann type transformation 
\[
\mathfrak{B}:\sum\,\mathcal{P}_{k}[\mathbb{X}_{\mathbb{O}}]\longrightarrow C^{\infty}(P^2\mathbb{O})
\]
with respect to the parameter family of the inner products 
$\{\boldsymbol{(}*\,,\,*\boldsymbol{)}_{\varepsilon}\}_{-22<\varepsilon}$ 
on the space $\sum\,\mathcal{P}_{k}[\mathbb{X}_{\mathbb{O}}]$ on its
boundedness and invertibility.
It has a dense image from $\sum
\,\mathcal{P}_{k}[\mathbb{X}_{\mathbb{O}}]$ always for a possible
value 
of the parameter $\varepsilon$, but 
unlike the cases of spheres and other projective spaces (see \cite{Ra2},\cite{Fu1},
\cite{FY}), 
it need not be an isomorphism when $\varepsilon=0$. This
means in cases of the values of the parameter $\varepsilon>-47/2$
there are quantum states in $L_{2}(P^{2}\mathbb{O})$ which can not be
seen by classical observables.  

\subsection{Inverse transformation}

Let $\mathcal{A}_{k}$ 
be a transformation defined by
\begin{align}
&\mathcal{A}_{k}:H_{k}\ni \varphi\longmapsto
\int_{P^2\mathbb{O}}\,\varphi(X)\cdot \big(\text{tr}\,(X\circ A)\big)^{k}\,dv_{P^2\mathbb{O}}(X)\in
  \mathcal{P}_{k}[\mathbb{X}_{\mathbb{O}}].\label{eigenfunction to polynomial}
\intertext{and}
&\mathbb{A}_{k}:H_{k}\ni\varphi\longrightarrow \mathbb{A}_{k}(\varphi)
=\gamma\circ \mathcal{A}(\varphi)=\mathcal{A}_{k}(\varphi)\cdot{\bf t}_{0}
\otimes\Omega_{\mathbb{O}}\in 
\Gamma_{\mathcal{G}}\left(\mathbb{L}\otimes K^{\mathcal{G}},\mathbb{X}_{\mathbb{O}}\right).
\end{align}
The correspondence by $\gamma$ is defined in \eqref{holomorphic function to L-valued 16 form}.
\begin{proposition}
For any inner product defined on the space
$\mathcal{P}_{k}[\mathbb{X}_{\mathbb{O}}]$ according to the value of the
parameter $\varepsilon$,
the operator $\mathcal{A}_{k}$ is a constant times a unitary operator.
\end{proposition}

\begin{proof} For $\varphi\in H_{k}$ the inner product
\begin{equation}\label{inner product of Bkvarphi}
\boldsymbol{(}\,\mathcal{A}_{k}(\varphi),\mathcal{A}_{k}(\varphi)\boldsymbol{)}_{\varepsilon} 
\end{equation} 
is expressed as
\begin{align*}
&\boldsymbol{(}\,\mathcal{A}_{k}(\varphi),\mathcal{A}_{k}(\varphi)\boldsymbol{)}_{\varepsilon} \\
&=\int_{\mathbb{X}_{\mathbb{O}}}\,\left|\int_{P^2\mathbb{O}}\,\varphi(X)
(\,\text{tr}\,(X\circ A)\,)^{k}dv_{P^2\mathbb{O}}(X)\right|^{2}\cdot
  e^{-2\sqrt{2}\pi||A||^{1/2}}||A||^{\varepsilon}\Omega_{\mathbb{O}}\wedge\overline{\Omega_{\mathbb{O}}}\\
&=\int_{P^2\mathbb{O}}\,\int_{P^2\mathbb{O}}\,
\left(\int_{\mathbb{X}_{\mathbb{O}}}\,\big(\text{tr}\,(\tilde{X}\circ A)\,\big)^{k}
\big(\,\text{tr}\,({X}\circ\overline{A})\,\big)^{k}\cdot 
e^{-2\sqrt{2}\pi||A||^{1/2}}\cdot ||A||^{\varepsilon}\cdot\Omega_{\mathbb{O}}\wedge\overline{\Omega_{\mathbb{O}}}\right)
\times\\
&\qquad\qquad\qquad\qquad\qquad\qquad\qquad\qquad\qquad\qquad\quad\quad
\times\varphi(\tilde{X})\overline{\varphi(X)}dv_{P^2\mathbb{O}}(X)dv_{P^2\mathbb{O}}(\tilde{X}).
\end{align*}
Here we consider the operator $\mathcal{B}_{k}$
\[
\mathcal{P}_{k}[\mathbb{X}_{\mathbb{O}}]\ni h\longmapsto
\mathcal{B}_{k}(h):=\int_{\mathbb{X}_{\mathbb{O}}}\,h(A)\cdot \big(\text{tr}\,(X\circ
\overline{A})\big)^{k}e^{-2\sqrt{2}\pi||A||^{1/2}}\cdot||A||^{\varepsilon}\cdot\Omega_{\mathbb{O}}(A)
\wedge\overline{\Omega_{\mathbb{O}}(A)}\in H_{k}.
\]
Since $H_{k}$ consists of linear sums of functions of the form
$\big(\text{tr}\,(X\circ A)\big)^{k}$ by arbitrary
$A\in\mathbb{X}_{\mathbb{O}}$ (see Proposition \ref{vanishing of singular harmonic functions}),
we see that $\mathcal{B}_{k}(h)\in H_{k}$.
Then
the inner product \eqref{inner product of Bkvarphi}
is understood as
\[
\boldsymbol{(}\,\mathcal{A}_{k}(\varphi),\mathcal{A}_{k}(\varphi)\boldsymbol{)}_{\varepsilon} 
=(\mathcal{B}_{k}\circ \mathcal{A}_{k}(\varphi),\varphi)^{P^2\mathbb{O}}.
\]
Then the operator $\mathcal{B}_{k}\circ\mathcal{A}_{k}$ commutes with
the $F_{4}$ action on $H_{k}$. Hence it must be a constant times
identity operator (which constant we put $b_{k}$) so that
the kernel function defined by the integral
\[
L_{k}(X,\tilde{X}):=
\left(\int_{\mathbb{X}_{\mathbb{O}}}\,\big(\,\text{tr}\,(\tilde{X}\circ
  A)\,\big)^{k}
\big(\,\text{tr}\,({X}\circ\overline{A})\,\big)^{k}\cdot
e^{-2\sqrt{2}\pi||A||^{1/2}}\cdot||A||^{\varepsilon}\cdot
\Omega_{\mathbb{O}}\wedge\overline{\Omega_{\mathbb{O}}}\right)
\]
must satisfies the invariance:
\begin{equation}\label{invariance of L}
L_{k}(g(X),\tilde{X})=L_{k}(X,g^{-1}(\tilde{X})),~\text{for}~{g}\in
F_{4},~X,\tilde{X}\in P^{2}\mathbb{O}.
\end{equation}
Then the constant $b_{k}$ is given by
\begin{equation}\label{trace of L}
\text{Trace of the operator $\mathcal{B}_{k}\circ\mathcal{A}_{k}$}=
\int_{P^2\mathbb{O}}L_{k}(X,X)dv_{P^2\mathbb{O}}=b_{k}\dim H_{k}.
\end{equation}
and the integral $\displaystyle{\int_{P^2\mathbb{O}}\,L_{k}(X,X)dv_{P^2\mathbb{O}}}$ is
given by
\[
{\int_{P^2\mathbb{O}}\,L_{k}(X,X)dv_{P^2\mathbb{O}}}\equiv L_{k}(X,X)\cdot V\hspace{-0.13cm}{ol(P^{2}\mathbb{O})},
\]
since by the invariance \eqref{invariance of L} the function 
$L_{k}(X,X)$ is a constant function and apparently is non-zero.

Now we know $\mathcal{B}_{k}$ is injective and so
\[
\dim H_{k}\leq \dim \mathcal{P}_{k}[\mathbb{X}_{\mathbb{O}}].
\]
On the other hand, degree $k$ polynomials generated by the
invariant polynomials which are naturally extended to the complexification $\mathcal{J}(3)^{\mathbb{C}}$, 
that is, 
the polynomials
\[
\sum_{i=0}^{k} \mathcal{P}_{k-i}[\mathcal{J}(3)^{\mathbb{C}}]\cdot
I_{k-i}=\sum_{i=0}^{k} H_{k-i}\cdot I_{i}
\]
({see Lemma \ref{two expressions of the complement of harmonic space}})
are all vanishing on the manifold $\mathbb{X}_{\mathbb{O}}$ so that
\[
\dim \mathcal{P}[\mathbb{X}_{\mathbb{O}}]\leq \dim
\mathcal{P}[\mathcal{J}(3)^{\mathbb{C}}]
-\sum_{i=0}^{k}\,\dim H_{k}\cdot\dim I_{k},
\]
(see Proposition \ref{direct sum}).

Hence the operator $\mathcal{B}_{k}$ is also surjective to the space
$\mathcal{P}_{k}[\mathbb{X}_{\mathbb{O}}]$.
Consequently, the operator $\mathcal{B}_{k}$ is a constant times a unitary operator.
\end{proof}
Next, we determine the concrete value of the constant $b_{k}$:
\begin{proposition}\label{the value bk=Lk(X)}
\[
L_{k}(X,X)=b_k\cdot \dim H_{k}
={2^{26}}\cdot V\hspace{-0.11cm}{ol(S(P^{2}\mathbb{O}))}\cdot\frac{\Gamma(4k+44+2\varepsilon)}{2^{8k+66+3\varepsilon}\pi^{4k+44+2\varepsilon}},
\]
where the constant $V\hspace{-0.1cm}{ol}(S(P^{2}\mathbb{O}))$ is the
volume of the unit cotangent sphere bundle $S(P^{2}\mathbb{O})$ of $P^{2}\mathbb{O}$ with respect
to the volume form 
\[
d\sigma_{S(P^{2}\mathbb{O})}:=\frac{1}{16!}\cdot{
\theta^{P^2\mathbb{O}}\wedge \big(\omega^{P^{2}\mathbb{O}})^{15}}_{\big|S(P^2\mathbb{O})}.
\]
\end{proposition}
\begin{proof}
Since $L_{k}(X,X)= 
\displaystyle{\int_{\mathbb{X}_{\mathbb{O}}}\,
\Big|\,\text{tr}\,({X}\circ A)\,\Big|^{2k}\cdot
e^{-2\sqrt{2}\pi||A||^{1/2}}||A||^{\varepsilon}\cdot \Omega_{\mathbb{O}}(A)\wedge\overline{\Omega_{\mathbb{O}}(A)}}$
does not depend on the point $X\in P^{2}{\mathbb{O}}$, we have
\begin{align}
&L_{k}(X,X)=\int_{\mathbb{X}_{\mathbb{O}}}\Big|\,\text{tr}\,(X\circ A)\,\Big|^{2k}\cdot 
e^{-2\sqrt{2}\pi||A||^{1/2}}\cdot||A||^{\varepsilon}\cdot\Omega_{\mathbb{O}}(A)\wedge\overline{\Omega_{\mathbb{O}}(A)}\notag\\
&=\int_{F_{4}}
\left(\int_{\mathbb{X}_{\mathbb{O}}}\Big|\,\text{tr}\,(g^{-1}(X)\circ A)\,\Big|^{2k}
\cdot e^{-2\sqrt{2}\pi||A||^{1/2}} \cdot||A||^{\varepsilon}\cdot
\Omega_{\mathbb{O}}(A)\wedge\overline{\Omega_{\mathbb{O}}(A)}\right)dv_{F_{4}}(g)\notag\\
&=\int_{\mathbb{X}_{\mathbb{O}}}\left(\int_{F_{4}}\,\left|\,\text{tr}\,\left(X\circ g\left(\frac{A}{||A||}\right)
\,\right)\right|^{2k}dv_{F_{4}}(g)\right) 
\cdot||A||^{2k+\varepsilon}\cdot e^{-2\sqrt{2}\pi||A||^{1/2}}\Omega_{\mathbb{O}}(A)\wedge\overline{\Omega_{\mathbb{O}}(A)},
\label{bk and intgration on F4}
\end{align}
where $dv_{F_{4}}$ is the normalized Haar measure on $F_{4}$. 

The function 
\begin{equation}\label{invariant integrand function}
\int_{F_{4}}\,\left|\,\text{tr}\,\left(X\circ g\left(\frac{A}{||A||}\right)\,\right)\right|^{2k}\,dv_{F_{4}}(g)
\end{equation}
does not depend neither on $X\in P^{2}\mathbb{O}$ nor on $A\in\mathbb{X}_{\mathbb{O}}$, since 
the trace function $A\longmapsto\text{tr}\,(A)$ is $F_{4}$-invariant,
the group $F_{4}$ acts both on the spaces $P^{2}\mathbb{O}$ and 
the cotangent sphere bundle $S(P^{2}\mathbb{O})\stackrel{\tau_{\mathbb{O}}}\cong S(\mathbb{X}_{\mathbb{O}})$
transitively and the Haar measure $dv_{F_{4}}$ is bi-invariant. 

Let $(X,Y)\in T^{*}_{0}(P^{2}\mathbb{O})$. Put
$A_{g}(X,Y):=g(\tau_{\mathbb{O}}(X,Y))$, then
\begin{align*}
&g(\tau_{\mathbb{O}}(X,Y))=g\left(||Y||^2X-Y^2+\sqrt{-1}\otimes\frac{||Y||}{\sqrt{2}}Y\right)
=g(||Y||)^2g(X)-g(Y)^{2}+\sqrt{-1}\otimes\frac{||g(Y)||}{\sqrt{2}}g(Y).
\end{align*}
Hence
\[
{\tau_{\mathbb{O}}}^{-1}(A_{g}(X,Y))=(\,X(A_{g}(X,Y)),
Y(A_{g}(X,Y)\,)=(\,g(X),g(Y)\,)\in T^{*}_{0}(P^{2}\mathbb{O}).
\]
The integral \eqref{invariant integrand function} is expressed as
\begin{align*}
&\frac{1}{||A_{g}(X,Y)||^{2k}}
\int_{F_{4}}\,\big|\text{tr}\,X(A_{g}(X,Y))\circ A_{g}(X,Y)\,\big|^{2k}dv_{F_{4}}(g)\\
&=\frac{1}{||Y||^{4k}}
\int_{F_{4}}\left|\,\text{tr}\,g(X)\circ\left(||g(Y)||^2g(X)-g(Y)^2
+\sqrt{-1}\otimes \frac{||g(Y)||}{\sqrt{2}}g(Y)\right)\,\right|^{2k}dv_{F_{4}}(g)\\
&=\frac{1}{||Y||^{4k}}\int_{F_{4}}\,\left(\frac{1}{2}||g(Y)||^2\right)^{2k}dv_{F_{4}}=\frac{1}{2^{2k}},
\end{align*}
since
\[
g(X)^2=g(X),\,\,\text{tr}\,g(X)=1\,,\, g(X)\circ g(Y)=\frac{1}{2}g(Y)
\]
and we used the property
\[
\text{tr}\,(X\circ Y)\circ Z=\text{tr}\,X\circ (Y\circ Z).
\]
Now the integral \eqref{bk and intgration on F4} is
\begin{align}
&\int_{\mathbb{X}_{\mathbb{O}}}\int_{F_{4}}\Big|\text{tr}\,{X}\circ g(A)\Big|^{2k}dv_{F_{4}}(g)\cdot 
e^{-2\sqrt{2}\pi||A||^{1/2}}||A||^{\varepsilon}\cdot\Omega_{\mathbb{O}}(A)
\wedge\overline{\Omega_{\mathbb{O}}(A)}\notag\\
&=
\int_{\mathbb{X}_{\mathbb{O}}}\frac{1}{2^{2k}}||A||^{2k+\varepsilon}\cdot
  e^{-2\sqrt{2}\pi||A||^{1/2}}\cdot\Omega_{\mathbb{O}}(A)\wedge\overline{\Omega_{\mathbb{O}}(A)}\notag\\
&
=\frac{2^{26}}{2^{2k}}\int_{T^{*}_{0}(P^2\mathbb{O})}\,
  ||Y||^{4k+28+2\varepsilon}\cdot e^{-2\sqrt{2}\pi||Y||}\cdot dV_{T^{*}(P^2\mathbb{O})}.
\label{one step final expression of bk}
\end{align}
where we used the relation \eqref{Calabi-Yau and Liouville}.
Then according to the decomposition of the space
$T^{*}_{0}(P^2\mathbb{O})\cong\mathbb{R}_{+}\times S(P^2\mathbb{O})$,
we can decompose the Liouville volume form $dV_{T^{*}(P^{2}\mathbb{O})}$ as
\[
dV_{T^{*}(P^{2}\mathbb{O})}=t^{15}dt\wedge d\sigma_{S(P^2\mathbb{O})},
\]
where $d\sigma_{S(P^2\mathbb{O})}$ is the volume form on the unit
cotangent sphere bundle $S(P^{2}\mathbb{O})$. Finally
we have the integral \eqref{one step final expression of bk} as
\begin{align}
&\frac{2^{26}}{2^{2k}}\cdot
\int_{T^{*}_{0}(P^2\mathbb{O})}\, ||Y||^{4k+28+2\varepsilon}\cdot
  e^{-2\sqrt{2}\pi ||Y||} dV_{P^2\mathbb{O}}
=\frac{2^{26}}{2^{2k}}\cdot\int_{S(P^2\mathbb{O})}d\sigma_{S(P^2\mathbb{O})}\int_{0}^{\infty}
 t^{4k+28+2\varepsilon}e^{-2\sqrt{2}\pi t}\cdot t^{15}dt\notag\\
&=\frac{2^{26}}{2^{2k}}\cdot V\hspace{-0.11cm}{ol(S(P^{2}\mathbb{O}))}\cdot
  \frac{\Gamma(4k+44+2\varepsilon)}{(2\sqrt{2}\pi)^{4k+44+2\varepsilon}}=\frac{1}{2^{40}\pi^{44}}\cdot V\hspace{-0.11cm}{ol(S(P^{2}\mathbb{O}))}\cdot
\frac{\Gamma(4k+44+2\varepsilon)}{2^{8k+3\varepsilon}\pi^{4k+2\varepsilon}},\notag
\end{align}
and
\begin{equation}\label{value of bk}
b_{k}=\frac{1}{2^{40+3\varepsilon}\pi^{44+2\varepsilon}}\cdot
V\hspace{-0.11cm}{ol(S(P^{2}\mathbb{O}))}\cdot
\frac{\Gamma(4k+44+2\varepsilon)}{2^{8k}\pi^{4k}\dim H_{k}}.
\end{equation}
\end{proof}
\begin{proposition}
Since both of the transformations $\mathcal{A}_{k}$ and 
the restriction of the transformation $\mathfrak{B}$ to 
the space $\mathcal{P}[\mathbb{X}_{\mathbb{O}}]$ {\em(}for short we denote
it by $T_{k}:=\mathfrak{B}_{|\mathcal{P}[\mathbb{X}_{\mathbb{O}}]}${\em)}
commute with  $F_{4}$ action and the representation of $F_{4}$ on
$H_{k}$ is irreducible {\em(see \eqref{irreducible rep})}, the composition 
$T_{k}\circ\mathcal{A}_{k}$ on ${H_{k}}_{|P^2\mathbb{O}}$ is a constant
multiple operator $T_{k}\circ\mathcal{A}_{k}=a_{k}Id$ 
and the constant $a_{k}$ is given by
\begin{align}
a_{k}&=
{2^6\cdot V\hspace{-0.13cm}{ol(S^{15})}\cdot V\hspace{-0.13cm}{ol{(P^2\mathbb{O})}}}\cdot
\frac{\Gamma(2k+22)}{2^{k+11}\cdot\pi^{2k+22}\dim H_{k}}\notag\\
&
=\frac{1}{2^{5}\pi^{22}}{\cdot} V\hspace{-0.13cm}{ol(S^{15})}\cdot 
V\hspace{-0.13cm}{ol{(P^2\mathbb{O})}}\cdot\frac{\Gamma(2k+22)}{2^{2k}\cdot\pi^{2k}\dim
  H_{k}}.
\label{constant ak}
\end{align}
\end{proposition}
\begin{proof}
Let $f\in {H_{k}}$  
then by Corollary \ref{Bargmann type transform by Liouville volume form}
\begin{align*}
&T_{k}\Big(\mathcal{A}_{k}(f)\Big)({X})\cdot dv_{P^2\mathbb{O}}({X})\\
&=2^{6}{\bf q}_{*}
\left(\int_{P^2\mathbb{O}}\,f(\tilde{X})\cdot\{\text{tr}\,(\tilde{X}\circ
  \tau_{\mathbb{O}}({X},*)\,)\}^{k}\cdot dv_{P^2}(\tilde{X})
\cdot e^{-\sqrt{2}\pi\cdot ||*||}\cdot ||*||^{6}\cdot dV_{T^{*}(P^2\mathbb{O})}(X,*)
\right)\\
&=2^6\int_{P^{2}\mathbb{O}}\,f(\tilde{X})
{\bf q}_{*}\Big(\{\text{tr}\,(\tilde{X}\circ \tau_{\mathbb{O}}(X,*)\,)\}^{k}  
\cdot e^{-\sqrt{2}\pi\cdot||*||}\cdot ||*||^{6}dV_{T^{*}(P^2\mathbb{O})}(X,*)\Big)dv_{P^2\mathbb{O}}(\tilde{X})\\
&= 2^6\int_{P^{2}\mathbb{O}}\,f(\tilde{X})K_{{k}}(\tilde{X},X)dv_{P^2\mathbb{O}}(\tilde{X})\cdot dv_{P^2\mathbb{O}}(X),
\end{align*}
where we put the fiber integral as
\[
K_{k}(\tilde{X},X)\cdot dv_{P^2\mathbb{O}}(X):=
{\bf q}_{*}
\Big(\{\text{tr}\,\tilde{X}\circ \tau_{\mathbb{O}}(X,*)\}^{k}
\cdot e^{-\sqrt{2}\pi\cdot||*||}\cdot||*||^{6}dV_{T^{*}(P^2\mathbb{O})}(X,*)\Big).
\]
The kernel function $K_{k}(\tilde{X},X)$ satisfies the
property similar to the kernel function $L_{k}(\tilde{X},X)$:
\begin{equation}\label{diagonal is invariant}
K_{k}(g\cdot\tilde{X},\,X)= K_{k}(\tilde{X},\,g^{-1}(X)). 
\end{equation}
Then by this property \eqref{diagonal is invariant} 
that $K_{k}(X,X)$ is constant and we have
\begin{align*}
&\text{tr}\,(T_{k}\circ \mathcal{A}_{k})=a_{k}\cdot \dim H_{k}%\\
%&
=2^6\cdot\int_{P^2\mathbb{O}}\,K_{k}(X,X)dv_{P^2\mathbb{O}}(X)
=2^6\cdot K_{k}(X,X)\cdot V\hspace{-0.13cm}ol(P^2\mathbb{O}).
\end{align*}
Since $\text{tr}\,\big(X\circ \tau_{\mathbb{O}}(X,Y)\,\big)=1/2||Y||^2$,
\begin{align*}
&{\bf q}_{*}\Big(\{\text{tr}\,(X\circ \tau_{\mathbb{O}}(X,*)\,)\}^{k}\cdot
e^{-\sqrt{2}\pi\cdot ||*||}\cdot ||*||^{6}\cdot
  dV_{T^{*}(P^2\mathbb{O})}(X,Y)\Big)\\
&=(1/2)^{k}\cdot{\bf q}_{*}\Big(||*||^{2k+6}e^{-\sqrt{2}\pi\cdot ||*||}\cdot dV_{T^{*}(P^2\mathbb{O})}(X,*)\Big).
\end{align*}

If we choose a point $X=X_{1}$, then the above fiber integral is expressed as
\begin{align}
&(1/2)^{k}\cdot{\bf q}_{*}\Big(||*||^{2k+6}e^{-\sqrt{2}\pi\cdot ||*||}\cdot dV_{P^2\mathbb{O}}(X_{1},*)\Big)\notag\\
&=(1/2)^{k}\cdot \int_{{\bf q}^{-1}(X_{1})}||Y||^{2k+6}e^{-\sqrt{2}\pi||Y||}
d\beta_{0}{\wedge\cdots}{\wedge d\beta_{7}}{\wedge d\gamma_{0}}
{\wedge\cdots}{\wedge d\gamma_{7}}{\wedge}\notag\\
&\qquad\qquad\qquad\qquad\qquad\qquad\qquad\qquad\qquad\qquad\qquad\qquad
{\wedge db_{0}}{\wedge\cdots}{\wedge db_{7}}{\wedge dc_{0}}{\wedge\cdots}{\wedge dc_{7}},\notag\\
&=(1/2)^{k}\cdot \int_{{\bf q}^{-1}(X_{1})}||Y||^{2k+6}e^{-\sqrt{2}\pi||Y||} 
{d\beta_{0}}{\wedge\cdots}{\wedge d\beta_{7}}{\wedge d\gamma_{0}}{\wedge\cdots}
{\wedge d\gamma_{7}}
{\wedge dv_{P^2\mathbb{O}}(X_{1})},\label{fiber integral at X1}
\end{align}
where we express the integral using the local coordinates
on $\widetilde{\mathcal{W}}_{1}$ (see \eqref{coordinate  X1}) around the
point $X_{1}$ and the dual coordinates
$(X,Y)=(b,c,\beta,\gamma)\longleftrightarrow\sum_{i}\beta_{i}db_{i}+\gamma_{i}dc_{i}\in
T^{*}_{X}(\widetilde{\mathcal{W}}_{1})$. 
Then the integral \eqref{fiber integral at X1} over the point $X_{1}$ is
\begin{align*}
&(1/2)^{k}\cdot \int_{{\bf q}^{-1}(X_{1})}||Y||^{2k+6}e^{-\sqrt{2}\pi||Y||}
d\beta_{0}\wedge\cdots\wedge d\beta_{7}\wedge d\gamma_{0}\wedge\cdots\wedge d\gamma_{7}\\
&=(1/2)^{k}\cdot
  \int_{\mathbb{R}^{16}}\left(\sum{\beta_{i}}^2+{\gamma_{i}}^2\right)^{k+3}
e^{-\sqrt{2}\pi\sqrt{\sum{(\beta_{i}}^2+{\gamma_{i}}^2)}}
d\beta_{0}\cdots d\beta_{7} d\gamma_{0}\cdots d\gamma_{7}=\frac{\Gamma(2k+22)}{2^{2k+11}\cdot\pi^{2k+22}}\cdot V\hspace{-0.13cm}{ol(S^{15})}. 
\end{align*}
Here $V\hspace{-0.13cm}{ol(S^{15})}$ is the volume of the standard $15$-sphere.
\end{proof}

Now we have
\begin{align*}
a_{k}\dim H_{k}
&=2^6\cdot\frac{\Gamma(2k+22)}{2^{2k+11}\cdot\pi^{2k+22}}\cdot
V\hspace{-0.13cm}{ol(S^{15})}\cdot
V\hspace{-0.13cm}{ol{(P^2\mathbb{O})}}=
\frac{1}{2^{5}\pi^{22}}{\cdot}V\hspace{-0.13cm}{ol(S^{15})}{\cdot}V\hspace{-0.13cm}{ol{(P^2\mathbb{O})}}
\cdot\frac{\Gamma(2k+22)}{2^{2k}\cdot\pi^{2k}}. 
\end{align*}
\begin{proposition}\label{inverse of mathfrak{B}}
\begin{align*}
&\mathfrak{B}_{|\mathcal{P}_{k}[\mathbb{X}_{\mathbb{O}}]}\circ\mathcal{A}_{k}
=\frac{1}{2^5\pi^{22}}\frac{\Gamma(2k+22)}{2^{2k}\pi^{2k}\cdot
  \dim H_{k}}\cdot V\hspace{-0.13cm}{ol(S^{15})}\cdot
  V\hspace{-0.13cm}{ol{(P^2\mathbb{O})}}\,Id. 
\end{align*}
\end{proposition}
\begin{corollary}
The operator norm $||{\mathfrak{B}^{-1}}_{|\mathcal{P}_{k}[\mathbb{X}_{\mathbb{O}}]}||$
is given by
\begin{align}
&||{\mathfrak{B}^{-1}}_{|\mathcal{P}_{k}[\mathbb{X}_{\mathbb{O}}]}||=\frac{\sqrt{b_{k}}}{a_{k}}
=\frac{\sqrt{\frac{Vol(S(P^2\mathbb{O})\Gamma(4k+44+2\varepsilon))}{2^{40+3\varepsilon}\pi^{44+2\varepsilon}\cdot
  2^{8k}\pi^{4k}\dim H_{k}}}}{\frac{Vol(S^{15})\cdot
  Vol(P^{2}\mathbb{O})\cdot\Gamma(2k+22)}{2^{2k}\pi^{2k}\cdot \dim
  H_{k}}}:=C(\varepsilon)\cdot N(k),\label{norm}
\end{align}
where $C(\varepsilon)~{\text{includes only $\varepsilon$}}$ and $N(k)$ is
a function of $k$ and
\begin{equation}\label{norm dependence with respect to k}
N(k)^{2}=\frac{2^{4k}\cdot\dim H_{k}\cdot\Gamma(4k+44+2\varepsilon)}{2^{8k}\Gamma(2k+22)^2}.
\end{equation}
\end{corollary}
It is enough to see \eqref{norm dependence with respect to k} for the behavior of the norm \eqref{norm} when
$k\longrightarrow \infty$ and for this purpose we
mention two properties of the Gamma function.
\begin{lemma}\label{formula 1}
\[\lim_{k\to\infty}
\frac{\Gamma(k+\alpha_{1})\cdots\Gamma(k+\alpha_{\ell})}{\Gamma(k+\beta_{1})\cdots\Gamma(k+\beta_{\ell})}
=\left\{
\begin{array}{l}
+\infty,~\text{if}~\sum \,\alpha_{i}>\sum\,\beta_{i},\\
1, ~\text{if}~\sum \,\alpha_{i}=\sum\,\beta_{i},\\
0,~\text{if}~\sum \,\alpha_{i}<\sum\,\beta_{i}.
\end{array}\right.
\]
\end{lemma}
\begin{lemma}\label{formula 2}
\[
\Gamma(nz)= \frac{n^{nz-1/2}}{(2\pi)^{(n-1)/2}}\cdot \prod_{j=0}^{n-1} \Gamma\left(z+\frac{j}{n}\right). 
\]
\end{lemma}
Then by Lemma \ref{formula 2}
\begin{equation}\label{function of k}
N(k)^{2}=\frac{2^{44+4\varepsilon}}{\sqrt{2\pi}}\cdot 
\dim H_{k}\cdot 
\frac{\prod_{j=0}^3\Gamma(k+11+\varepsilon/2+j/4)}{\Gamma(k+11)^2 \cdot\Gamma(k+11+1/2)^{2}}.
\end{equation}
By the relation of the Poincar\'e polynomials $PP(t)=PH(t)\cdot PI(t)$
(see \eqref{relation of three Poincare polynomials}), 
the dimension of $H_{k}$ is given as
\begin{equation}\label{dim hk}
\dim H_{k}={}_{24+k-1}C_{k}+2{\cdot}_{24+k-2}C_{k-1}+2{\cdot}_{24+k-3}C_{k-2}+{}_{24+k-4}C_{k-3}.
\end{equation}
Hence \eqref{function of k} is
\begin{align*}
N(k)^2=\frac{2^{44+4\varepsilon}}{\sqrt{2\pi}}\cdot& \left(
\frac{\Gamma(24+k)}{\Gamma(k+1)}+2\frac{\Gamma(23+k)}{\Gamma(k)}+2\frac{\Gamma(22+k)}{\Gamma(k-1)}+\frac{\Gamma(21+k)}{\Gamma(k-2)}
\right)\times\\
&\times\frac{\prod_{j=0}^3\Gamma(k+11+\varepsilon/2+j/4)}{\Gamma(k+11)^2\cdot\Gamma(k+11+1/2)^{2}}.
\end{align*} 
Hence finally by Lemma \ref{formula 1} have 
\begin{theorem}
{\em (1)} Let  $\varepsilon= -\frac{47}{4}$, then the Bargmann type transformation  
\[\mathfrak{B}:\mathfrak{F}_{-47/4}\longrightarrow L_{2}(P^{2}\mathbb{O},dv_{P^2\mathbb{O}})\]
is an isomorphism, although it is not unitary.

{\em (2)} If $-22<\varepsilon< -\frac{47}{4}$, then the inverse of the
Bargmann type transformation 
\[\mathfrak{B}^{-1}:L_{2}(P^{2},dv_{P^2\mathbb{O}}) \longrightarrow \mathfrak{F}_{\varepsilon}
\]
is bounded, but and the Bargmann type transformation can not be extended to the whole
Fock-like space $\mathfrak{F}_{\varepsilon}$.

{\em (3)} If $\varepsilon> -\frac{47}{4}$, then 
the Bargmann type transformation is bounded with the dense image, but not an isomorphism between
the spaces $\mathfrak{F}_{\varepsilon}$ and $L_{2}(P^{2},dv_{P^2\mathbb{O}})$.

{\em (4)} Let $\varepsilon\leq -22$. Then, 
for such a $k$ that
$4k+44+2\varepsilon\leq 0$, the integral \eqref{parameter family of inner products} 
does not converge, although the Bargmann type transformation is defined for
such polynomials. Hence by defining an inner product on the finite
dimensional space $\sum_{4k+44+2\varepsilon\leq 0}\mathcal{P}_{k}[\mathbb{X}_{\mathbb{O}}]$
in a suitable way, the Bargmann type transformation 
behave in the same way as the case of {\em (2)} {\em(}see Remark \ref{Fock-like space note}{\em)}.
\end{theorem}

\begin{remark}
The result in the above theorem differs from the
original Bargmann transformation and other cases of the spheres, complex
projective spaces and quaternion projective spaces for which
the Bargmann type transformations are always isomorphisms {\em (\cite{Ba},
\cite{Ra2}, \cite{Fu1}, \cite{FY})} without a modification factor in
the weight for defining an inner product in the Fock-like space.
\end{remark}

\section{Some additional results}

\subsection{Reproducing kernel of the Fock-like space $\mathfrak{F}_{\varepsilon}$}
As an application of the explicit determination of the constant $b_{k}$
we show our Fock-like space $\mathfrak{F}_{\varepsilon}$ has the reproducing kernel.

Since the operator $\mathcal{A}_{k}$ is an isomorphism from $H_{k}$ to
$\mathcal{P}_{k}[\mathbb{X}_{\mathbb{O}}]$ and the
operator $\mathcal{B}_{k}\circ \mathcal{A}_{k}\equiv b_{k}$, 
the composition $\mathcal{A}_{k}\circ \mathcal{B}_{k}\equiv
b_{k}$ too. The kernel function (we put it as $R_{k}(A,B)$,~
$(A,B)\in\mathbb{X}_{\mathbb{O}}\times \mathbb{X}_{\mathbb{O}}$) 
of the composition 
\[
\frac{\mathcal{A}_{k}\circ \mathcal{B}_{k}}{b_{k}},
\] 
which is the
identity operator on $\mathcal{P}_{k}[\mathbb{X}_{\mathbb{O}}]$, 
is expressed as
\begin{align*}
&R_{k}(A,B)\\
&=\frac{\int_{P^{2}\mathbb{O}}\,(\text{tr}\,X\circ A)^{k}(\text{tr}\,X\circ
\overline{B})^{k}dv_{P^{2}\mathbb{O}}\cdot
e^{-2\sqrt{2}\pi(||A||^{1/2}+||B||^{1/2})}(||A||\cdot||B||)^{14+\varepsilon}}{b_{k}}\\
&=\frac{\int_{P^{2}\mathbb{O}}\,(\text{tr}\,(X\circ A/||A||))^{k}
(\text{tr}\,(X\circ\overline{B}/||B||))^{k}\cdot dv_{P^{2}\mathbb{O}}
\cdot e^{-2\sqrt{2}\pi(||A||^{1/2}+||B||^{1/2})}(||A||\cdot||B||)^{k+14+\varepsilon}}{b_{k}}.
\end{align*}
Hence the sum
\begin{align*}
&R(A,B)):=\sum_{k=0}^{\infty}\,R_{k}(A,B)\\
&\sum_{k=0}^{\infty}
\frac{\int_{P^{2}\mathbb{O}}\,(\text{tr}\,X\circ A/||A||)^{k}(\text{tr}\,X\circ\overline{B}/||B||)^{k}dv_{P^{2}\mathbb{O}}\cdot e^{-2\sqrt{2}\pi(||A||^{1/2}+||B||^{1/2})}(||A||\cdot||B||)^{k+14+\varepsilon}}{b_{k}}.
\end{align*}
is estimated as
\begin{align*}
&\left|
\sum_{k=0}^{\infty}
\frac{\int_{P^{2}\mathbb{O}}\,(\text{tr}\,X\circ A/||A||)^{k}
(\text{tr}\,X\circ\overline{B}/||B||)^{k}\cdot dv_{P^{2}\mathbb{O}}
\cdot e^{-2\sqrt{2}\pi(||A||^{1/2}+||B||^{1/2})}(||A||\cdot||B||)^{k+14+\varepsilon}}{b_{k}}\right|\\
&\leq\frac{V\hspace{-0.12cm}{ol(P^2\mathbb{O})}2^{40+3\varepsilon\pi^{44+2\varepsilon}}}{V\hspace{-0.12cm}{ol(S(P^2\mathbb{O}))}}
\cdot e^{-2\sqrt{2}\pi(||A||^{1/2}+||B||^{1/2})}(||A||\cdot||B||)^{14+\varepsilon}\cdot
\sum\,\frac{2^{8k}\pi^{4k}\cdot(||A||||B||)^{k}}{\Gamma(4k+44+2\varepsilon)}\times\\
&\qquad\qquad\times\left(\frac{\Gamma(24+k)}{\Gamma(k+1)}
+2\frac{\Gamma(23+k)}{\Gamma(k)}+2\frac{\Gamma(22+k)}{\Gamma(k-1)}+\frac{\Gamma(21+k)}{\Gamma(k-2)}
\right)
\end{align*}

This inequality implies that the series 
converges locally uniformly on the space
$\mathbb{X}_{\mathbb{O}}\times\mathbb{X}_{\mathbb{O}}$ and 
the function $R(A,\overline{B})$ is holomorphic there.
So $R(A,{B})$ is the reproducing kernel of the Hilbert space
$\mathfrak{F}_{\varepsilon}$ ($\varepsilon > -22$).

\subsection{Geodesic flow and eigenspaces of Laplacian on $P^{2}\mathbb{O}$}
Let $\phi_{t}$ ($t\in\mathbb{R}$) be an action on $\mathbb{X}_{\mathbb{O}}$ defined
by
\[
\mathbb{X}_{\mathbb{O}}\ni A\longmapsto \phi_{t}(A)= e^{2\sqrt{-1}t}\cdot A.
\]
Then this is an interpretation of the geodesic flow action onto the space
$\mathbb{X}_{\mathbb{O}}$ through the map $\tau_{\mathbb{O}}$.

Let $p\in \mathcal{P}_{k}[\mathbb{X}_{\mathbb{O}}]$. Then
\begin{equation}\label{geodesic flow action on degree k polynomial}
{\phi_{t}}^{*}(p\cdot {\bf
  t}_{0}\otimes\Omega_{\mathbb{O}})(A)=e^{2\sqrt{-1}t(11+k)}\cdot
p(A)\cdot{\bf t}_{0}(A)
\otimes\Omega_{\mathbb{O}}(A).
\end{equation}

Let $p\in\mathcal{P}_{k}[\mathbb{X}_{\mathbb{O}}]$ and $q \in\mathcal{P}_{\ell}[\mathbb{X}_{\mathbb{O}}]$
with $k\not=\ell$, then
\begin{lemma}
\[
\boldsymbol{(}p,,\,q\boldsymbol{)}_{\varepsilon}
=\int_{\mathbb{X}_{\mathbb{O}}}\,p\cdot \overline{q}\cdot
{g_{0}}^2\cdot ||A||^{\varepsilon}\cdot\Omega_{\mathbb{O}}\wedge\overline{\Omega_{\mathbb{O}}}=0.
\]
\end{lemma}

\begin{proof}
The transformation ${\phi_t}^{*}$ on
$\Gamma_{\mathcal{G}}(\mathbb{L}\otimes K^{\mathcal{G}},\mathbb{X}_{\mathbb{O}})$
is unitary, hence
\[
\boldsymbol{(}{\phi_{t}}^{*}(p),{\phi_{t}}^{*}(q)\boldsymbol{)}_{\varepsilon}\equiv
\boldsymbol{(}p,q\boldsymbol{)}_{\varepsilon}~\text{for any $t\in\mathbb{R}$}.
\]
On the other hand
\[
{\phi_{t}}^{*}\big(p\cdot \overline{q}\cdot 
<{\bf t}_{0},{\bf t}_{0}>^{\mathbb{L}}\Omega_{\mathbb{O}}\wedge\overline{\Omega_{\mathbb{O}}}\big) 
= e^{2\sqrt{-1}(k-\ell)t}\cdot \big(p\cdot \overline{q}\cdot<{\bf
  t}_{0},{\bf t}_{0}>^{\mathbb{L}} \Omega_{\mathbb{O}}\wedge\overline{\Omega_{\mathbb{O}}}\big).
\]
Hence $\boldsymbol{(}p,q\boldsymbol{)}_{\varepsilon}=0$.
\end{proof}

Let $\Delta^{P^2\mathbb{O}}$ be the Laplacian on $P^2\mathbb{O}$. Then
\begin{proposition}
The geodesic flow action on $\mathbb{X}_{\mathbb{O}}$ and the action 
given by the one parameter group $\{e^{2\sqrt{-1}\,t\,\sqrt{\Delta^{P^2\mathbb{O}}+11}}\}$
of unitary transformations consisting of the Fourier integral operators
commute through the Bargmann type transformation.
\end{proposition}
\begin{proof}
This is shown based on the data that the eigenvalues of the Laplacian
$\Delta^{P^2\mathbb{O}}$
is given by $k^2+11k$ and 
the Bargmann type transformation
on each subspace $\mathcal{P}_{k}[\mathbb{X}_{\mathbb{O}}]$ maps to
$H_{k}$ which coincides with the $k$-th eigenspace of the Laplacian
(Propositions \ref{Killing form metric}, \ref{irreducible rep}). 
\end{proof}
\begin{remark}Finally we mention that in a forthcoming paper
the reproducing kernel above will be made clear to relate with a
differential equation satisfied by some hypergeometric
functions and also a T\"oplits operator theory on $P^{2}\mathbb{O}$ will be
discussed.
 \end{remark}

\appendix

       \setcounter{secnumdepth}{1}
       \setcounter{theorem}{0}
       \setcounter{section}{0}
       \setcounter{equation}{0}
\section{Appendix:Generating functions of Poincar\'e series}

%Let ${\bf a}=(a_{0},a_{1},\ldots,\cdots)$ be a 
%sequense of complex numbers. We denote the {\it i}-th component of ${\bf a}$ by $\{ {\bf
%  a}\}_{i}$
%and define a product $~{\bf a}* {\bf b}~$ of two such sequences, we call it a convolution,
%by
%
%\[
%\{{\bf a}*{\bf b}\}_{n}=\sum\limits_{k=0}^{n}\,a_{n-k}b_{k}.
%\]
%This product naturally appears in the Definition \ref{Cayley harmonic polynomial} 
%of the {\it Cayley harmonic polynomials}. %
%
%If we identify the space of sequences with the usual sum and the
%convolution product and the algebra of formal power
%series $\mathbb{C}[[~t~]]$ in an natural way, they are isomorphic as algebras. 
%\smallskip

In this Appendix we consider the generating functions of the
Poincar\'e series of 
\begin{align*}
&(1) ~\text{\it the polynomial algebra} : PP(t)=\sum \,\dim
P_{k}\,t^{k},\\
&(2)~\text {\it the algebra of invariant polynomials} : PI(t)=\sum
  \dim\,I_{k}\,t^{k}~\text{and} \\
&(3)~\text{\it the space of the Cayley harmonic polynomials} : PH(t)=\sum {\dim
  H_{k}}\,t^{k}, 
\end{align*}
and prove the inequality : 
\begin{equation}\label{basic inequality}
\dim~H_{k+1}>\dim~H_{k}.
\end{equation}
In fact, these formal power series converge for $|t|<1$, which will be
seen by explicitly determining their generating functions.

The generating function $PI(t)$ of the
{\it Poincar\'e series} of the dimensions of invariant
polynomials $I=\sum I_{k}$
is determined as  
\begin{align*}
&PI(t)=\sum\limits_{k=0}^{\infty} \,\dim I_{k}\,t^{k}
=\sum\limits_{k=0}^{\infty}\,\sum\limits_{\ell=0}^{[k/3]}\,\left(\,\left[\frac{k-3\ell}{2}\right]+1\,\right)\,t^{k}
=\sum\limits_{k=0}^{\infty}\,\,\,\sum\limits_{i_{1}+2i_{2}+3i_{3}=k,~i_{1},i_{2},i_{3}\in\mathbb{N}_{0}}\, t^{k}\\
&=\sum\limits_{(i_{1},i_{2},i_{3})\in{\mathbb{N}_{0}\times\mathbb{N}_{0}\times\mathbb{N}_{0}}}t^{i_{1}+2i_{2}+3i_{3}}
=\frac{1}{1-t}\cdot\frac{1}{1-t^2}\cdot\frac{1}{1-t^3}.
\end{align*}
The generating function $PP(t)$ of the polynomial algebra $\mathbb{C}[s_{1},\cdots,s_{N}]=\sum~P_{k}$ is given by
\begin{align*}
&PP(t)=
\sum\,\dim P_{k}\,t^{k}
=\sum\limits_{k=0}^{\infty}\,_{N+k-1}C_{k}\,t^{k}\\
&=\sum\limits_{(r_{1},r_{2},\ldots,r_{N})\in{\mathbb{N}_{0}}^{N}}\,t^{r_{1}+r_{2}+\cdots+r_{N}}
=\left(\frac{1}{1-t}\right)^{N}, ~\text{in which $N=27$ for our case}.
\end{align*}
Let $PH(t)$  be the generating function of the Poincar\'e series of
the dimensions of Cayley harmonic polynomials,
then by Lemma \ref{product of dimensions} and Proposition \ref{direct sum}
\begin{equation}\label{relation of three Poincare polynomials}
PP(t)=PH(t)\cdot PI(t)
\end{equation}
and we have
\begin{align}
PH(t)&=\left(\frac{1}{1-t}\right)^{24}\cdot (1+t)(1+t+t^2)\notag\\
&=
\sum\limits_{k=0}^{\infty} {}_{24+k-1}C_{k}\,t^{k}\cdot (1+2t+2t^2+t^3)=
\sum\limits_{k=0}^{\infty} \,\dim H_{k}\,t^k.\label{dim Hk}
\end{align}
Then
\begin{proposition}\label{A1}
$\dim H_{k}<\dim H_{k+1}$.  
\end{proposition}

This can be proved by the following elementary fact:
\begin{lemma}
Let $f(t)=\sum\, a_{k}t^{k}$ and $g(t)=\sum\, b_{k}t^{k}$ be formal power
series
with positive coefficients and satisfies the condition that
\[
\text{for all}~n,~b_{n}\leq b_{n+1}.
\]
Then the coefficients of the product formal power series $f\cdot g$
is increasing.
\end{lemma}
\begin{proof}
Since the {\it n}-th coefficient $c_{n}$ of the product $fg$ is
\[
c_{n}=\sum\limits_{i=0}^{n}\, a_{n-i}b_{i}
\]
\[
c_{n+1}-c_{n}=a_{n+1}b_{0}+a_{n}(b_{1}-b_{0})+\cdots+ a_{0}(b_{n+1}-b_{n}).
\]
In the above expression, each term is non-negative by assumption
so that $c_{n+1}-c_{n}\geq 0$. In addition if $\{b_{n}\}$ is strictly
increasing, then $\{c_{n}\}$ is also strictly increasing at least one
of the coefficient being $a_{k}>0$.
\end{proof}
\smallskip

{\bf Proof of Proposition \ref{A1}.}
\smallskip

In our case, all the coefficients of the polynomial
$(1+t)(1+t+t^2)=1+2t+2t^2+t^3$ are positive 
and the coefficients of the power series expansion of the 
factor $\displaystyle{\left(\frac{1}{1-t}\right)^{24}}$ are positive
and strictly increasing. In fact, the {\it k}-th coefficient of the
power series expansion of the function
$\displaystyle{\left(\frac{1}{1-t}\right)^{24}}$ 
is $_{24+k-1}C_{k}$ and strictly increasing, since it is a generating
function of the Poincar\'e power series of the polynomial algebra $\mathbb{C}[s_{1},\cdots,s_{24}]$
of $24$ variables.
Hence
the assertion 
for our power series $PH(t)=\displaystyle{(1+2t+2t^2+t^3)\cdot \left(\frac{1}{1-t}\right)^{24}}$
is proved.
\hfill{$\Box$}

\end{document}